\documentclass{amsart}

\usepackage{verbatim}
 \usepackage{amsmath, amssymb, amsfonts, euscript, enumerate,latexsym}
\usepackage{amsthm}

\usepackage{pxfonts}
\usepackage{fancyhdr}
 \usepackage{mathrsfs}
\usepackage{mathtools}
\usepackage{epsfig,subfigure}
\usepackage{graphicx}
\usepackage{color,wrapfig}
\usepackage{txfonts}
 
 \usepackage{hyperref}

\newcommand{\R}{{\mathbb R}}
\newcommand{\N}{{\mathbb N}}

\newcommand{\cG}{{\mathcal G}}

\newcommand{\cA}{{\mathcal A}}

\newcommand{\cH}{{\mathcal H}}
\newcommand{\cD}{{\mathcal D}}

\newcommand{\cE}{{\mathcal E}}

\newcommand{\cM}{{\mathcal M}}

\newcommand{\cS}{{\mathcal S}}
\newcommand{\cJ}{{\mathcal J}}

\newcommand{\sH}{{\mathscr H}}
\newcommand{\sS}{{\mathscr S}}

\newcommand{\e}{\epsilon}
\newcommand{\ve}{\varepsilon}

\newcommand{\ld}{\lambda}
\newcommand{\Ld}{\Lambda}
\newcommand{\p}{\partial}
\newcommand{\vp}{\varphi}

\newcommand{\supp}{\operatorname{supp}}
\newcommand{\diag}{\operatorname{diag}}
\newcommand{\Cut}{\operatorname{Cut}}

\newcommand{\Jac}{\operatorname{Jac}}
\newcommand{\tr}{\operatorname{tr}}
\newcommand{\Ric}{\operatorname{Ric}}
\newcommand{\Sec}{\operatorname{Sec}}
\newcommand{\Sym}{\operatorname{Sym}}

\newcommand{\D}{\nabla}

\newcommand{\La}{\Delta}

\newcommand{\Div}{\operatorname{div}}
\newcommand{\vol}{\operatorname{Vol}}
\newcommand{\diam}{\operatorname{diam}}

\newtheorem{thm}{Theorem}[section]
\newtheorem{lemma}[thm]{Lemma}
\newtheorem{cor}[thm]{Corollary}
\newtheorem{remark}[thm]{Remark}
\newtheorem{prop}[thm]{Proposition}
\newtheorem{definition}[thm]{Definition}
\newtheorem{example}[thm]{Example}

\theoremstyle{definition}

\begin{document}
\title[Harnack inequality for degenerate and  singular       operators on manifolds]
{Harnack inequality for degenerate and  singular       operators of   $p$-Laplacian type  on Riemannian manifolds}
 
\author[Soojung Kim]{Soojung Kim}
\address{Soojung Kim :
Institute of Mathematics, Academia Sinica\\ 
6F, Astronomy-Mathematics Building No.1, Sec.4, Roosevelt Road\\
Taipei 10617, TAIWAN
}
\email{soojung26@gmail.com; soojung26@math.sinica.edu.tw}


\maketitle
\begin{abstract}
We study   viscosity solutions to degenerate and singular elliptic equations of $p$-Laplacian type   on Riemannian manifolds.   
The Krylov-Safonov type Harnack inequality for the $p$-Laplacian operators with $1<p<\infty$  is established on the  manifolds  with Ricci curvature bounded from below based on    ABP type estimates.  We also   prove   the Harnack inequality for nonlinear $p$-Laplacian type operators assuming that   a nonlinear perturbation of     Ricci curvature is  bounded   below.

 \end{abstract} 
 \tableofcontents 


\section{Introduction}
In this paper, we prove  the Krylov-Safonov    Harnack inequality  for viscosity  solutions to    degenerate and singular elliptic  equations of  $p$-Laplacian type with   $1<p<\infty$ on Riemannian manifolds.   The $p$-Laplacian operator    
$\La_p u=\Div\left(|\D u|^{p-2}\D u\right)$    
appears  in  the Euler-Lagrange equation of the $L
^p$-norm of the gradient of 
functions, and    
      can also  be expressed     in  nondivergence form:
$$\La_p u=|\D u|^{p-2}\tr\left[ \left( {\bf{I}}+(p-2)\frac{\D u}{|\D u|}\otimes \frac{\D u}{|\D u|}\right)D^2u\right],$$
where  a  tensor product $X \otimes X$  for a vector field $X$ over a Riemannian manifold $M$  is a symmetric bilinear form on $TM$ defined by    $X \otimes X\, (Y,Z)  := \left\langle X,Y\right\rangle \left\langle  X, Z\right\rangle$ for $Y,Z\in TM$.  
Along the lines of  a fundamental  work of  Yau \cite{Y},  differential Harnack  inequalities  for  $p$-Laplacian operators  have been obtained  in \cite{KN,WxZ} on Riemannian   manifolds with   Ricci curvature bounded below.     
For   divergence form operators,    the De Giorgi-Nash-Moser Harnack inequality was  extended for uniformly parabolic operators in \cite{G, SC},  and for  the $p$-Laplacian operators  in   \cite{H,RSV}  on Riemannian   manifolds
 satisfying certain properties: a volume doubling property and a  weak version of Poincar\'e's  inequality. 
 
  Cabr\'e  \cite{Ca}   in his remarkable paper investigated   the Krylov-Safonov type Harnack inequality  for uniformly elliptic operators  on Riemannian manifolds   by
      establishing  the  ABP type estimates.   
 The ABP  estimate is a cornerstone of the Krylov-Safonov regularity theory, which is proved  using affine functions in the Euclidean space; refer to   \cite{CC} for instance. 
  In the Riemannian setting,   Cabr\'e  used   the squared distance function as the appropriate replacement for affine functions to derive the ABP type estimates  
  on  Riemannian manifolds with nonnegative sectional curvature.    
  The idea of sliding       paraboloids  
 was also used by Savin \cite{S}  to show the ABP type measure  estimate for small perturbation solutions in $\R^n$; see also \cite{M}.  
   The sectional curvature assumption of Cabr\'e's  result   was weakened   into    the certain conditions on   the Hessian of the distance function   in    \cite{K}, which  gave in particular  a new proof of  the  Harnack inequality on Riemannian manifolds with nonnegative Ricci curvature.   Recently, Wang and Zhang \cite{WZ}  studied     the ABP type estimates 
   and a locally uniform Harnack inequality  for   uniformly elliptic operators on the  manifolds with  sectional curvature bounded below.   The Harnack inequality for  viscosity solutions to uniformly parabolic equations  on   Riemannian  manifolds  has been  obtained   in \cite{KKL,KL}.

This paper deals  with the 
 Harnack inequality for viscosity solutions to the $p$-Laplacian type equations with $1<p<\infty$  on Riemannnian manifolds  employing      the ABP type   method.  To study the ABP type estimates, we slide 
     the $\frac{p}{p-1}$-th power of the distance function   from below,    which is   the squared distance  in the special case    $p=2,$ and 
  introduce the $p$-contact set adapted to the $p$-Laplacian operator (see Definition \ref{def-p-contact-set}). Let $d_y$ denote the Riemannian distance  from a point $y\in M$.  For  $u\in C(\overline\Omega) $ and  a compact set $E\subset M,$ 
  the $p$-contact set associated with $u$  and  the  vertex set $E$   is defined by
  $$\cA^p\left(E;\overline \Omega;u\right):=\left\{x\in\overline\Omega : \,\exists y\in E \,\mbox{such that}\,\inf_{ \overline \Omega}\left\{u+\frac{p-1}{p}d_y^{\frac{p}{p-1}}\right\}=u(x)+ \frac{p-1}{p}d_y^{\frac{p}{p-1}}(x) \right\}.$$ 
Then  it is shown   in Proposition \ref{prop-jacobi-estimate} that   a  vertex point  $y\in E$ associated with a contact point $x\in \Omega$ with  nonvanishing gradient of $u$  is  given by 
 $$y=\exp_x |\D u|^{p-2}\D u(x)=:\Phi_p(x).$$
By estimating  an upper bound of     the Jacobian determinant of the map $\Phi_p(x)$
 over the  $p$-contact set, 
we prove    the ABP type estimate for the degenerate cases $p\geq 2$   in Theorem   \ref{thm-abp-type-p-large}  stating that  the measure of the vertex set is bounded by the integral  over the   $p$-contact set    in terms of   the $p$-Laplacian   operator,   provided that   Ricci curvature of the underlying manifold  is  bounded   below. This generalizes     \cite[Theorem 1.2]{WZ}, the     case $p=2$.     
  For $1<p<2,$   the $p$-Laplacian operator becomes singular when the gradient   vanishes. To cope with  singularities, we 
  make use of    a regularized  operator 
  $\left(|\D u|^2+\delta\right)^{\frac{p-2}{2}}\cM^-_{p-1,1}(D^2u)$ for $  \delta>0,$
   which    has been considered  in \cite{DFQ1, ACP} for the Euclidean case. This  leads to introduce  a regularized   map 
 $$x\mapsto \exp_x\left(|\D u|^2+\delta\right)^{\frac{p-2}{2}} \D u(x)=:\Phi_{p,\delta} (x)\quad(\delta>0),$$ which lies on the minimal geodesic  joining $x$ to $y=\exp_x |\D u|^{p-2}\D u(x)=\Phi_p(x)$ at time $\left(\frac{|\D u(x)|^2}{|\D u(x)|^2+\delta}\right)^{\frac{2-p}{2}}$ for the contact  point $x$ with nonvanishing gradient of $u$.   
    A uniform  Jacobian estimate of $\Phi_{p,\delta}$ over the $p$-contact set    with respect to $\delta>0$ will imply the ABP type estimate for singular cases in  Theorem  \ref{thm-abp-type-p-1}.   
Based on the ABP type estimates,  a locally uniform  Harnack inequality  for the $p$-Laplacian operators is  established      by means of   the volume comparison   and the Laplacian comparison     on   Riemannian  manifolds with Ricci curvature bounded   below. 
   More generally, we  are concerned with 
 the nonlinear degenerate and singular  equations of $p$-Laplacian type
$$|\D u|^{p-2}F\left(D^2u\right)+\left\langle b,\D u\right\rangle |\D u|^{p-2} =f ,$$
where $F$ is a   uniformly elliptic operator and $b$ is a bounded vector field  over  $M$. 
  From \cite{K},   we recall  the Pucci   operator of the   Ricci transform;     for    $0<\lambda\leq \Lambda<\infty$,   
 and for any $x\in M$ and  any unit vector $e\in T_xM,$  
\begin{equation*}
\cM^-_{\lambda,\Lambda}(R(e)):=\Lambda\sum_{\kappa_i <0} \kappa_i +\lambda\sum_{\kappa_i>0} \kappa_i,
\end{equation*}
where $\kappa_i $ are the eigenvalues of the Ricci transform $R(e)$. Note that   $\cM^-_{1,1}(R(e))=\Ric(e,e).$  We refer to  Section \ref{sec-pre} for the definitions. 
 In place  of the Ricci curvature bound,   we assume    that $\cM^-_{\lambda,\Lambda}(R(e))$ is uniformly bounded below for any unit vector $e\in TM$ when we study    qualitative properties for viscosity solutions to  
 the nonlinear   degenerate and singular  equations of   $p$-Laplacian type.

%
We would like to mention  related  results on estimates for        the $p$-Laplacian type  operators. In the Euclidean space, the ABP estimates   for  the   $p$-Laplacian type  operators      have been    obtained    in  \cite{DFQ1, I, ACP}  on the basis of the   estimates over a  contact set by the affine functions. 
A recent work by    Imbert and Silvestre \cite{IS} addresses    the  H\"older  estimates  and Harnack inequality  for viscosity solutions  satisfying a uniformly elliptic equation only at points where the gradient is large.   In the proof  the ABP type measure estimates,  they    used   genuinely  a cusp,  the square root of  the distance  which corresponds to the case  $p=-1$ in our setting. Assuming  the Pucci operator of Ricci transform to be bounded below, 
their method    can be applied  to      Riemaninian manifolds     
with the use of  the argument in 
 Proposition  \ref{prop-jacobi-estimate}.  

Our    approach in this paper    yields   a    different proof 
   involving      more  intrinsic  geometric   quantities    to the geometry of the   $p$-Laplacian operators  on   Riemannian manifolds with Ricci curvature bounded below.  It      can also  be       adapted   to show  a parabolic analogue of  the  ABP  type estimate      by  using  a   generalized Legendre transform with respect to the $\frac{p}{p-1}$-th power of the  distance    in the singular cases  $1<p\leq2$  as in \cite{KKL}  since   $ d_y^{\frac{p}{p-1}}$  is twice differentiable on $M\setminus\Cut(y)$.

 Now  we state our main results  as    follows.  Throughout  this paper, let  $(M, g)$ be a smooth, complete Riemannian manifold of dimension $n$, and $B_R(z_0)$ denote a geodesic ball of radius $R$ centered at $z_0.$
 
  \begin{thm}[Harnack inequality]
 
   Let  $1<p<\infty $,  and $\Ric \geq - {(n-1)} \kappa$   for $\kappa\geq0$. 
For  $z_0\in M$ and $0< R\leq R_0,$   let $u$ be a nonnegative  viscosity solution  to 
\begin{equation*}
  \La_pu=  f  \quad\mbox{ in $B_{2R}(z_0)$. } 
\end{equation*} 
Then   
\begin{equation*} 
\sup_{B_R(z_0)} u\leq C\left(\inf_{B_R(z_0)}u+R^{\frac{p}{p-1}} \|f\|^{\frac{1}{p-1}}_{L^\infty(B_{2R}(z_0))}\right),
  \end{equation*}
where a    constant  $C>0$ depends only on $n, p$, and $\sqrt{\kappa}R_0.$
  \end{thm}
  In particular, when   Ricci curvature of the underlying manifold is nonnegative,  the above Harnack inequality   is a  global  estimate, which implies    the   Liouville theorem. 

   \begin{cor}
   Let  $1<p<\infty $,  and $\Ric \geq 0$. 
If  $u$ is a    viscosity solution  to 
$  \La_pu=  0$  in $ M$, which is bounded from below, then $u$ is a constant function. 
  \end{cor}
  
 As a consequence of the Harnack inequality, we have  a locally uniform   H\"older estimate for viscosity solutions to the $p$-Laplacian equations. 
  
    \begin{cor}[H\"older estimate]
  Let  $1<p<\infty $,  and $\Ric \geq - {(n-1)} \kappa$   for $\kappa\geq0$. 
For  $z_0\in M$ and $0< R\leq R_0,$     let $u$ be a    viscosity solution  to  $\La_pu=  f $ in $B_{2R}(z_0)$. 
 Then
\begin{equation*} 
R^\alpha[u]_{C^\alpha(B_R(z_0))} \leq C\left(\|u\|_{L^\infty(B_{2R}(z_0))}+R^{\frac{p}{p-1}} \|f\|^{\frac{1}{p-1}}_{L^\infty(B_{2R}(z_0))}\right),
  \end{equation*}
where   the   constants $\alpha\in(0,1)$  and  $C>0$ depend only on $n, p$, and $\sqrt{\kappa}R_0.$
  \end{cor}

 We also obtain similar results   for nonlinear $p$-Laplacian  type operators. 
  
  \begin{thm}[Harnack inequality] 
   Let   $1<p<\infty $, and $\cM^-_{\ld,\Ld}(R(e)) \geq - {(n-1)} \kappa$  with  $\kappa\geq0$ for any unit vector $e\in TM$.  Let  $z_0\in M$ and $0< R\leq R_0.$ For $\beta\geq0 $ and $ C_0\geq0,$   let $u$ be a nonnegative  viscosity solution  to
\begin{equation}\label{eq-nonlinear-p-laplace-eq-intro}
\left\{
\begin{split}
&|\D u|^{p-2}\cM^-_{\ld,\Ld}(D^2u)-\beta |\D u|^{p-1}\leq  C_0\quad\mbox{ in $ B_{2R}(z_0)$},\\
 &|\D u|^{p-2}\cM^+_{\ld,\Ld}(D^2u)+\beta |\D u|^{p-1}\geq -  C_0 \quad\mbox{ in $ B_{2R}(z_0)$.   }
\end{split}\right.
\end{equation}
Then   
\begin{equation*} 
\sup_{B_R(z_0)} u\leq C\left(\inf_{B_R(z_0)}u+R^{\frac{p}{p-1}}C_0^{\frac{1}{p-1}} \right),
  \end{equation*}
where a    constant  $C>0$ depends only on $n,p, \sqrt{\kappa}R_0 , \ld, \Ld, $ and $ \beta R_0.$
  \end{thm}

    \begin{cor}[H\"older estimate]
    Let   $1<p<\infty $, and $\cM^-_{\ld,\Ld}(R(e)) \geq - {(n-1)} \kappa$  with  $\kappa\geq0$ for any unit vector $e\in TM$.  Let  $z_0\in M$ and $0< R\leq R_0.$ For $\beta\geq0 $ and $ C_0\geq0,$ let  $u$ be a    viscosity solution  to \eqref{eq-nonlinear-p-laplace-eq-intro}. Then 
 \begin{equation*} 
R^\alpha[u]_{C^\alpha(B_R(z_0))} \leq C\left(\|u\|_{L^\infty(B_{2R}(z_0))}+R^{\frac{p}{p-1}} C_0^{\frac{1}{p-1}}\right),
  \end{equation*}
  where the    constants $\alpha\in(0,1) $ and   $C>0$ depend only on $n,p, \sqrt{\kappa}R_0, \ld, \Ld,$ and $ \beta R_0$.
  \end{cor}

The rest of the paper is organized as follows. In Section \ref{sec-pre}, we collect some of the known results on Riemannian geometry that are used in this paper. In Section \ref{sec-viscosity}, we   study  the properties  of the adapted viscosity solutions to     singular elliptic    equations,  and recall a  regularization by Jensen's  inf-convolution on  Riemannian manifolds. Section \ref{sec-abp} is devoted to the   proof of the ABP type estimates for  degenerate and singular operators of the $p$-Laplacian type with  $1<p<\infty.$  Section \ref{sec-L-e-est} contains the proof of the $L^\epsilon$-estimates by means of a barrier function   and  the volume  comparison. 
  In Section \ref{sec-Harnack}, we prove the  Harnack inequality. 
 


\section{Preliminary: Riemannian geometry}\label{sec-pre}

Let  $(M, g)$ be a smooth, complete Riemannian manifold of dimension $n$, equipped with  the  Riemannian metric $g$. A Riemannian metric defines a scalar product and a norm on each tangent space, i.e., $\langle X,Y\rangle_x:=g_x(X,Y)$  and $|X|_x^2:=\langle X,X\rangle_x$ for $X,Y\in T_xM$,  where $T_x M$ is the tangent space at $x\in M$.  Let $d(\cdot,\cdot)$ be the Riemannian distance   on $M$.  For a given point $y\in M$,   $d_y(x)$  stands for   the distance   to $x$ from $y$,  namely, $d_y(x):=d(x,y)$.   
A Riemannian manifold is equipped with     the Riemannian measure $\vol=\vol_g$ on $M$ which  is     denoted  by   $|\cdot|$  for simplicity. 

The exponential map $\exp: TM \to  M$ is defined as 
$$\exp_x X:=\gamma_{x,X}(1),$$ 
where   $\gamma_{x,X}:\R\to M$ is the unique geodesic   starting at $x\in M$ with velocity $X\in T_xM.$ We note   that the geodesic $\gamma_{x,X}$ is defined for all time since $M$ is complete, but it is not minimizing in general. This leads to define the cut time $t_c(X)$:  for $X\in T_xM$ with $|X|=1$, 
 $$t_c(X) := \sup\left\{ t >0 :  \exp_x sX \,\,\,\mbox{is minimizing between}\,\,\, x \,\,\,\mbox{and $\exp_x t X $}\right\}.$$
The cut locus of $x\in M,$ denoted by $\Cut(x),$ is defined  by
 $$ \Cut(x):= \left\{\exp_x t_c(X) X  : X\in T_x M \,\,\,\mbox{with $|X|=1$}\right\}. $$
Let  
$\cE_x := \left\{t X \in T_xM :  0\leq t<t_c(X),\,\, X\in T_x M \,\,\,\mbox{with $|X|=1$}\right\},$  and $\cE:=\{X\in TM :X\in \cE_x,\,\,  x\in M\}.$  
 In fact,   the exponential map $\exp\left|_{\cE}: \cE\to M\right.$ is smooth.  One can prove   that for any $x\in M, $ 
 $\Cut(x)= \exp_x(\p \cE_x),$  $M=\exp_x(\cE_x)\cup \Cut(x),$ and $\exp_x : \cE_x \to \exp_x(\cE_x)$ is a diffeomorphism.  We recall that $\Cut(x)$ is closed and of measure zero.  Given two points $x $ and $y\notin \Cut(x)$,  there exists a unique minimizing geodesic $\exp_x tX$ (for   $X\in \cE_x$)  joining  $x$ to $y= \exp_xX,$ and we will denote $X= \exp_x^{-1}(y)$ in the case.   The Gauss lemma implies that  $$  \exp_x^{-1}(y)=-\D d^2_y(x)/2,\quad \forall  y\not\in\Cut(x).$$

For a $C^2$-function  $u:M\to\R,$  the gradient $\D u$ of $u$ is defined by 
$\langle\D u, X \rangle := du(X) $
for any vector field $X$ on $M,$ where $du:TM\to\R$ is the differential of $u.$
The Hessian $D^2u$ of $u$  is defined as 
$$D^2u  \,  (X, Y):=  \left\langle\nabla_X \nabla u, Y\right\rangle,$$
 for any vector fields $X, Y$ on $M,$ where $\D$ denotes the Riemannian connection of $M.$   We observe that the Hessian $D^2u$ is a  symmetric 2-tensor over $M,$   and $D^2u(X,Y)$  at $x\in M$ depends only on  the values $X, Y$ at $x,$ and  $u$   in a small neighborhood of $x.$ 
 By the metric,  the Hessian of $u$ at $x$  is canonically identified with  
 a  symmetric endomorphism of $T_xM$: 
 $$
D^2u(x)  \cdot X=  \nabla_X \nabla u,\quad\forall X, Y\in T_xM. 
$$
We will   write $D^2u(x)  \,  (X, Y)=\left\langle D^2u (x)\cdot X,\, Y\right\rangle$ for $X\in T_xM.$ 
If $\xi$ is a    $C^1$-vector field  on $M,$ the divergence of $\xi$  is defined as
$\Div\xi:=\tr\left\{X\mapsto \D_X\xi\right\}.$ 
For  $u\in C^2(M),$ the Laplacian operator   $\La u=\tr  (D^2u ) $    coincides with $\Div( \D u)$. 

Denote by $R$   the Riemannain curvature tensor   defined as
$$ R(X, Y)Z = \D_X \D_YZ- \D_Y\D_XZ  -\D_{ [X,Y ]}Z$$
for any vector fields $X, Y,Z$ on $M.$ 
For 
 two linearly independent  vectors $X,Y\in T_xM,$       the sectional curvature of the
plane generated by $X$ and $Y$  is defined as 
 $$\Sec(X,Y):=\frac{\langle R(X, Y)Y,X\rangle}{|X|^2|Y|^2-\langle X,Y\rangle^2}.$$
For a unit vector $e \in T_xM$, we denote by $R(e)$   the Ricci transform of $T_xM$ into itself defined by 
$$R(e)X := R(X,e)e\quad\forall X\in T_xM.$$
Note that Ricci transform is symmetric. 
  The Ricci curvature is the trace of the Ricci transform, which can be expressed as follows: 
 for a unit vector $e\in T_xM$ and    an orthonormal basis $\{e,e_2,\cdots,e_n\}$ of $T_xM,$ $$\Ric(e,e)=\sum_{j=2}^n \Sec(e,e_j).$$  
As usual,    $\Ric\geq \kappa$ on $M\,\,\,(\kappa\in\R)$ stands  for   $\Ric_x \geq \kappa g_x $ for any $x\in M.$   
When   dealing with
a class of nonlinear elliptic operators,  we involve   
 Pucci's   operator of Ricci transform instead of the trace operator;   for    $0<\lambda\leq \Lambda$,   
 and for any $x\in M$ and  any unit vector $e\in T_xM,$   define 
\begin{align*}
\cM^-_{\lambda,\Lambda}(R(e)):=\Lambda\sum_{\kappa_i <0} \kappa_i +\lambda\sum_{\kappa_i>0} \kappa_i,
\end{align*}
where $\kappa_i$ are the eigenvalues of $R(e)$.  
 In the special case when $\lambda=\Lambda=1,$    Pucci's   operator simply coincides with the trace operator,  and hence  $\cM^{-}_{1,1}(R(e))=\Ric(e,e)$.  Notice  that $\ld \Ric(e,e) \geq \cM^-_{\ld,\Ld}(R(e))$, so a lower bound of $\cM^-_{\ld,\Ld}(R(e))$ guarantees  one  for Ricci curvature.

Now we  recall  the volume comparison theorem 
 assuming the Ricci curvature to be bounded from below. 
 In terms of polar normal coordinates
at $x\in M$, the area element   of a 
geodesic sphere $\p B_r(x)$ of radius $r$ centered at $x$ is  written  by $r^{n-1}A(r,\theta)d\theta,$  where $A(r,\theta)$ is the Jacobian determinant  of the map $\exp_x$ at $r\theta \in T_xM$.  With the use of  a standard theory of Jacobi fields, the Jacobian determinant of the exponential  map     has  an upper bound depending on  a lower bound of  Ricci curvature; the proof can be found in   \cite[Chapter 11]{BC}.    
   Bishop-Gromov's volume comparison theorem  relies on  the Jacobian estimates, which  
   says that the volume of balls does not increase faster than the volume of balls in the model space (see also \cite{V}).  In particular,   the volume     comparison   implies the     (locally uniform) volume doubling property. 
   \begin{thm}[Bishop-Gromov]\label{thm-BG}
  Assume  $\Ric 
  \geq -(n-1)\kappa $ for $ \kappa\geq0$.
  \begin{enumerate}[(i)]
  \item   Let $x \in M$, $y\notin\Cut(x)\cup\{x\}$,  and   $\gamma(t):=\exp_x t\xi$   be the    minimizing geodesic joining $x = \gamma(0)$ to $y =\gamma(1)$ for   $\xi\in \cE_x$.  
For $t\in[0,1],$ let  ${ \bf  J}(t)$ be the differential  of $\exp_x$ at $t \xi\in T_xM,$   namely,  ${  \bf J}(t):=d\exp_x (t \xi).$ Then 
$$t\mapsto  \det { \bf  J}(t) \cdot\sS\left(\sqrt{\kappa}t| \xi |\right)^{-(n-1)}$$
is a nonincreasing function of $t\in(0,1],$ where $\sS(\tau):=\sinh(\tau)/\tau$ for $\tau>0$ with $\sS(0)=1.$ Furthermore,   for $t\in[0,1]$
$$0<\det { \bf  J}(t)\leq \sS\left({\sqrt{\kappa}t| \xi |}\right)^{n-1}=\left(\frac{\sinh\left(\sqrt{\kappa}\,t| \xi |\right)}{\sqrt{\kappa}\,t| \xi |}\right)^{n-1}.$$
\item   Let $$V(r):=\omega_{n,\kappa}\int_0^r \sinh^{n-1}\left(\sqrt{\kappa}\,t\right)dt\quad\forall r>0,$$
  where $V(r)$  denotes the volume of a ball of radius $r$  in the $n$-dimensional  space form   of constant   curvature   $-\kappa.$  
  Then  
  $$r\mapsto \frac{\vol\left(B_r(x)\right)}{V(r)}$$ is a  nonincreasing function  of  $r>0$. 
  In particular,     for any $0<r < R$ 
  \begin{equation}\label{eq-doubling}
  \frac{\vol(B_{2r}(x))}{\vol(B_r(x))}\leq 2^n\cosh^{n-1}\left(2\sqrt{\kappa}R\right)=:\cD,
  \end{equation}
    where    $\cD $  is a so-called doubling constant.  
    
\end{enumerate}
    \end{thm}
    
      One can see  that the doubling property \eqref{eq-doubling}  yields that for any $0<r< R\leq R_0,$  
  \begin{equation}\label{eq-doubling-cont}
  \frac{\vol\left(B_{R}(x)\right)}{\vol\left(B_r(x)\right)}\leq  \cD\, \left(\frac{R}{r}\right)^{\log_2\cD},
  \end{equation}
  where  $\cD:= 2^n\cosh^{n-1}\left(2\sqrt{ \kappa}R_0\right).$ 
      According to the volume   comparison, it is not difficult  to prove the following lemma by  a similar argument to  the proof  of \cite[Lemma 2.1]{IS}     (see also \cite{CC}).

\begin{lemma}
\label{lem-doubling-property}
Let $\Ric\geq -(n-1)\kappa$ for $\kappa\geq0 $ and $0<R\leq R_0.$ 
Let $E \subset F $ be two open subsets in $  B_R(x)\subset M$. Assume that for some $\delta\in(0,1)$, 
\begin{enumerate}[(a)]
\item if any ball $B\subset B_R(x)$ satisfies $|E\cap B|>(1-\delta)|B|,$ then $B\subset F,$
\item $|E|\leq (1-\delta)|B_R(x)|.$
\end{enumerate}
Then  there exists a constant $c_0\in(0,1)$, depending only on $n,$ and $\sqrt{\kappa} R_0,$ such that 
$$|E| \leq (1 - c_0\delta)|F|.$$
\end{lemma}

  Using  a standard theory of Jacobi fields,   
    the Jacobian determinant of   the   exponential  map involving a  vector field  has the following   formula; we refer to    \cite[Lemma 3.2]{Ca} for the proof, and see also   \cite{CMS} and   \cite[Chapter 14]{V}. 
  \begin{lemma}\label{lem-jacobian-exp}
  Let $\xi$ be a  {smooth} vector field on $M.$ Define a map $\phi(z):= \exp_z \xi(z).$ For  a given $x\in M,$ assume that  $\xi(x)\in \cE_x$, and let $y:=\phi(x).$ Then 
  $$ d\phi(x) \cdot X= d \exp_x(\xi(x))  \cdot\left\{ \left(\D\xi   + D^2d^2_y/2\right)(x)\cdot X\right\} \quad\forall X \in T_xM ,$$ 
 and 
  $$\Jac\phi(x) =\left(\Jac \exp_x(\xi(x))\right)\cdot \left|\det\left(\D\xi+ D^2d^2_y/2\right)(x)\right|,$$ 
  where $\Jac\phi(x):=|\det d\phi(x)|$, $\Jac\exp_x(\xi(x))$  is the Jacobian determinant of the map $\exp_x$ at $\xi(x)\in T_xM$, and 
  $\D \xi$ denotes the  covariant derivative of $\xi$.
  \end{lemma}
  
  With the same notation as the lemma above,  
we         define a map $\phi(x,\,t):=\exp_xt\,\xi(x) $ for    $t\in[0,1]$, and     assume that $\D \xi(x)$ is  symmetric.    
  From the next lemma, we observe     that  if $ \left(\D\xi+ D^2d^2_{\phi(x,\,1)}/2\right)(x) $ is positive semi-definite, 
  then $ \left(t\D\xi+ D^2d^2_{\phi(x,\,t)}/2\right)(x)$ remains positive semi-definite for $t\in[0,1]$.   Making appropriate modifications,  this  fact  will play an important  role  when we deal with  the  approximation of    singular    operators in Section \ref{sec-abp}.  
     The  following lemma  can be found in   \cite[Lemma 2.3]{CMS}; see also  \cite[Chapter 14]{V},   in particular the third appendix.

  \begin{lemma}\label{lem-hess-dist-sqrd-geodesic}
  Let $x\in M$ and $y\not\in\Cut(x) .$ 
  Let $\gamma:[0,1]\to M$ be the minimizing geodesic joining $\gamma(0)=x$  to $\gamma(1)=y$.  Then  
the self-adjoint operator   
 $$ D^2 (d^2_{\gamma(t)}/2)(x) - t\, D^2 (d^2_{\gamma(1)}/2)(x) $$ 
 defined on $T_x M$  is positive semi-definite for $t\in[0,1]$.
  \end{lemma}

An upper Hessian bound for the squared distance function is proved  in \cite[Lemma 3.12]{CMS}  using the formula for the second variation of energy  provided  that the sectional curvature is bounded from below. We quote it as follows. 


 \begin{lemma}\label{lem-hess-dist-sqrd}
 Let $x, y\in M.$  If $\Sec
  \geq -\kappa\,\, ( \kappa\geq0)$ along a minimizing geodesic  joining $x$ to $y,$ then for any $X\in T_xM$ with $|X|=1,$
  $$\limsup_{t\to0}\frac{d_y^2\left(\exp_xtX\right)+d_y^2\left(\exp_x -tX\right)-2d^2_y(x)}{t^2}\leq 2\sqrt{\kappa}d_y(x)\coth\left(\sqrt{\kappa}d_y(x)\right).$$
  \end{lemma}

   The well-known Laplacian comparison theorem    implies   that  if the  Ricci curvature  is bounded from below by  $-(n-1)\kappa$ for $\kappa\geq0$,  then
    $$\La \left(d_y^2/2(x)\right) \leq 1+ (n-1)  \sqrt{ {\kappa}{ }}d_y(x)\coth\left(\sqrt{ {\kappa}{ }}d_y(x)\right)\quad\mbox{for $x\not\in \Cut(y) $}. $$
 As a    generalization, we are  concerned with an upper estimate  of  Pucci's  maximal operator for   the squared distance  function   when   $\cM^-_{\ld,\Ld}(R(e))$ has a lower bound. The proof  for the following lemma   uses 
 the formula for the second variation of the  energy, and  closely  follows  the argument in the proof    of \cite[Lemma 3.12]{CMS}; see also    \cite[Lemma 2.1]{K}. 

   \begin{lemma}\label{lem-pucci-ric-dist-sqrd}
   For  $x\in M$ and $y\not\in\Cut(x) ,$   let $\gamma$ be the  minimizing geodesic  joining $x$ to $y.$  
    Assume that $\cM^-_{\ld,\Ld}(R(e))\geq-(n-1)\kappa$  with $\kappa\geq0$  for any unit vector $e\in T_{\gamma(t)}M.$ 
    Then 
  \begin{equation*}
  \cM^+_{\ld,\Ld}\left(D^2d_y^2/2(x)\right)\leq \Lambda+ (n-1)\Ld\, \sH\left(\sqrt{\frac{\kappa}{\Ld}}d_y(x)\right)
  \end{equation*}
  where $\sH(\tau):=\tau\coth(\tau)$ for $\tau>0$ with $\sH(0)=1.$
  \end{lemma}  
  
  \begin{proof} 
 We may assume that  $x\not\in\Cut(y)\cup\{y\}$ since $D^2d^2_y/2(y)={\bf I}.$
 Let $\gamma:[0,l]\to M$ be the minimizing geodesic  parametrized by arc length such that $\gamma(0)=x$ and $\gamma(l)=y.$  We introduce an orthonormal basis $\{e_i\}_{i=1}^n$ on $T_xM$ such that $e_1=\dot \gamma(0)$ and $\{e_i\}_{i=1}^n$ are eigenvectors of $D^2d_y^2/2(x)$ on $T_xM.$ Note that a unit vector $e_1=-\D d_y(x)$ is   an eigenvector of $D^2d^2_y/2(x)$ associated with the eigenvalue $1.$   
By the parallel transport, we extend $\{e_i\}_{i=1}^n$ to $\{e_i(t)\}_{i=1}^n$  along $\gamma(t)$,  
 and define the vector fields 
  $$X_i(t)=\alpha(t)e_i(t),\quad\mbox{for}\,\,\, \alpha(t):=\frac{\sinh\left(\sqrt{\frac{\kappa}{\Ld}}(l-t)\right)}{\sinh\left(\sqrt{\frac{\kappa}{\Ld}}\,l\right)},$$ 
  which satisfy $X_i(0)=e_i(0),$ and $X_i(l)=0.$   For small $s\in(-\delta,\delta),$ consider 
  $$\gamma^i_{s}(t):=f_i(s,t)=\exp_{\gamma(t)}sX_i(t),$$
  which connects $\exp_x se_i$ to $y.$ 
  With the help of the  H\"older inequality, we have that  for $i=2,\cdots,n$ and small $s\in(-\delta,\delta)$, 
  \begin{align*}
 \frac{1}{2} d_y^2\left(\exp_xs e_i\right) &\leq  \frac{1}{2}\left(\int_0^l |\dot\gamma^i_{s}(t)|dt\right)^2\leq \frac{l}{2}\int_0^l |\dot\gamma^i_{s}(t)|^2dt=: { l}E_i(s),
    \end{align*}  
    where the equality holds for $s=0,$ i.e., 
    $d^2_y/2(x)= lE_i(0).$   This implies that 
    \begin{align*}
    \cM^+_{\ld,\Ld}(D^2d^2_y/2(x))&= \Ld+\sup_{\ld\leq a_i\leq \Ld} \sum_{i=2}^na_i\left\langle D^2d^2_y/2(x)\cdot e_i,e_i\right\rangle \leq \Ld+  \sup_{\ld\leq a_i\leq \Ld} \sum_{i=2}^na_i l\left.\frac{d^2}{d s^2} \right |_{s=0}E_i(s)
    \end{align*}
from the definition of of the Pucci operator. 
  Since     for  each $i=2,\cdots n,$ and $t\in[0,l],$ a curve  $s\mapsto f_i(s,t)$ is also a  geodesic,   
 it follows from     the formula for the second variation of the energy  that 
  \begin{align*}
  \cM^+_{\ld,\Ld}(D^2d^2_y/2(x))-\Ld 
&\leq  \sup_{\ld\leq a_i\leq \Ld} \sum_{i=2}^na_i l\left.\frac{d^2}{d s^2} \right |_{s=0}E_i(s) \\
  &= \sup_{\ld\leq a_i\leq \Ld} \sum_{i=2}^na_il\int_0^l\left\{| \dot X_i(t)|^2-\left\langle R(X_i(t), \dot\gamma(t))  \dot\gamma(t), X_i(t)\right\rangle\right\}dt \\
       &\leq  (n-1)\Ld l\int_0^l  \dot\alpha(t)^2dt- l \inf_{\ld\leq a_i\leq \Ld}  \int_0^l\alpha(t)^2\ \sum_{i=2}^n a_i\left\langle R(e_i(t), \dot\gamma(t)) \dot\gamma(t),e_i(t)\right\rangle dt \\
          &=  (n-1)\Ld l\int_0^l  \dot\alpha(t)^2dt- l \inf_{\ld\leq a_i\leq \Ld}  \int_0^l\alpha(t)^2 \sum_{i=2}^n a_i\left\langle R( \dot\gamma(t))\cdot e_i(t),e_i(t)\right\rangle dt.
          \end{align*} 
          Using the definition of Pucci's operator and the assumption, one can check that for any $\ld\leq a_i\leq \Ld,$  $\sum_{i=2}^n a_i\left\langle R( \dot\gamma(t))\cdot e_i(t),e_i(t)\right\rangle\geq \cM^-_{\ld,\Ld}(R(\dot\gamma(t)) \geq -(n-1)\kappa,$  and hence    
          \begin{align*}
  \cM^+_{\ld,\Ld}(D^2d^2_y/2(x))-\Ld        &\leq  (n-1)\Ld l\int_0^l  \dot\alpha(t)^2dt- l \inf_{\ld\leq a_i\leq \Ld}  \int_0^l\alpha(t)^2 \sum_{i=2}^n a_i\left\langle R( \dot\gamma(t))\cdot e_i(t),e_i(t)\right\rangle dt\\
              &\leq  (n-1)\Ld l \int_0^l\left\{\dot \alpha(t)^2+\frac{\kappa}{\Ld}\alpha(t)^2 \right\}dt\\
              &= (n-1)\Ld l\frac{\kappa}{\Ld}\sinh^{-2}\left(\sqrt{\frac{\kappa}{\Ld}}\,l\right)\int_0^l\cosh\left(2\sqrt{\frac{\kappa}{\Ld}}(t-l)\right)dt \\
        & = (n-1)\Ld  \sqrt{\frac{\kappa}{\Ld}}\,l\coth\left(\sqrt{\frac{\kappa}{\Ld}}\,l\right),
  \end{align*}
  which finishes the proof.
    \end{proof}

In  \cite[Proposition 2.5]{CMS},  it is shown that 
   the cut locus  of  $y\in M$ is characterized   as the set of points at which   the squared distance function $d_{y}^2$  fails to be semi-convex. 
   \begin{lemma}[Proposition 2.5 of \cite{CMS}]\label{lem-dist-sqrd-cut}
 Let $x, y\in M.$  If $x\in \Cut(y),$ then  there is  a unit vector $X\in T_xM$ such that 
 $$\liminf_{t\to0}\frac{d_y^2\left(\exp_xtX\right)+d_y^2\left(\exp_x -tX\right)-2d^2_y(x)}{t^2} =-\infty.$$
  \end{lemma}
  
     \begin{cor}\label{cor-dist-sqrd-cut-compoed-increasing-ft}
 Let $x, y\in M$  and let $\psi$  be a $C^2$-function in $(0,\infty)$ such that $\psi'>0$ on $(0,\infty)$.    If $x\in \Cut(y)$, then  there is  a unit vector $X\in T_xM$ such that 
 $$\liminf_{t\to0}\frac{\psi\left(d_y^2\left(\exp_xtX\right)\right)+\psi\left(d_y^2\left(\exp_x -tX\right)\right)-2\psi\left(d^2_y(x)\right)}{t^2} =-\infty.$$
  \end{cor}
  \begin{proof} 

  By using the Taylor  expansion, we have that for any unit vector $X\in T_xM$ and small $|t|\in(0,1),$
  \begin{align*}
 & \frac{\psi\left(d_y^2\left(\exp_xtX\right)\right)+\psi\left(d_y^2\left(\exp_x -tX\right)\right)-2\psi\left(d^2_y(x)\right)}{t^2}
  \\&=  \psi'\left( d_y^2(x)\right)\frac{d_y^2\left(\exp_xtX\right)+d_y^2\left(\exp_x -tX\right)-2d^2_y(x)}{t^2}
  \\
  &+\frac{\psi''\left(a(t)\right)}{2}\left(\frac{d_y^2\left(\exp_x tX\right)-d^2_y(x)}{t}\right)^2
+\frac{\psi''\left(b(t)\right)}{2}\left(\frac{d_y^2\left(\exp_x -tX\right)-d^2_y(x)}{t}\right)^2,
  \end{align*}
  where $a(t) $ and $ b(t)$ converge to $d_y^2(x)>0$ as $t$ tends to $0$. From the triangle inequality,  it follows      that for any unit vector 
  $X\in T_xM,$ 
  $$\limsup_{t\to 0}\left|\frac{d_y^2\left(\exp_x tX\right)-d^2_y(x)}{t}\right| \leq 
  2d(y,x).
$$
  Therefore the result follows from Lemma \ref{lem-dist-sqrd-cut} since $d_y^2(x)>0$ and $\psi'(d_y^2(x))>0.$ 
    \end{proof}

Lastly, we recall   the definition of semiconcavity  of functions on  Riemannian manifolds. 
%

  
  

\begin{definition}
Let $\Omega$ be an open set of $M.$  A function $u:\Omega\to\R$ is said to be semiconcave at $x_0\in\Omega$  if  there exist a geodesically convex  ball $B_r(x_0)\subset\Omega$  with  $0<r< i_M(x_0),$ and a smooth function 
  $ \Psi : B_r(x_0) \to \R$ such that  $u+\Psi$ is geodesically concave in  $B_r(x_0),$ where $i_{M}(x_0)$  denotes the injectivity radius    at $x_0$.   A function $u$ is semiconcave in $\Omega$ if it is semiconcave at each point of  $\Omega.$

For $C>0,$  we say that  a function $u:\Omega\to\R$ is  $C$-semiconcave   at    $x_0\in\Omega$ if   there exists a geodesically convex  ball $B_r(x_0)\subset\Omega$  with  $0<r<i_M(x_0)$ such that  
$ u-Cd^2_{x_0}(x)$ is geodesically concave in  $B_r(x_0).$ A function $u$ is  $C$-semiconcave in $\Omega$ if it is $C$-semiconcave at each point of $\Omega.$
\end{definition}

We have the     local characterization of semiconcavity    from \cite[Lemma 3.11]{CMS}.  According to Lemma \ref{lem-hess-dist-sqrd}, the following lemma implies that the squared distance function  is locally uniformly semiconcave.    
 
\begin{lemma}
 Let $u:\Omega\to\R$ be a continuous function and let  $x_0\in \Omega,$ where $\Omega\subset M$ is open.    
 Assume that there exist a neighborhood $U$ of $x_0,$ and a  constant $C>0$ such that for any $x \in U$ and $X \in T_x M$ with $|X|=1,$  
$$\limsup_{t\to 0}\frac{u\left(\exp_x tX \right)+u\left(\exp_x -tX \right)-2u(x)}{t^2}\leq C. $$
Then $u$ is $C$-semiconcave at $x_0.$
\end{lemma}

The following result by Bangert \cite{B} is an extension of Aleksandrov's second differentiability theorem in the Euclidean space that a convex function is twice differentiable almost everywhere   \cite{A}; see also \cite[Chapter 14]{V}.  
 
\begin{thm}[Aleksandrov--Bangert]\label{thm-AB}
Let $\Omega\subset M$ be an open set  and let $u:\Omega\to \R$   be semiconcave. Then  for almost every $x\in\Omega, $ $u$ is differentiable at  $x,$ and   there exists a symmetric operator $A(x):T_xM\to T_xM$ 
characterized by any one of the two equivalent properties:
\begin{enumerate}[(a)]
\item  For any   $X\in T_xM,\,\,\,$ $A(x)\cdot X=\D_X\D u(x),$ \\
\item
$u(\exp_x X)=u(x)+\left\langle\D u(x), X\right\rangle+\frac{1}{2}\left\langle A(x)\cdot X, X\right\rangle+o\left(|X|^2\right)\quad\mbox{ as $X\to0.$}
$\\
\end{enumerate}
The operator $A(x)$   and its associated  symmetric bilinear from  on $T_xM$  are  denoted by $D^2u(x)$ and  called the Hessian of $u$ at $x$ when   no confusion is possible. 
\end{thm}
 We refer to \cite{CMS, V, AF} for more  properties of  semiconcave functions.

\section{Viscosity solutions for  singular  operators      }\label{sec-viscosity}
 
In this section, we study viscosity solutions to   degenerate and singular  
operators  of    $p$-Laplacian type  for  $1<p<\infty$ on Riemannian manifolds.  As seen before, the $p$-Laplacian operator   
$\La_p u=\Div\left(|\D u|^{p-2}\D u\right)$
    can be written    in  nondivergence form:
$$\La_p u=|\D u|^{p-2}\tr\left[ \left( {\bf{I}}+(p-2)\frac{\D u}{|\D u|}\otimes \frac{\D u}{|\D u|}\right)D^2u\right],$$
where  a  tensor product $X \otimes X$  for a vector field $X$ over $M$  is a symmetric bilinear form on $TM$ defined by    
$$(X \otimes X) (Y,Z) := \left\langle X,Y\right\rangle \left\langle  X, Z\right\rangle\quad\mbox{for $Y,Z\in TM$}.$$  The   tensor product $X \otimes X$   is  considered as a symmetric endomorphism: $X \otimes X \cdot Y =\left\langle X , Y\right\rangle X$,  so we write $(X \otimes X) (Y,Z)=\left\langle X \otimes X\cdot Y, Z\right\rangle=\left\langle X,Y\right\rangle \left\langle  X, Z\right\rangle.$  
 For  $p\geq2,$ the operator   $G(\D u,D^2u)= \La_p u$ is     continuous    with respect to $\D u$ and $D^2 u,$ while   the $p$-Laplacian operator  for $1<p<2$ becomes singular at the points with    vanishing gradient.    
More generally, we are concerned with  the degenerate and singular fully nonlinear  equations given by
$$|\D u|^{p-2}F\left(D^2u\right)+\left\langle b,\D u\right\rangle |\D u|^{p-2} =f ,$$
where $F$ is a   uniformly elliptic operator and $b$ is a bounded vector field  over  $M.$   
Let  $\Sym TM$ be the bundle of symmetric 2-tensors over $M.$ 
An   operator  $F : \Sym TM\rightarrow \R$ 
 is  said  to be uniformly elliptic with the so-called ellipticity constants $0<\lambda\leq\Lambda$ if the following holds:  for  any $S\in \Sym TM,$ and for any positive semi-definite $P\in\Sym TM,$
 \begin{equation}\label{Hypo1}
   \lambda\,\mathrm{trace}(P_x)\leq F(S_x+P_x)-F(S_x)\leq \Lambda\, \mathrm{trace}(P_x),\quad\forall x\in M.
 \end{equation}
As extremal cases of the uniformly elliptic operators, we recall Pucci's   operators:    
 for any $x\in M,$ and  $S_x\in\Sym TM_x,$
\begin{align*}
\cM^+_{\lambda,\Lambda}(S_x):= \lambda\sum_{\mu_i<0}\mu_i+\Lambda\sum_{\mu_i>0}\mu_i,\quad 
\cM^-_{\lambda,\Lambda}(S_x):= \Lambda\sum_{\mu_i<0}\mu_i+\lambda\sum_{\mu_i>0}\mu_i,
\end{align*}
where $\mu_i=\mu_i(S_x)$ are the eigenvalues of $S_x.$    We observe  that   \eqref{Hypo1}   is equivalent to the following: 
 for any $S, P\in\Sym TM,$
    \begin{equation*}
   \cM^-_{\ld,\Ld} (P_x) \leq F(S_x+P_x)-F(S_x) \leq\cM^+_{\ld,\Ld} (P_x) \quad\forall x\in M. 
 \end{equation*} 
 Assuming    that $F(0)=0,$ we restrict ourselves to the degenerate and singular   operators involving    the  Pucci operators. 

For  singular elliptic operators,   
we adopt  the concept   of  viscosity solutions    proposed by  Birindelli and Demengel;   see  for instance,   \cite{BD, DFQ2,ACP} and the references therein. The notion of  adapted  viscosity solutions  takes into account  the fact that we can not test functions when the gradient of the test functions  vanishes  at the testing point.  

 \begin{definition}[Viscosity solutions, \cite{BD}] \label{def-visc-sol}
Let $\left(TM\setminus\{{\bf0 }\}\right)\times_M \Sym TM:=\{ (\zeta,A): \zeta\in T_xM\setminus\{0\}, A\in \Sym TM_x\,\,\forall x\in M \}$, where ${\bf 0}$ denotes the zero section.    
 Let $G: \left(TM\setminus\{{\bf0 }\}\right)\times_M \Sym TM \to \R$, and let $\Omega\subset M$ be an open set.
For a function $f:\Omega\to\R,$ we say that  $u\in C( \Omega)$ is a viscosity supersolution (respectively  subsolution) of the equation $$G\left(\D u, D^2u\right)=f\quad\mbox{in $\Omega$}$$  if the following holds:  for any $x\in\Omega,$    
\begin{enumerate}[(i)]
\item either  for any  $\varphi\in C^2(\Omega)$ such that 
$u-\varphi$ has a local minimum (respectively  maximum)   at $x$ with $\D \varphi(x)\not=0,$ we have 
$$G\left( \D\varphi(x),D^2\varphi(x)\right)\leq f(x)\quad\mbox{(respectively $\geq $)},$$
\item or  there exists a ball $B_\rho(x)\subset \Omega$ for some $\rho>0$ such that $u$ is a constant function on $B_\rho(x)$,   and   $f\geq 0$ (respectively  $f\leq 0$) on $B_\rho(x).$ 
\end{enumerate}
We say $u$ is  a viscosity solution if $u$ is  both a viscosity subsolution and a viscosity supersolution.
 
\end{definition}
The adapted  definition     is equivalent to  the usual viscosity solution  when  the   operator $G$ defined in $TM\times_M \Sym TM$  is  continuous.  
To prove this fact,  we modify   the Euclidean argument   in the proof of \cite[Lemma 2.1]{DFQ2}. 

\begin{lemma}\label{lem-visc-sol-singular-usual}
Let $\Omega\subset M$ be an open set, and  $f\in C(\Omega).$ Let   $G: TM\times_M \Sym TM\to \R$ 
be  continuous such that  
\begin{equation}\label{eq-assump-cont-op}
G\left({\bf0},{\bf0}\right)=0 \quad\mbox{and}\quad  G\left(\zeta,A\right)\leq 0
\end{equation}
for any $(\zeta, A)\in T_xM\times \Sym TM_x$ with $A\leq 0$ for any $x\in\Omega$.
Then $u$ is a viscosity supersolution  of
\begin{equation}\label{eq-general}
G\left( \D u, D^2 u \right)= f\quad\mbox{in $\Omega$  }
\end{equation}
 in the sense of Definition \ref{def-visc-sol} if and only if $u$ is a viscosity supersolution  of \eqref{eq-general} in the usual sense. 
\end{lemma}
 
\begin{proof}
First we assume that $u$ is a viscosity solution in the usual sense. Then it is not difficult  to show $u$ is a viscosity solution  in the sense of Definition \ref{def-visc-sol}  using  the assumption that  $G\left({\bf0},{\bf0}\right)=0.$

We will prove that $u$ is a viscosity solution in the usual sense 
    if $u$ is a viscosity solution  in the sense of Definition \ref{def-visc-sol}. Let $\vp\in C^2(\Omega)$,  and $x_0 \in \Omega$ be such that $u-\vp$ has a local minimum   at $x_0$. We may  assume  that $x_0$ is a strict local minimum of $u-\vp$.  If $u$ is a  constant function  in a small ball   $B_ \rho(x_0)\subset \Omega$ ($\rho>0$), then $\vp$ has a local maximum at $x_0$. Then we have   $D^2\vp(x_0) \leq 0$, and hence 
$$G\left(\D \vp(x_0),D^2\vp(x_0) \right)=G\left(0,D^2\vp(x_0) \right)\leq0\leq f(x_0)$$ from the assumption  \eqref{eq-assump-cont-op}  and    Definition \ref{def-visc-sol}. 

Now we assume that $u$ is not constant in a small ball $B_\rho(x_0)\subset\Omega$ and  $\D \vp(x_0) =0$ since  there is nothing to prove in the case   $\D\vp(x_0)\not\eq0.$  
First, we consider the case that  $D^2\vp(x_0)$ is nonsingular. 
We introduce a coordinate map $\psi:B_r(0)\subset \R^n\to B_{\rho}(x_0)\subset \Omega$ such that $\psi(0)=x_0$, where  we assume that  $\rho$ is  sufficiently small.   Define $$\tilde u:=u\circ\psi\quad\mbox{and}\quad \tilde \vp:=\vp\circ \psi.$$ Then $\tilde u-\tilde\vp $ has a local minimum at $0$ in $B_{r}(0),$ and $\tilde u$ is not constant in $B_{r}(0).$ One can check that $\D\tilde\vp(0)=0,$ and  $D^2\tilde\vp(0)$ is   nonsingular
since  $\D \vp(x_0) =0$ and $D^2\vp(x_0)$ is nonsingular. 
 Then $0\in\R^n$ is  the only critical point of $\D \tilde\vp$ in  $B_{r'}(0)\subset\R^n$ for a small  $0<r'<r.$  
Using the argument   in the proof of \cite[Lemma 2.1]{DFQ2}, we find   sequences $\{\tilde \vp_k\}_{k=1}^\infty \subset C^2(\overline B_{r'}(0)),$  and  
 $\left\{\tilde x_k\right\}_{k=1}^\infty\subset B_{r'}(0),$ 
 such that 
 \begin{equation*}
 \left\{
 \begin{split} 
&\tilde u-\tilde \vp_k\,\,\mbox{has a local minimum at $\tilde x_k$ in $B_{r'}(0)\subset \R^n,$} \\
 &\D \tilde\vp_k(\tilde x_k)\not\eq 0\quad\forall k=1,2,3,\cdots,\\
 & \tilde\vp_k \to \tilde \vp\quad\mbox{in $C^2(\overline B_{r'}(0))$},\quad\mbox{and}\quad\tilde x_k\to 0 \quad\mbox{as $k\to\infty$}.
 \end{split}
 \right.
 \end{equation*}
Thus there exist  
  sequences $\{\vp_k:=\tilde \vp_k\circ\psi^{-1}|_{\overline B_{\rho'}(x_0)}\}_{k=1}^\infty \subset C^2(\overline  B_{\rho'}(x_0)),$  and  
 $\left\{  x_k:=\psi(\tilde x_k)\right\}_{k=1}^\infty\subset B_{\rho'}(x_0)$ for a sufficiently small $0<\rho'<\rho,$ such that 
 \begin{equation*}
 \left\{
 \begin{split} 
&  u-  \vp_k\,\,\mbox{has a local minimum at $  x_k$ in $B_{\rho'}(x_0)\subset\Omega,$} \\
 &\D  \vp_k(x_k)\not\eq 0\quad\forall k=1,2,3,\cdots,\\
   &  \vp_k \to   \vp\quad\mbox{in $C^2(\overline  B_{\rho'}(x_0))$},\quad\mbox{and}\quad  x_k\to x_0 \quad\mbox{as $k\to\infty$}.
 \end{split}
 \right.
 \end{equation*} 
 From   Definition \ref{def-visc-sol}, it follows that 
$$G\left(\D\vp_k(x_k), D^2\vp_k(x_k) \right)\leq f(x_k) \quad\forall k=1,2,\cdots,$$
and hence the continuity of the operator $G$ and the function $f$ implies that 
$$G\left(\D\vp(x_0), D^2\vp(x_0) \right)\leq f(x_0).$$

In the case that $D^2\vp(x_0)$ is   singular,  we consider   for   $\delta>0,$
$$\vp_\delta(x):=\vp(x)-\delta d_{x_0}^2(x)/2\quad\forall x\in B_{\rho}(x_0),$$
where we assume that $\rho>0$ is   small enough. 
For $\delta>0,$ the function $u-\vp_\delta$ has   a local minimum  at $x_0$  in $B_\rho(x_0)$ with $\D \vp_\delta(x_0)=\D \vp(x_0)=0.$  
 Let  $\displaystyle\delta_0:=\min_{\mu_i\not\eq0}|\mu_i|$ for the eigenvalues $\mu_i$ of $D^2\vp(x_0).$
     For $0<\delta< \delta_0,$
   the Hessian 
$D^2\vp_\delta(x_0)=D^2\vp(x_0)-\delta {\bf I}$ is nonsingular.  
Then we apply the previous argument to $\vp_\delta$ to obtain 
$$G\left(\D\vp(x_0), D^2\vp(x_0)-\delta {\bf I}\right)=G\left(\D\vp_\delta(x_0), D^2\vp_\delta(x_0)\right)\leq f(x_0).$$ Letting 
   $\delta$ go to $ 0,$  we conclude that 
$$G\left(\D\vp(x_0), D^2\vp_(x_0) \right)\leq f(x_0),$$ which finishes the proof.    
\end{proof}

 Lemma \ref{lem-visc-sol-singular-usual} asserts that  a viscosity solution $u$  to the $p$-Laplacian type equation   for $p\geq 2$ with a continuous source term   solves the equation in the usual viscosity sense.  

For the operators with singularities   ($1<p<2$),  we make use of  the following lemma  in \cite[Lemma 2]{ACP}; see also \cite{DFQ1}. 
\begin{lemma}\cite[Lemma 2]{ACP}\label{lem-pucci-delta-p-laplace-delta}
Let $1<p<2$, 
and let $u$ be a viscosity supersolution to $\La_pu\leq f$ in $\Omega$ with $f\geq0.$ Then  $u$ is  a viscosity supersolution to
$$\left(|\D u(x)|^2+\delta\right)^{\frac{p-2}{2}}\cM^-_{p-1,1}(D^2u)\leq f\quad\mbox{in $\Omega$}$$
for any $  \delta>0.$
\end{lemma}   
According to  Lemma \ref{lem-visc-sol-singular-usual},  if $u $ is a viscosity supersolution to
   the $p$-Laplacian   equation   for $1<p<2$ with a nonnegative, continuous source term,  then $u$ solves the regularized equations above   in the usual viscosity sense.

Now we recall   
the sup- and inf-convolutions  introduced by Jensen \cite{J} (see also   \cite[Chapter 5]{CC}) to approximate    viscosity solutions.    Let $\Omega\subset M$  be  a bounded open set,
and    $u$ be  continuous   on $\overline \Omega.$    For $\ve>0,$  let $u_\ve$ denote 
  the inf-convolution  of $u$   (with respect to $ \Omega $)   defined as follows: for $x_0\in  \overline\Omega,$
\begin{equation*}
u_{\ve}(x_0 ):=\inf_{y\in  \overline\Omega  } \left\{ u(y) +\frac{1}{2\ve}d^2(y, x_0)\right\}.
 \end{equation*}
 The sup-convolution of $u$ denoted by $u^\ve$ is defined in a similar way. 
In the following lemmas, we quote the important properties of the  inf-convolution from      \cite[Section 3]{KL};  refer to  \cite{J}  for the Euclidean case.
 
\begin{lemma} \label{lem-visc-u-u-e-properties}
 Let $\Omega\subset M$  be  a bounded open set,
and    $u\in C\left(\overline \Omega\right).$ 
  Assume that  for any $x,y\in\overline \Omega$,    the sectional curvature   along  the  minimizing geodesic joining $x$ to $y$  
   has a uniform lower bound $-\kappa $ for $\kappa\geq0.$ 
 \begin{enumerate}[(a)]
\item $u_\ve \uparrow u$ uniformly in $\overline\Omega$ as $\ve \downarrow0.$ 
 \item $u_\ve$ is Lipschitz continuous in $\overline\Omega$:   for $x_0, x_1\in \overline\Omega,$
 $$|u_{\ve}(x_0)-u_\ve(x_1)|\leq \frac{3}{2\ve}\diam(\Omega) d(x_0,x_1).$$
 \item $u_\ve$ is $C_\ve$-semiconcave in $\Omega$, where $C_\ve:=\frac{1}{\ve}{\sqrt{\kappa}\diam(\Omega)\coth\left(\sqrt \kappa \diam(\Omega)\right)}.$   For almost every $x\in\Omega,$ $u_\ve$ is differentiable at $x,$ and there exists  the Hessian  $D^2u_\ve(x)$ in the sense of Theorem \ref{thm-AB} such that
 \begin{equation*}
 u_\ve\left(\exp_x \xi \right)=u_\ve(x)+\left\langle\D u_\ve(x),\xi\right\rangle+\frac{1}{2}{\left\langle D^2u_\ve(x)\cdot\xi,\xi\right\rangle}+o\left(|\xi|^2\right)
 \end{equation*} 
  as $\xi\in T_xM\to 0.$  
   \item $\displaystyle D^2u_\ve\leq \frac{1}{\ve}{\sqrt{\kappa}\diam(\Omega)\coth\left(\sqrt \kappa \diam(\Omega)\right)}\,  {\bf I}\,\,\,$ a.e. in $\Omega.$
 \item  Let $H$  be an  open set such that $  \overline H\subset\Omega.$ 
  Then  there exist a smooth function $\psi$ on $M$ satisfying    $$\mbox{$0\leq\psi\leq1$ on $M$},\,\,\,\,\,\psi\equiv1 \,\,\,\,\mbox{in $\overline H\,\,\,\,\,$ and }\,\,\,\,\supp\psi\subset \Omega, $$ and a sequence $\{w_k\}_{k=1}^\infty$ of smooth functions on $M$  such that
\begin{equation*} 
\left\{
\begin{array}{ll}
 w_k\to \psi u_\ve\qquad &\mbox{uniformly in $M\,\,$  as $k\to\infty,$}\\ 
 |\D w_k| \leq C,\,\,\,
      D^2w_k\leq C  {\bf I}\qquad &\mbox{in $M$, }\\
\D w_k\to \D u_\ve,\,\,\, D^2 w_k\to D^2u_\ve\quad&\mbox{a.e. in $H\,\,$ as $k\to\infty,$}
 \end{array}\right. \qquad \qquad
\end{equation*}
where the constant $C>0$ is independent of $k$.
  \end{enumerate}

\end{lemma}

We define the second order superjet and subjet of semi-continuous functions. 

\begin{definition}
 Let 
  $u: \Omega\to\R$ be a lower semi-continuous function on an open set $\Omega\subset M$. We define the second order subjet of $u$ at $x\in \Omega $ by
\begin{align*}
\cJ^{2,-}u(x):=&\left\{\left(\D\varphi(x),D^2\varphi(x)\right)\in T_xM\times\Sym TM_x: \varphi\in C^2(\Omega), \right.\\
&\left. u-\varphi \,\,\,\mbox{has a local minimum at $x$}\right\}.
\end{align*}
If $(\zeta,A)\in \cJ^{2,-}u(x),$ $\zeta$ and $A$ are  called a first order subdifferential  and a second order subdifferential  of $u$ at $x,$ respectively.  

Similarly, for an upper semi-continuous function
 $u: \Omega\to \R,$ we define the second order superjet of $u$ at $x\in \Omega$ by
\begin{align*}
\cJ^{2,+}u(x):=&\left\{\left(\D\varphi(x),D^2\varphi(x)\right)\in T_xM\times\Sym TM_x: \varphi\in C^2(\Omega), \right.\\
&\left.u-\varphi \,\,\,\mbox{has a local maximum at $x$}\right\}. 
\end{align*}
\end{definition}

From the proof of  \cite[Proposition 3.3]{KL}, we deduce  the following lemma.  
  \begin{lemma}\label{lem-visc-u-u-e}  
 Let $H$ and $ \Omega$  be   bounded open sets in $M$ 
 such that $  \overline H\subset\Omega.$   Assume that $\Sec
  \geq -\kappa $ on $ \Omega$ for $\kappa\geq0$.  Let   $u\in C\left(\overline\Omega\right)$   and let $\omega$ denote  a modulus of continuity of $u$ on $\overline \Omega,$ which is nondecreasing  on $(0,+\infty)$ with $\omega(0+)=0.$  
  For $\ve>0,$ let $u_\ve$ be  the inf-convolution of $u$ with respect to $ \Omega.$  
Then  there exists $\ve_0>0$ depending only  on $||u||_{L^{\infty}\left(\overline\Omega\right)} , H,$ and $\Omega,$   such that   if $0<\ve<\ve_0,$  then the following    holds: let  $x_0\in \overline H,$  and let $y_0\in\overline\Omega$ be such that   
$$u_\ve(x_0)=u(y_0)+\frac{1}{2\ve} d^2(y_0,x_0).$$
\begin{enumerate}[(a)]
\item    $y_0$ is an interior point of $\Omega,$ and there is a unique minimizing geodesic joining  $x_0$ to $y_0$,  which is contained in $\Omega.$
\item If $(\zeta,A)\in \cJ^{2,-}u_\ve(x_0),$ 
then   
$y_0=\exp_{x_0}(-\ve\zeta),$ and 
 $$ \left(\, L_{x_0,y_0}\zeta, \,L_{x_0,y_0}  A-{\color{black}  {\kappa}{ } \min\left\{\ve  |\zeta|^2,  2\omega\left(2\sqrt{\ve||u||_{L^{\infty}\left(\overline\Omega\right)}}\right)\right\} }\,  {\bf I}\right)\,\in \cJ^{2,-}u(y_0),$$
where $L_{x_0,y_0}$ stands for the parallel transport along the unique minimizing geodesic joining $x_0$ to $y_0,$ and     $L_{x_0,y_0}  A$  is a symmetric bilinear form on $T_{y_0}M$ defined as   
$$\left\langle\left( L_{x_0,y_0}  A\right)\cdot\nu,\,\nu\right\rangle_{y_0}:=\left\langle A\cdot\left(L_{y_0,x_0}\nu\right),\, L_{y_0,x_0}\nu\right\rangle_{x_0}\quad\forall\nu\in T_{y_0}M.$$

 \end{enumerate}
\end{lemma}

Using Lemma  \ref{lem-visc-u-u-e},  it can be  proven  that the inf-convolution  solves approximated equation in the viscosity sense of Definition \ref{def-visc-sol}, provided that the sectional curvature   is bounded from below, and the operator $G$ is intrinsically uniformly continuous with respect to $x.$
 The  intrinsic uniform continuity of the operator with respect to spatial variables  is  a natural extension of the Euclidean concept of uniform continuity of  the operator  with respect to $x$; we recall    from \cite{AFS} the definition and important examples.  
 
\begin{definition}
The operator
 $G$ is {\color{black}intrinsically uniformly continuous with respect to $x$} in $\left(TM \setminus\{{\bf 0}\}\right)\times_M \Sym TM$
if there exists a modulus of continuity $\omega_G:[0,+\infty)\to[0,+\infty)$ with $\omega_G(0+)=0$  such that 
\begin{equation}\label{Hypo2}
G\left(\zeta, A\right)-G\left(L_{x,y} \zeta, L_{x,y}  A\right)\leq \omega_G\left(d(x,y)\right)
\end{equation}
for  any $(\zeta, A)\in \left(T_xM\setminus\{0\} \right)\times \Sym TM_x $, and    $x,y\in M$ {\color{black} with $d(x,y)<\min\left\{i_M(x),\,i_M(y)\right\},$} where $i_M(x)$ denotes the injectivity radius at $x$ of $ M$. 

\end{definition}


\begin{example} \label{exp-AFS}
{\rm 
\begin{enumerate}[(a)]
\item   
 When $M=\R^n,$  we have $L_{x,y}   \zeta \equiv \zeta$, $L_{x,y}   A\equiv A$ so \eqref{Hypo2} holds.    
\item 
Consider the operator $G,$  which depends only on $|\zeta|$ and  the eigenvalues of $A\in\Sym TM ,$ of the form :
\begin{equation}\label{eq-F-eigenvalue}
G\left(\zeta, A\right)=\tilde G\left(|\zeta|, \,\mathrm{eigenvalues \,\,of } \,\,A \right)\quad\mbox{for some $\tilde G.$}
\end{equation} 
Since the parallel transport preserves inner products,  we have $|\zeta|=|L_{x,y}\zeta|$, and $A$ and $L_{x,y}  A$ have the same eigenvalues. Thus  the operator $G$ satisfies  intrinsic uniform continuity with respect to $x$    with  $\omega_G\equiv0$. The trace and determinant of $A$  are typical examples   of the  operator satisfying \eqref{eq-F-eigenvalue}. 
The Pucci  extremal operators    also  satisfy \eqref{eq-F-eigenvalue} and hence \eqref{Hypo2} with  $\omega_G\equiv0$. 
\item As  seen before, the $p$-Laplacian operator  
  can be expressed as  
$$ G\left(\D u, D^2 u\right)=|\D u|^{p-2}\tr\left[ \left( {\bf{I}}+(p-2)\frac{\D u}{|\D u|}\otimes \frac{\D u}{|\D u|}\right)D^2u\right].$$
One can check    that  for  any $(\zeta, A)\in \left(T_xM\setminus\{0\} \right)\times \Sym TM_x $ and    $x,y\in M$ {\color{black} with $d(x,y)<\min\left\{i_M(x),\,i_M(y)\right\},$} 
\begin{align*}
\left\langle \left(L_{x,y}\zeta\otimes L_{x,y}\zeta \right)L_{x,y}A\cdot \nu,\,\nu\right\rangle_y=\left\langle \left( \zeta\otimes \zeta \right) A\cdot  L_{y,x}\nu, \,L_{y,x}\nu \right\rangle_x\quad\forall \nu \in T_yM. 
\end{align*}
Thus the $p$-Laplacian operator is   {\color{black}intrinsically uniformly continuous} in $\left(TM \setminus\{{\bf 0}\}\right)\times_M \Sym TM$ with $\omega_G\equiv0$.  
\end{enumerate}
}

\end{example}

By making use of Lemma \ref{lem-visc-u-u-e},  we have the following lemma  whose   proof  is similar to \cite[Lemma 3.6]{KL}.  
\begin{lemma}\label{prop-u-u-e-superj}
 Under the same assumption  as Lemma  \ref{lem-visc-u-u-e}, we also assume that  $G$ is {\color{black}intrinsically uniformly continuous} in $\left(TM \setminus\{{\bf 0}\}\right)\times_M \Sym TM$. 
 Let  $ f\in C ( \Omega    ),$ 
and  $u$  be  a viscosity supersolution to  
 $$G\left(\D u, D^2u\right)\leq f\quad \mbox{in $\Omega$}.$$   
If $0<\ve<\ve_0,$ then  $u_\ve$ is a viscosity supersolution of 
  $$G_\ve\left(\D u, D^2u\right):=G\left(\D u, D^2u- {\kappa}{} \min\left\{\ve |\D u|^2,  2\omega\left(2\sqrt{ \ve m }\right)\right\} {\bf I}\right) \leq f_\ve\quad\mbox{in $H$},$$
  where 
  $\ve_0>0$ is the constant as in Lemma  \ref{lem-visc-u-u-e},  and 
  $$f_{\ve}(x):= \sup_{\overline B_{2\sqrt{m\ve}}(x)}f^++{\color{black}\omega_G}\left(2\sqrt{ \ve m}\right);\quad m:=||u||_{L^{\infty}\left(\overline\Omega\right)}.$$  
 Moreover,  we have 
   $$G\left(\D u_\ve, D^2u_\ve- {\kappa}{} \min\left\{\ve |\D u_\ve|^2,  2\omega\left(2\sqrt{ \ve m}\right)\right\} {\bf I}\right) \leq f_\ve \quad\mbox{ a.e. in $H\cap\{ |\D u_\ve|>0\}.$}$$
\end{lemma}

\begin{remark}\label{rmk-p-La-eq-u-e}
 Under the same assumption  as Lemma  \ref{prop-u-u-e-superj}, if  
$$G\left(\D u, D^2 u\right):=\La_pu=|\D u|^{p-2}\tr\left[ \left( {\bf{I}}+(p-2)\frac{\D u}{|\D u|}\otimes \frac{\D u}{|\D u|}\right)D^2u\right],$$
then   $u_\ve$ satisfies 
\begin{equation}\label{eq-u-e-a.e.}
 \La_pu_\ve - \kappa(n+p-2)|\D u_\ve|^{p-2}  {}{ } \min\left\{\ve|\D u_\ve|^2,  2{ }\omega\left(2\sqrt{ \ve m}\right)\right\} \leq f_\ve\quad\mbox{in $H$}
 \end{equation}
in the viscosity sense,   and   almost everywhere in $H\cap\{ |\D u_\ve|>0\}$,  where  $\omega_G\equiv0$  from Example \ref{exp-AFS}.   
When $p\geq 2,$ the $p$-Laplacian operator is continuous,  and hence \eqref{eq-u-e-a.e.} is satisfied almost everywhere in $H$ by   Lemma \ref{lem-visc-sol-singular-usual}. 
\end{remark}

\section{ABP type estimates }\label{sec-abp}

In this section, we  establish  the ABP type estimates for the $p$-Laplacian type operators for $1<p<\infty$.
We begin with the definition of the $p$-contact set involving   
the $\frac{p}{p-1}$-th power of the distance.       

   \begin{definition}[Contact sets]\label{def-p-contact-set}
Let $1<p<\infty.$  Let $\Omega$ be a bounded open set in $M$ and   $u\in C(\overline\Omega).$ For a given $a>0$ and a compact set $E\subset M,$ the $p$-contact set associated with $u$ of opening $a$ with the  vertex set $E$ is defined by
  $$\cA^p_{a}\left(E;\overline \Omega;u\right):=\left\{x\in\overline\Omega : \,\exists y\in E \,\mbox{such that}\,\inf_{ \overline \Omega}\left\{u+a\frac{p-1}{p}d_y^{\frac{p}{p-1}}\right\}=u(x)+ a\frac{p-1}{p}d_y^{\frac{p}{p-1}}(x) \right\}.$$
 \end{definition}
 When  $p=2,$ a contact point is touched by a concave paraboloid  from below, which is a squared distance function;  refer to \cite{WZ, Ca, K}.  
   Notice that  the $p$-contact set  $\cA^p_{a}\left(E;\overline \Omega;u\right)$ is a closed set. 

\subsection{Jacobian estimates}

First, we study  Jacobian estimates for certain exponential maps  which arise in the proof   of the ABP type estimate. 
\begin{prop}[Jacobian estimate on the   contact set]\label{prop-jacobi-estimate}
Let $1<p<\infty$, $E\subset M$ be a compact set,       
and    $u\in    C(\overline\Omega)$ be smooth    in  a bounded open set  $\Omega\subset M.$  Define $\phi_p:  \{z\in\Omega : |\D u(z)|>0\}\times[0,1]\to M$  as 
 $$\phi_p(z,t):=\exp_z t{|\D u(z)|^{p-2}}\D u(z).$$
Assume that   $x\in \cA^p_{1}\left(E;\overline \Omega;u\right) \cap \{z\in\Omega : |\D u(z)|>0\}. $ 
  Then  the following holds. 
  \begin{enumerate}[(a)]
  \item    If $y\in E$ satisfies  
  $$  \inf_{ \overline \Omega}\left\{u+\frac{p-1}{p}d_{y}^{\frac{p}{p-1}}\right\}=u(x)+ \frac{p-1}{p}d_{y}^{\frac{p}{p-1}}(x), $$
then      $$y=\exp_x {|\D u|^{p-2}}\D u(x)\not\in \Cut(x)\cup\{x\},\quad\mbox{and}\quad {|\D u|^{p-2}}\D u(x)\in\cE_x.$$ 
  The curve   
        $\gamma:=\phi_p(x,\cdot):[0,1]\to M$ is a unique minimizing geodesic joining $ x$ to  $ y=\exp_x {|\D u|^{p-2}}\D u(x)$. 

\item   $\displaystyle    |\D u|^{p-2}   D^2u(x) +\left({\bf I}+\frac{2-p}{p-1} \frac{\D u}{|\D u|}\otimes \frac{\D u}{|\D u|}\right)\circ D^2 (d_{\phi_p(x,1)}^2/2)(x)   $ is {symmetric and}  positive semi-definite. 

    \item 
    \begin{align*}
&\Jac_z\phi_p(x,1)\\
= &\Jac  \exp_x \left({|\D u(x)|^{p-2}}\D u(x)\right) \\
& \cdot\det \left\{  |\D u|^{p-2}  \left({\bf I}+(p-2) \frac{\D u}{|\D u|}\otimes \frac{\D u}{|\D u|}\right)\circ D^2u(x) +D^2 (d_{\phi_p(x,1)}^2/2)(x)\right\} \\
= &\Jac  \exp_x \left({|\D u(x)|^{p-2}}\D u(x)\right)\cdot \det\left({\bf I}+(p-2) \frac{\D u}{|\D u|}\otimes \frac{\D u}{|\D u|}(x)\right) \\
\,\,&\cdot\det \left\{  |\D u|^{p-2}   D^2u(x) +\left({\bf I}+\frac{2-p}{p-1} \frac{\D u}{|\D u|}\otimes \frac{\D u}{|\D u|}\right)\circ D^2 (d_{\phi_p(x,1)}^2/2)(x)\right\},  
\end{align*}
where 
 $\Jac  \exp_x \left({|\D u(x)|^{p-2}}\D u(x)\right)$ stands for  the Jacobian determinant  of $\exp_x$, a map from $T_xM$ to $M,$ at the point $  {|\D u(x)|^{p-2}}\D u(x)\in T_xM. $
   \item   For $t\in[0,1]$,  the operator $\displaystyle \,t  |\D u|^{p-2}   D^2u(x) +\left({\bf I}+{\color{black} t\,\frac{2-p}{p-1} }\frac{\D u}{|\D u|}\otimes \frac{\D u}{|\D u|}\right)\circ D^2 (d_{\phi_p(x,t)}^2/2)(x)  $ is symmetric and  positive semi-definite.

\item If  $1<p<2,$ then  for $ t\in(0,1]$ $$t{|\D u(x)|^{p-2}}  D^2u(x)+  \left({\bf I}+{\frac{2-p}{p-1}} \frac{\D u}{|\D u|}\otimes \frac{\D u}{|\D u|}\right) \circ D^2(d^2_{\phi_p(x,t)}/2)(x),$$
  is symmetric and  positive semi-definite,  and 
  \begin{align*}
&\Jac_z\phi_p(x,t)\\
= &\Jac  \exp_x \left(t{|\D u(x)|^{p-2}}\D u(x)\right) \\
& \cdot\det \left\{  t|\D u|^{p-2}  \left({\bf I}+(p-2) \frac{\D u}{|\D u|}\otimes \frac{\D u}{|\D u|}\right)\circ D^2u(x) +D^2 (d_{\phi_p(x,t)}^2/2)(x)\right\} \\
= &\Jac  \exp_x \left(t{|\D u(x)|^{p-2}}\D u(x)\right)\cdot \det\left({\bf I}+(p-2) \frac{\D u}{|\D u|}\otimes \frac{\D u}{|\D u|}\right) \\
\,\,&\cdot\det \left\{  t|\D u|^{p-2}   D^2u(x) +\left({\bf I}+\frac{2-p}{p-1} \frac{\D u}{|\D u|}\otimes \frac{\D u}{|\D u|}\right)\circ D^2 (d_{\phi_p(x,t)}^2/2)(x)\right\}.
\end{align*}
 
  
\end{enumerate}
\end{prop}
\begin{proof}
  \begin{enumerate}[(a)]
  \item
  For any  $x\in\cA^p_1\left(E;\overline \Omega;u\right),$ we find $y\in E$ such that 
\begin{equation}\label{eq-inf-contact-pt}
\inf_{  \Omega}\left\{u+\frac{p-1}{p}d_y^{\frac{p}{p-1}}\right\}=u(x)+\frac{p-1}{p}d_y^{\frac{p}{p-1}}(x).
\end{equation} Then we have 
\begin{align*}
    \frac{p-1}{p}\left\{ d_y^{\frac{p}{p-1}}(z)-d_y^{\frac{p}{p-1}}(x)\right\}\geq - u(z)+u(x) \quad\forall z\in\Omega.
\end{align*}
We first claim that $x\not\in\Cut(y)$ for $x\in\cA^p_1\left(E;\overline \Omega;u\right)\cap\Omega$  (see also \cite{WZ}).  Suppose to the contrary that $x\in\Cut(y).$ Note that $x\not\eq y$ if $x\in \Cut(y).$ 
Letting   $\psi_p(s):=\frac{p-1}{p}s^{\frac{p}{2(p-1)}},$  we have    $\psi'_p >0$ in $(0,\infty).$ According to Corollary \ref{cor-dist-sqrd-cut-compoed-increasing-ft}, there is  a unit vector $X\in T_xM$ such that 
 $$\liminf_{t\to0}\frac{\psi\left(d_y^2\left(\exp_xtX\right)\right)+\psi\left(d_y^2\left(\exp_x -tX\right)\right)-2\psi\left(d^2_y(x)\right)}{t^2} =-\infty,$$
which contradicts to 
\begin{align*}
&\liminf_{t\to0}\frac{\psi\left(d_y^2\left(\exp_xtX\right)\right)+\psi\left(d_y^2\left(\exp_x -tX\right)\right)-2\psi\left(d^2_y(x)\right)}{t^2}\\
&\geq -\lim_{t\to0} \frac{ u\left(\exp_xtX\right)+u\left(\exp_x-tX\right) -2 u\left(x\right)}{t^2}
=-\left\langle D^2u(x)\cdot X,X\right\rangle.
\end{align*} 
 Hence $x$ is not a cut point of $y.$ 
 

Since  $x\not\in\Cut(y)$ and $p>1,$ 
   \eqref{eq-inf-contact-pt}   yields   that 
   $$\D u(x)=-\frac{p-1}{p}\D d_y^{\frac{p}{p-1}}(x)$$
from which we see that $${\color{black} y=x \iff \D u(x)= 0}. $$
 Thus for $x\in \cA^p_{1}\left(E;\overline \Omega;u\right) \cap \{z\in\Omega : |\D u(z)|>0\},$  we have   $y\not\in \Cut(x)\cup \{x\},$ and hence  
\begin{align*}
&\D u(x)=-d_y^{\frac{1}{p-1}}(x)\D d_y(x),\quad |\D u(x)|=  d_y^{\frac{1}{p-1}}(x),
\end{align*}
and 
\begin{equation*}
  y=\exp_x  -\D d^2_y(x)/2=\exp_x {|\D u(x)|^{p-2}} \D u(x)\not\in \Cut(x)\cup\{x\},
\end{equation*}
which implies  (a). 

  \item  From \eqref{eq-inf-contact-pt},  we     have  that for $y=\phi_p(x,1),$
\begin{align*}
D^2u(x)&\geq  -\frac{p-1}{p}D^2d_y^{\frac{p}{p-1}}(x)\\
&=
-d_y^{\frac{2-p}{p-1}}(x)\left\{ D^2(d_y^2/2)(x)+\frac{2-p}{p-1} \D d_y(x)\otimes \D d_y(x)\right\}\\
&=-|\D u(x)|^{2-p}\left\{D^2(d_y^2/2)(x)+ \frac{2-p}{p-1}\D d_y(x)\otimes \D d_y(x)\right\},
\end{align*}
where we recall  the symmetric operator $\D d_y(x)\otimes \D d_y(x)$  defined by 
$$\left\langle \D d_y(x)\otimes \D d_y(x)\cdot X,Y\right\rangle:=\left\langle\D d_y(x), X\right\rangle \left\langle\D d_y(x), Y\right\rangle\quad\forall X, Y\in T_xM.$$
Since $\D d_y(x) $ is an eigenvector of  $D^2(d_{y}^2/2)(x)$ associated with the  eigenvalue $1$, 
we obtain that for $X, Y\in T_xM,$
\begin{align*}
\left\langle \D d_y(x)\otimes \D d_y(x) \circ  D^2(d^2_{y}/2)(x) \cdot X, Y\right\rangle&= \left\langle  \D d_y(x),  D^2(d^2_{y}/2)(x) \cdot X\right\rangle \left\langle\D d_y(x), Y\right\rangle\\
&= \left\langle X,  D^2(d^2_{y}/2)(x) \cdot  \D d_y(x)\right\rangle \left\langle\D d_y(x), Y\right\rangle\\
&= \left\langle X,   \D d_y(x)\right\rangle \left\langle\D d_y(x), Y\right\rangle\\
&= \left\langle \D d_y(x)\otimes \D d_y(x)\cdot X,Y\right\rangle,
\end{align*}
and hence 
$ \D d_y(x)\otimes \D d_y(x) =\D d_y(x)\otimes \D d_y(x) \circ  D^2(d^2_{y}/2)(x) =  \frac{\D u}{|\D u|}\otimes \frac{\D u}{|\D u|} \circ  D^2(d^2_{y}/2)(x). 
$ 
Therefore,   we deduce that 
\begin{equation}\label{eq-nonvanishing-jac-det}
\begin{split}
0\leq &{|\D u(x)|^{p-2}}  D^2u(x)+  D^2(d_y^2/2)(x)+ \frac{2-p}{p-1}  \D d_y(x)\otimes \D d_y(x)\\ 
= &{|\D u(x)|^{p-2}}  D^2u(x)+    \left({\bf I}+\frac{2-p}{p-1} \frac{\D u}{|\D u|}\otimes \frac{\D u}{|\D u|}\right)\circ  D^2(d^2_{y}/2)(x)
\end{split}
\end{equation}
which is symmetric.  


    \item From Lemma \ref{lem-jacobian-exp}, we have  that for $y=\phi_p(x,1),$
  \begin{equation*}
  \begin{split}
 \Jac_z \phi_p(x,1)
=& \Jac \exp_x \left({|\D u(x)|^{p-2}}\D u(x)\right)
\\& \cdot\left|\det \left\{  |\D u(x)|^{p-2}  \left({\bf I}+(p-2) \frac{\D u}{|\D u|}\otimes \frac{\D u}{|\D u|}\right)\circ D^2u(x) +D^2 (d_{y}^2/2)(x)\right\} \right|\\
= &\Jac  \exp_x \left({|\D u(x)|^{p-2}}\D u(x)\right)\cdot \det\left({\bf I}+(p-2) \frac{\D u}{|\D u|}\otimes \frac{\D u}{|\D u|}\right) \\
\,\,&\cdot\left|\det \left\{  |\D u|^{p-2}   D^2u(x) +\left({\bf I}+\frac{2-p}{p-1} \frac{\D u}{|\D u|}\otimes \frac{\D u}{|\D u|}\right)\circ D^2 (d_{y}^2/2)(x)\right\}\right|. 
\end{split}
\end{equation*}
Thus  \eqref{eq-nonvanishing-jac-det} implies (c).  

  \item According to Lemma \ref{lem-hess-dist-sqrd-geodesic}, we see that  for $t\in[0,1]$
  \begin{align*}
   D^2(d^2_{\gamma(t)}/2)(x) -t \,D^2(d^2_{\gamma(1)}/2)(x)\geq0. 
  \end{align*} 
  Since $ \D d_y(x)=-\frac{\D u(x)}{|\D u(x)|}$ is an eigenvector of $D^2(d_{\gamma(t)}^2/2)(x)$ associated with the eigenvalue $1$ for any $t\in[0,1]$,  the above argument yields that 
  \begin{equation}\label{eq-eigenvector-along-gamma}
   \frac{\D u}{|\D u|}\otimes \frac{\D u}{|\D u|} \circ  D^2(d^2_{\gamma(t)}/2)(x)=\frac{\D u}{|\D u|}\otimes \frac{\D u}{|\D u|}(x)  =  \frac{\D u}{|\D u|}\otimes \frac{\D u}{|\D u|} \circ  D^2(d^2_{y}/2)(x).
   \end{equation}
From    \eqref{eq-nonvanishing-jac-det}, 
we deduce  that  for $t\in[0,1]$
  \begin{equation*} 
\begin{split}
 0 & \leq  t{|\D u(x)|^{p-2}}  D^2u(x)+    t\left({\bf I}+\frac{2-p}{p-1} \frac{\D u}{|\D u|}\otimes \frac{\D u}{|\D u|}\right)\circ  D^2(d^2_{\gamma(1)}/2)(x)\\
  &\leq t{|\D u(x)|^{p-2}}  D^2u(x)+  \left({\bf I}+t\frac{2-p}{p-1} \frac{\D u}{|\D u|}\otimes \frac{\D u}{|\D u|}\right)  \circ D^2(d^2_{\gamma(t)}/2)(x). 
\end{split}
\end{equation*}

  \item 
  Since  $1<p<2,$   (d) combined with \eqref{eq-eigenvector-along-gamma}  implies that 
$$t{|\D u(x)|^{p-2}}  D^2u(x)+  \left({\bf I}+{\frac{2-p}{p-1}} \frac{\D u}{|\D u|}\otimes \frac{\D u}{|\D u|}\right) \circ D^2(d^2_{\gamma(t)}/2)(x)\geq0
\quad\forall t\in(0,1].$$
  Using Lemma \ref{lem-jacobian-exp} again, we have    
  \begin{equation*}
  \begin{split}
 \Jac_z \phi_p(x,t)
=& \Jac \exp_x \left(t{|\D u(x)|^{p-2}}\D u(x)\right)
\\& \cdot\left|\det \left\{  t|\D u(x)|^{p-2}  \left({\bf I}+(p-2) \frac{\D u}{|\D u|}\otimes \frac{\D u}{|\D u|}\right)\circ D^2u(x) +D^2 (d_{\gamma(t)}^2/2)(x)\right\} \right|,
\end{split}
\end{equation*}
  and       
\begin{align*}
&\left|\det \left\{  t|\D u(x)|^{p-2}  \left({\bf I}+(p-2) \frac{\D u}{|\D u|}\otimes \frac{\D u}{|\D u|}\right)\circ D^2u(x) +D^2 (d_{\gamma(t)}^2/2)(x)\right\} \right|\\
=& \det  \left({\bf I}+(p-2) \frac{\D u}{|\D u|}\otimes \frac{\D u}{|\D u|}(x)\right)\\&\cdot \det \left\{t{|\D u(x)|^{p-2}}  D^2u(x)+  \left({\bf I}+\frac{2-p}{p-1} \frac{\D u}{|\D u|}\otimes \frac{\D u}{|\D u|}\right) \circ D^2(d^2_{\gamma(t)}/2)(x)\right\},
\end{align*}
completing the proof of (e). 
\end{enumerate}
\end{proof}
 
Proposition \ref{prop-jacobi-estimate} holds for semiconcave functions for almost all contact points.  

\begin{cor} \label{cor-jacobi-estimate}
Let $1<p<\infty$, $E\subset M$ be a compact set,       
and    $u\in    C(\overline\Omega)$  be      semiconcave   in  a bounded open set  $\Omega\subset M.$ Denote by $N$  a subset of $\Omega $ of measure zero such that $u$ is twice differentiable in $\Omega\setminus N$ in the sense of the   Theorem \ref{thm-AB}.   
   Define $\phi_p:  \{z\in\Omega\setminus N : |\D u(z)|>0\}\times[0,1]\to M$  as 
 $$\phi_p(z,t):=\exp_z t{|\D u(z)|^{p-2}}\D u(z).$$
Then   
Proposition \ref{prop-jacobi-estimate}      holds  true  for      $x\in \cA^p_{1}\left(E;\overline \Omega;u\right) \cap \{z\in\Omega\setminus N : |\D u(z)|>0\}. $
\end{cor}
\begin{proof}
Using   the semi-concavity,  note that  for any $z\in\Omega,$ and a  unit vector $X\in T_zM,$ 
$$\limsup_{t\to0} \frac{ u (\exp_ztX )+u (\exp_z-tX ) -2 u (z )}{t^2}<\infty.$$
 Fix         a point   $x\in  \cA^p_{1}\left(E;\overline \Omega;u\right) \cap \{z\in\Omega\setminus N : |\D u(z)|>0\}.$  Arguing similarly  as in the proof of Proposition \ref{prop-jacobi-estimate} together  with the above property, 
  if  $y\in E$ satisfes 
  $$  \inf_{ \overline \Omega}\left\{u+\frac{p-1}{p}d_{y}^{\frac{p}{p-1}}\right\}=u(x)+ \frac{p-1}{p}d_{y}^{\frac{p}{p-1}}(x), $$ then 
       $y=\phi_p(x,1)=\exp_{x}|\D u|^{p-2}\D u(x)\not\in \Cut(x)\cup\{x\}$, and $|\D u|^{p-2}\D u(x)\in \cE_x\setminus\{0\}.$ 
 Using   Lemma \ref{lem-jacobian-exp}, we deduce that   for $(0,1]$ 
   \begin{equation*}   \begin{split}
d \phi_p(x,t)
=& d\exp_x \left(t{|\D u(x)|^{p-2}}\D u(x)\right) 
\\& \cdot   \left\{  t|\D u(x)|^{p-2}  \left({\bf I}+(p-2) \frac{\D u}{|\D u|}\otimes \frac{\D u}{|\D u|}\right)\circ D^2u(x) +D^2 (d_{\gamma(t)}^2/2)(x)\right\}  ,
\end{split}
\end{equation*} 
where $\gamma(t):=\phi_p(x,t)$  is a minimizing geodesic joining  $x$ to $\phi_p(x,1)\not\in \Cut(x)\cup\{x\}$, and  $D^2u(x)$ is the Hessian    in the sense of    Theorem \ref{thm-AB}; see     also  the proof of \cite[Proposition 4.1]{CMS}.     
The rest of the proof is the same as for   Proposition \ref{prop-jacobi-estimate}. 
\end{proof}


\subsection{Degenerate operators}
This subsection is devoted to the ABP type estimate for the $p$-Laplacian type operators with $2\leq p<\infty.$ 
 \begin{thm}[ABP type estimate]\label{thm-abp-type-p-large}
Assume that        $2\leq p<\infty$   and  $\Ric\geq -(n-1)\kappa$   for $\kappa\geq0$. 
  For a bounded open set  $\Omega\subset M,$ let $u    \in  C\left(\overline\Omega\right)$ be smooth in $\Omega$.     
For a compact set $E\subset M,$ we assume that 
  $$\cA^p_{1}\left(E;\overline \Omega;u\right)\subset \Omega.$$ Then  we have 
\begin{equation*}
|E|\leq \int_{\cA_1^{p}\left(E;\overline \Omega;u\right)}\sS^{n-1}\left(\sqrt{ {\kappa}}{|\D u(x)|}^{p-1} \right)\left\{\frac{\La_pu(x)}{n}+ \sH\left(\sqrt{ {\kappa}}  {|\D u(x)|}^{p-1}\right)\right\}^ndx,
\end{equation*}
where $\sH(\tau):=\tau\coth(\tau),$ and $\sS(\tau):= \sinh(\tau)/\tau$ for   $\tau > 0$ with $\sH(0)=\sS(0)=1.$  In particular, if $\Ric\geq0,$  we have 
\begin{equation*}
|E|\leq \int_{\cA_1^{p}\left(E;\overline \Omega;u\right)} \left\{\frac{\La_pu(x)}{n}+1\right\}^ndx.
\end{equation*}
  \end{thm}
  
\begin{proof}
For any  $x\in\cA^p_1\left(E;\overline \Omega;u\right)\subset\Omega,$ we find $y\in E$ such that 
\begin{equation}\label{eq-inf-contact-pt-p-geq2}
\inf_{  \Omega}\left\{u+\frac{p-1}{p}d_y^{\frac{p}{p-1}}\right\}=u(x)+\frac{p-1}{p}d_y^{\frac{p}{p-1}}(x). 
\end{equation} 
From the   argument as in  the proof of Proposition \ref{prop-jacobi-estimate}, 
we have  
  $$  y=x \iff \D u(x)= 0. $$
 When  $|\D u(x)|=0$ for $x\in\cA^p_1\left(E;\overline \Omega;u\right)\subset\Omega,$    we have  $y=x=\exp_x {|\D u|^{p-2}}{ }\D u(x) $ in \eqref{eq-inf-contact-pt-p-geq2}. 
 If $|\D u(x)|>0$ for $x\in\cA^p_1\left(E;\overline \Omega;u\right)\subset\Omega,$    then Proposition \ref{prop-jacobi-estimate}  says that 
\begin{align*} 
y=\exp_x {|\D u|^{p-2}}{ }\D u(x)\not\in \Cut(x)\cup\{x\},\quad\mbox{and}\quad {|\D u|^{p-2}}\D u(x)\in\cE_x. 
\end{align*}  
Now we define the map  $\Phi_p:\cA^p_{1}\left(E;\overline \Omega;u\right)\to M$  as 
$$\Phi_p(x):=\exp_x {|\D u(x)|^{p-2}}{}\D u(x),$$
which coincides with $\phi_p(x,1)$ in Proposition \ref{prop-jacobi-estimate}.  From the argument above and the definition of $\cA^p_{1}\left(E;\overline \Omega;u\right)$, 
 we observe    that 
\begin{equation}\label{eq-image-phi-contact-set}
E= \Phi_p\left(\cA^p_{1}\left(E;\overline \Omega;u\right)\right).
\end{equation}
since we assume $\cA^p_{1}\left(E;\overline \Omega;u\right)\subset \Omega.$

By using    Proposition \ref{prop-jacobi-estimate} 
together with   Theorem \ref{thm-BG},  the arithmetic-geometric means inequality yields that  for   $ x\in \cA^p_{1}\left(E;\overline \Omega;u\right)\cap \{z\in\Omega :|\D u(z)|>0\},$
\begin{align*}
&\Jac\Phi_p(x)= \Jac \exp_x \left({|\D u(x)|^{p-2}}\D u(x)\right) \\
& \cdot\det \left\{  |\D u|^{p-2}  \left({\bf I}+(p-2) \frac{\D u}{|\D u|}\otimes \frac{\D u}{|\D u|}\right)\circ D^2u(x) +D^2 (d_{y}^2/2)(x)\right\} \\
&= \Jac  \exp_x \left({|\D u(x)|^{p-2}}\D u(x)\right)\cdot \det\left({\bf I}+(p-2) \frac{\D u}{|\D u|}\otimes \frac{\D u}{|\D u|}(x)\right) \\
&\cdot\det \left\{  |\D u|^{p-2}   D^2u(x) +\left({\bf I}+\frac{2-p}{p-1} \frac{\D u}{|\D u|}\otimes \frac{\D u}{|\D u|}\right)\circ D^2 (d_{y}^2/2)(x)\right\}\\
&\leq\sS^{n-1}\left(\sqrt{\kappa} {|\D u|^{p-1}(x)}{ }\right)\left[\frac{1}{n} \tr\left\{ |\D u|^{p-2}  \left({\bf I}+(p-2) \frac{\D u}{|\D u|}\otimes \frac{\D u}{|\D u|}\right)\circ D^2u(x) +D^2 (d_{y}^2/2)(x)\right\} \right]^n,
\end{align*}
where $y=\Phi_p(x) $, and  we used at the last line  the arithmetic-geometric means inequality: $$\det(AB)\leq \left\{\frac{1}{n}\tr\left(AB\right)\right\}^n$$ for symmetric nonnegative  matrices $A $ and $B. $  
 From  Lemma \ref{lem-pucci-ric-dist-sqrd}, it follows that  for any $ x\in \cA^p_{1}\left(E;\overline \Omega;u\right)\cap \{z\in\Omega :|\D u(z)|>0\},$
\begin{align*}
\Jac\Phi_p(x)
&\leq {\sS^{n-1}\left(\sqrt{\kappa} {|\D u(x)|^{p-1}}{ }\right)} \left\{ \frac{{\La_pu}(x)+1+(n-1)\sH\left(\sqrt{\kappa}\, d_{\Phi_p(x) }(x)\right)}{n} \right\}^n\\
&= {\sS^{n-1}\left(\sqrt{\kappa} {|\D u(x)|^{p-1}}{ }\right)}\left\{ \frac{{\La_pu}(x)+1+(n-1)\sH\left(\sqrt{\kappa} {|\D u(x)|^{p-1}}{}\right)}{n} \right\}^n. 
\end{align*}
    Since $\La_p u$ is  continuous   in $\Omega$ for $p\geq 2, $    we deduce that 
   \begin{equation}\label{eq-jac-est-contact-set-p-geq2}
   \Jac\Phi_p\leq\sS^{n-1}\left(\sqrt{\kappa} {|\D u|^{p-1}}{}\right)\left\{ \frac{\La_pu}{n}+\sH\left(\sqrt{\kappa} {|\D u|^{p-1}}{}\right)\right\}^n\quad\mbox{in}\,\,\, \cA^p_{1}\left(E;\overline \Omega;u\right)
   \end{equation}
  since $\sH(\tau)\geq1$ for $\tau\geq0.$ 
  For $p\geq2, $ the map    $\Phi_{p}$ is  of  class $C^1$, and hence   Lipschitz continuous    on a compact set  $\cA^p_{1}\left(E;\overline \Omega;u\right)\subset\Omega$; notice that $|\D u|^{p-2}\D u(x)\in \cE_x$ for any $x\in\cA^p_{1}\left(E;\overline \Omega;u\right) $. 
 Therefore,  the result follows from the area formula 
     with the help of \eqref{eq-image-phi-contact-set}   and  \eqref{eq-jac-est-contact-set-p-geq2}. 
\end{proof}

 
The following ABP type estimate  is concerned with the nonlinear  operators of $p$-Laplacian type. 

\begin{cor}
\label{cor-abp-type-p-large}
Assume that        $2\leq p<\infty$    and  $\cM_{\ld,\Ld}^-(R(e))\geq -(n-1)\kappa$ with  $\kappa\geq0$ for      any unit vector $e\in TM$.  
  For a bounded open set  $\Omega\subset M,$   let $u\in    C\left(\overline\Omega\right)$ be smooth in $\Omega$ such that  
  $|\D u|^{p-2}\cM^-_{\ld,\Ld}(D^2u)\leq f$ in $\Omega.$  
For a compact set $E\subset M,$ we assume that 
  $\cA^p_{1}\left(E;\overline \Omega;u\right)\subset \Omega.$ Then  we have 
\begin{equation*}
|E|\leq \int_{\cA_1^{p}\left(E;\overline \Omega;u\right)}  \frac{(p-1)}{n^n\ld^n}\sS^{n-1}\left(\sqrt{\frac{\kappa}{\ld}} {|\D u|^{p-1}}{ }\right)\left\{f^++\frac{\Ld}{p-1}+(n-1)\Ld\sH\left(\sqrt{\frac{\kappa}{\Ld}} {|\D u|^{p-1}}{}\right) \right\}^ndx.
\end{equation*}
  In particular, if $\Ric\geq0,$  we have 
\begin{equation*}
|E|\leq (p-1)\int_{\cA_1^{p}\left(E;\overline \Omega;u\right)} \left\{\frac{f^+(x)}{n\ld}+\frac{\Ld}{\ld}\right\}^ndx.
\end{equation*}
  \end{cor}
\begin{proof}
Following the proof of Theorem \ref{thm-abp-type-p-large}, it suffices to  estimate 
the Jacobian determinant of $\Phi_p(x):=\exp_x {|\D u(x)|^{p-2}}{}\D u(x)$ on $\cA^p_{1}\left(E;\overline \Omega;u\right)$ in terms of the Pucci operator.    First we note that $\Ric(e,e)\geq \cM^-_{\ld,\Ld}(R(e))/\ld \geq-(n-1)\kappa/\ld $ for any unit vector $e\in TM.$ 
According to Proposition \ref{prop-jacobi-estimate} and Theorem \ref{thm-BG}, we have  that  for any $ x\in \cA^p_{1}\left(E;\overline \Omega;u\right)\cap \{z\in\Omega :|\D u(z)|>0\},$ 
\begin{align*}
&\Jac\Phi_p(x)= \Jac \exp_x \left({|\D u(x)|^{p-2}}\D u(x)\right) \cdot\det   \left({\bf I}+(p-2) \frac{\D u}{|\D u|}\otimes \frac{\D u}{|\D u|}\right)\\
& \cdot\det \left\{  |\D u|^{p-2}  D^2u(x) +\left({\bf I}+\frac{2-p}{p-1} \frac{\D u}{|\D u|}\otimes \frac{\D u}{|\D u|}\right)\circ D^2 (d_{y}^2/2)(x)\right\} \\
&\leq(p-1)\sS^{n-1}\left(\sqrt{\frac{\kappa}{\ld}} {|\D u|^{p-1}}{ }\right)\left[\frac{1}{n} \tr\left\{|\D u|^{p-2} D^2u(x) +  \left({\bf I}+\frac{2-p}{p-1}\frac{\D u}{|\D u|}\otimes \frac{\D u}{|\D u|}\right)\circ D^2 (d_{y}^2/2)(x)\right\} \right]^n\\
&=(p-1)\sS^{n-1}\left(\sqrt{\frac{\kappa}{\ld}} {|\D u|^{p-1}}{ }\right)\left[\frac{1}{n\ld } \cM^-_{\ld,\Ld}\left\{ |\D u|^{p-2}  D^2u(x) + \left({\bf I}+\frac{2-p}{p-1}\frac{\D u}{|\D u|}\otimes \frac{\D u}{|\D u|}\right)\circ D^2 (d_{y}^2/2)(x)\right\} \right]^n\\
&\leq   \frac{(p-1)}{n^n\ld^n}\sS^{n-1}\left(\sqrt{\frac{\kappa}{\ld}} {|\D u|^{p-1}}{ }\right) \left\{ |\D u|^{p-2} \cM^-_{\ld,\Ld}\left( D^2u\right) +\cM^+_{\ld,\Ld}\left(   \left({\bf I}+\frac{2-p}{p-1} \frac{\D u}{|\D u|}\otimes \frac{\D u}{|\D u|}\right)\circ D^2 (d_{y}^2/2)(x)\right) \right\}^n,
\end{align*}
where $y=\Phi_p(x)$. From \eqref{eq-eigenvector-along-gamma}, we notice that  for   $ x\in \cA^p_{1}\left(E;\overline \Omega;u\right)\cap \{z\in\Omega :|\D u(z)|>0\},$ 
$$    \left({\bf I}+\frac{2-p}{p-1} \frac{\D u}{|\D u|}\otimes \frac{\D u}{|\D u|}\right)\circ D^2 d_{y}^2/2(x)   =   D^2 d_{y}^2/2(x)+ \frac{2-p}{p-1}\D d_y(x)\otimes \D d_y(x)$$
    since $\D d_y(x)=- \frac{\D u(x)}{|\D u(x)|}$ is an eigenvector of $D^2d_{y}^2/2(x)$ associated with the eigenvalue $1$ for $y= \Phi_p(x)=\exp_x {|\D u|^{p-2}}{}\D u(x)$.  
By using   Lemma \ref{lem-pucci-ric-dist-sqrd}, we deduce that
$$\cM^+_{\ld,\Ld}\left(   \left({\bf I}+\frac{2-p}{p-1} \frac{\D u}{|\D u|}\otimes \frac{\D u}{|\D u|}\right)\circ D^2 (d_{y}^2/2)(x)\right) \leq\frac{\Ld}{p-1}+(n-1)\Ld\sH\left(\sqrt{\frac{\kappa}{\Ld}} d_y(x)\right),$$
and hence  for any $ x\in \cA^p_{1}\left(E;\overline \Omega;u\right)\cap \{z\in\Omega :|\D u(z)|>0\},$  
\begin{align*}
\Jac\Phi_p(x)&\leq \frac{(p-1)}{n^n\ld^n}\sS^{n-1}\left(\sqrt{\frac{\kappa}{\ld}} {|\D u|^{p-1}}{ }\right)\left\{f^++\frac{\Ld}{p-1}+(n-1)\Ld\sH\left(\sqrt{\frac{\kappa}{\Ld}} {|\D u|^{p-1}}{}\right) \right\}^n.
\end{align*}
   Since $p\geq 2,$  this estimate holds for  any $x\in \cA^p_{1}\left(E;\overline \Omega;u\right),$ and then  the result follows from \eqref{eq-image-phi-contact-set}   and  the area formula.  
\end{proof}

From  Theorem \ref{thm-abp-type-p-large}, we    obtain  the   ABP type estimate  of viscosity  supersolutions for $p$-Laplacian operators   with the help of   regularization by  inf-convolution.

 \begin{lemma}\label{lem-abp-type-measure-p-large}
Assume that     $2\leq p<\infty$ and $\Ric \geq - {(n-1)} \kappa$   for $\kappa\geq0$. 
For  $z_0,x_0\in M$ and $0<r\leq R\leq R_0,$ assume that $  \overline B_{4r}(z_0)\subset B_{R}(x_0).$   Let    
$f\in C\left(     B_{R}(x_0)\right)$   
 and  $u\in C\left( \overline  B_{R}(x_0)\right)$  be   such that   $$\La_pu\leq f \quad\mbox{ in $B_{R}(x_0)$}$$
in the viscosity sense,   
\begin{equation}\label{eq-abp-cond-u-p-geq2}
u\geq0\quad\mbox{on $  B_{R}(x_0) \setminus B_{4r}(z_0) \quad$ 
and $\quad \displaystyle\inf_{B_{r}(z_0)}u\leq \frac{p-1}{p} .$}
\end{equation}
Then 
\begin{equation}\label{eq-abp-type-measure-0-p-geq2}
\begin{split}
|B_r(z_0)|
&\leq  {\sS^{n-1}\left(2\sqrt{\kappa}R_0\right)}\int_{\{u\leq \tilde M_p\}\cap   B_{4r}(z_0)} \left\{\sH\left( 2\sqrt{\kappa}R_0\right) +\frac{r^p f^+}{n}\right\}^n 
\end{split}
\end{equation}
for a uniform constant $\tilde M_p:=\frac{p-1}{p}3^{\frac{p}{p-1}}$.  
  Moreover, if $r^pf\leq 1$ in $B_{4r}(z_0),$ then there exists a uniform constant $\delta\in(0,1)$ depending only on  $n, p$ and $\sqrt{\kappa}R_0,$ such that 
  \begin{equation}\label{eq-abp-type-measure-p-geq2}
  \left|\left\{ u\leq \tilde M_p\right\}\cap B_{4r}(z_0)\right|>\delta|B_{4r}(z_0)|.
  \end{equation}

    \end{lemma}

  \begin{proof}
 (i)  First, we claim that for $u\in C^\infty(B_R(x_0))\cap C\left(\overline B_R(x_0)\right)$ satisfying \eqref{eq-abp-cond-u-p-geq2}, 
  \begin{equation}\label{eq-abp-type-measure-0-p-geq2-smooth}
  |B_r(z_0)|
\leq  {\sS^{n-1}\left(2\sqrt{\kappa}R_0\right)} \int_{\left\{u< \tilde M_p,\,r^p|\D u|^{p-1}< 2R_0\right\}\cap   B_{4r}(z_0)} \left\{\sH\left( 2\sqrt{\kappa}R_0\right) +\frac{r^p \left(\La_pu\right)^+}{n}\right\}^n.
\end{equation}
 According to Theorem \ref{thm-abp-type-p-large}, 
 it suffices to prove  that 
 \begin{equation}\label{eq-abp-type-est-contact-set-level-set-grad}
   \cA^p_{1}\left( \overline B_{r}(z_0); \overline  B_{R}(x_0);r^{\frac{p}{p-1}}u\right)\subset B_{4r}(z_0)\cap \left\{ u< \tilde M_p,\, r^p|\D u|^{p-1}< 2R_0 \right\}
   \end{equation}
   since $\sS$ and $\sH$ are nondecreasing functions in $[0,\infty).$ 
  Indeed,  for a fixed $y\in \overline B_r(z_0)$,  we consider 
$$w_y:=r^{\frac{p}{p-1}}u+ \frac{p-1}{p}d_y^{\frac{p}{p-1}}.$$
Using  \eqref{eq-abp-cond-u-p-geq2}, we have 
\begin{align*}
r^{-\frac{p}{p-1}}w_y&\geq 0+\frac{p-1}{p}3^{\frac{p}{p-1}}\quad\mbox{outside $B_{4r}(z_0)$},\\
\inf_{B_{r}(z_0)}r^{-\frac{p}{p-1}}w_y&\leq  \frac{p-1}{p} +  \frac{p-1}{p}2^{\frac{p}{p-1}} <\frac{p-1}{p}3^{\frac{p}{p-1}}.
\end{align*}
Then we   find  a point $x\in   B_{4r}(z_0)$ such that 
$$\inf_{B_{R}(x_0)} r^{-\frac{p}{p-1}}w_y=  r^{-\frac{p}{p-1}}w_y(x)=u(x)+\frac{p-1}{p}\left(\frac{d_y(x)}{r}\right)^{\frac{p}{p-1}}< \frac{p-1}{p}3^{\frac{p}{p-1}}.$$ 
  Proposition \ref{prop-jacobi-estimate} implies that   $| r^{\frac{p}{p-1}}\D u(x)|^{p-1}=d_y(x)< 8r\leq 2R_0.$ 
Then we deduce  that
  \begin{align*}
   \cA^p_{1}\left( \overline B_{r}(z_0); \overline  B_{4r}(z_0);r^{\frac{p}{p-1}}  u\right)&= \cA^p_{1}\left( \overline B_{r}(z_0); \overline  B_{R}(x_0);r^{\frac{p}{p-1}}  u\right)\\
  & \subset B_{4r}(z_0)\cap \{   u< \tilde M_p,\, r^p|\D   u|^{p-1}< 2R_0 \}
   \end{align*}
  for   $\tilde M_p:= \frac{p-1}{p}3^{\frac{p}{p-1}}>1$. 
      Thus, \eqref{eq-abp-type-measure-0-p-geq2-smooth} follows from Theorem \ref{thm-abp-type-p-large}. 
     
(ii) Now we assume that $u$ is a continuous viscosity supersolution.  Let $\eta>0.$ According to  Lemma  \ref{lem-visc-u-u-e-properties}, and Remark \ref{rmk-p-La-eq-u-e},  the inf-convolution $u_\ve$ of $u$ with respect to $\overline B_{R}(x_0)$ (for small $\ve>0$) satisfies 
\begin{equation*}
\left\{
\begin{split}
&u_\ve \to u\quad\mbox{uniformly in $B_{R}(x_0)$},\\
&u_\ve\geq -\eta\quad\mbox{in $B_{R}(x_0)\setminus B_{4r}(z_0) $,}\qquad\inf_{B_{r}(z_0)}u_\ve\leq \frac{p-1}{p}+\eta,\\
&\La_p u_\ve -\ve\tilde\kappa(n+p-2)|\D u_\ve|^p\leq f_\ve\quad\mbox{a.e. in $B_{4r}(z_0)$},
\end{split}\right. 
\end{equation*}
where $-\tilde\kappa $ ($\tilde\kappa\geq0$) is a lower bound of the sectional curvature on $\overline B_{R}(x_0),$ and 
$$f_\ve(z):= \sup_{B_{2\sqrt{m\ve } }(z)}f^+, \quad\forall z\in B_{4r}(z_0);\quad m:=\|u\|_{L^{\infty}(B_R(x_0))}.$$ 
According to  Lemma \ref{lem-visc-u-u-e-properties}, we approximate $u_\ve$ by  a sequence $\{w_k\}_{k=1}^\infty$ of smooth functions    which satisfy the following:
\begin{equation*}
\left\{
\begin{split}
&w_k\to u_\ve \quad\mbox{uniformly in $B_{4r}(z_0)$},\\
&w_k\geq -2\eta\quad\mbox{in $B_{R}(x_0)\setminus B_{4r}(z_0)$,}\qquad\inf_{B_{r}(z_0)}w_k\leq \frac{p-1}{p}+2\eta,\\
&|\D w_k|\leq C,\,\, D^2w_k\leq C{\bf I} \quad\mbox{uniformly in  $M$ with respect to $k$},\\
& \D w_k\to\D u_\ve,\,\,\,D^2 w_k\to D^2u_\ve\quad\mbox{a.e. in $B_{4r}(z_0)\,\,$ as $k\to\infty,$}
\end{split}\right.
\end{equation*}   
     for a uniform constant  $C>0$  independent of  $k\in\N.$ Then we apply \eqref{eq-abp-type-measure-0-p-geq2-smooth} to  the function $\tilde w_k:=\frac{(p-1)/p }{(p-1)/p+4\eta} ({w_k+2\eta})$ to obtain 
       \begin{equation*}
  |B_r(z_0)|
\leq  {\sS^{n-1}\left(2\sqrt{\kappa}R_0\right)} \int_{\{\tilde w_k< \tilde M_p,\,r^p|\D \tilde w_k|^{p-1}< 2R_0\}\cap   B_{4r}(z_0)} \left\{\sH\left( 2\sqrt{\kappa}R_0\right) +\frac{r^p \left(\La_p\tilde w_k\right)^+}{n}\right\}^n.
\end{equation*}
Letting $k\to\infty$ and $\tilde u_\ve:=\frac{(p-1)/p }{(p-1)/p+4\eta} ({u_\ve+2\eta})$,  we deduce that  
       \begin{align*}
 & |B_r(z_0)|
\leq  {\sS^{n-1}\left(2\sqrt{\kappa}R_0\right)} \int_{\{\tilde u_\ve\leq \tilde M_p,\,r^p|\D \tilde u_\ve|^{p-1}\leq 2R_0\}\cap   B_{4r}(z_0)} \left\{\sH\left( 2\sqrt{\kappa}R_0\right) +\frac{r^p \left(\La_p\tilde u_\ve\right)^+}{n}\right\}^n
\end{align*} 
since $\left(\La_p\tilde w_k\right)^+$ 
 is uniformly bounded  with respect to $k$, and converges to $\left(\La_p\tilde u_\ve\right)^+$  almost everywhere in $B_{4r}(z_0)$ as $k$ tends to $\infty.$
  Since  
\begin{align*}
\left(\La_p\tilde u_\ve\right)^+&=\left(\frac{(p-1)/p}{(p-1)/p+4\eta}\right)^{p-1}\left(\La_p  u_\ve\right)^+\\
&\leq \left(\frac{(p-1)/p}{(p-1)/p+4\eta}\right)^{p-1}\left\{\ve\tilde\kappa(n+p-2)|\D u_\ve|^p+ f_\ve\right\}\\
&\leq\left(\frac{(p-1)/p+4\eta}{(p-1)/p}\right) \ve\tilde\kappa(n+p-2)(r^{-p}2R_0)^{\frac{p}{p-1}}+\left(\frac{(p-1)/p}{(p-1)/p+4\eta}\right)^{p-1}  f_\ve 
\end{align*}
almost everywhere in $ \left\{r^p|\D \tilde u_\ve|^{p-1}\leq 2R_0\right\}\cap   B_{4r}(z_0)$, 
 we  deduce  that 
 \begin{align*}
 & |B_r(z_0)|
\leq  {\sS^{n-1}\left(2\sqrt{\kappa}R_0\right)} \int_{\{  u\leq \tilde M_p\}\cap   B_{4r}(z_0)} \left\{\sH\left( 2\sqrt{\kappa}R_0\right) +\frac{r^p f^+}{n}\right\}^n
\end{align*}  
by letting     $\ve$ and    $\eta$ go  to $0$. 
 Lastly, the above estimate implies 
 \eqref{eq-abp-type-measure-p-geq2}   with the help of   Bishop-Gromov's volume comparison. 
  \end{proof}
  
We also have the  ABP type  estimate for viscosity supersolutions to  nonlinear $p$-Laplacian type equations. 
  
 \begin{cor}\label{cor-abp-type-measure-p-large-nonlinear}
 Assume that  $ 2\leq p<\infty,$ and 
 {$\cM^-_{\ld,\Ld }(R(e))\geq - {(n-1)} \kappa$}   with  $\kappa\geq0$ for any unit vector $e\in TM$.   
For  $z_0,x_0\in M$ and $0<r\leq R\leq R_0,$ assume that $ \overline  B_{4r}(z_0)\subset B_{R}(x_0).$     Let   $\beta\geq0,$  
and  $u \in C(\overline  B_{R}(x_0))$ be such that  $$|\D u|^{p-2}\cM^-_{\ld,\Ld}(D^2u)-  \beta |\D u|^{p-1}\leq { {}r^{-p}} \quad\mbox{ in $B_{R}(x_0)$}$$
in the viscosity sense,     
\begin{equation}\label{eq-abp-cond-u-p-geq2-nonlinear}
u\geq0\quad\mbox{on $  B_{R}(x_0)\setminus B_{4r}(z_0) \quad$ and $\quad \displaystyle\inf_{B_{r}(z_0)}u\leq \frac{p-1}{p} .$}
\end{equation}
Then  there exists a uniform constant $\delta\in(0,1)$ depending only on  $n, p, \sqrt{\kappa}R_0, \ld,\Ld,$ and $\beta R_0,$ such that 
  \begin{equation*} 
  \left|\left\{ u\leq \tilde M_p\right\}\cap B_{4r}(z_0)\right|>\delta|B_{4r}(z_0)|
  \end{equation*}
  for  $\tilde M_p=\frac{p-1}{p}3^{\frac{p}{p-1}}>1.$
\end{cor}
\begin{proof} 
  The  proof is similar to one for Lemma \ref{lem-abp-type-measure-p-large}    replacing Theorem  \ref{thm-abp-type-p-large} 
by  Corollary  \ref{cor-abp-type-p-large}.      We use    the same notation as in  the proof of Lemma \ref{lem-abp-type-measure-p-large},  and notice that  \eqref{eq-abp-type-est-contact-set-level-set-grad} follows from \eqref{eq-abp-cond-u-p-geq2-nonlinear}.  
    As in the proof of Lemma   \ref{lem-abp-type-measure-p-large}, let  $\eta>0,$ and  let $u_\ve$ be   the inf-convolution of $u $ with respect to $\overline B_R(x_0)$ for small $\ve>0$. 
Applying Corollary  \ref{cor-abp-type-p-large}  to  approximating smooth  functions for   $u_\ve$  from  Lemma \ref{lem-visc-u-u-e-properties},  we deduce that    for   sufficiently small $\ve>0,$
    \begin{equation*}
  \begin{split}
& |B_r(z_0)|\leq  \frac{(p-1)}{n^n\ld^n} {\sS^{n-1}\left(2\sqrt{\frac{\kappa}{\ld}}R_0\right)}  
  \\  &\cdot\int_{\left\{\tilde  u_\ve \leq\tilde M_p,\,r^p|\D \tilde u_\ve|^{p-1}\leq 2R_0\right\}\cap   B_{4r}(z_0)} \left\{ \frac{\Ld}{p-1}+(n-1)\Ld\sH\left( 2\sqrt{\frac{\kappa}{\Ld}}R_0\right) + {r^p  |\D \tilde u_\ve|^{p-2}\left(\cM^-_{\ld,\Ld}(D^2 \tilde  u_\ve)\right)^+}\right\}^n, 
  \end{split}
\end{equation*}  
where $\displaystyle \tilde u_\ve:=\frac{(p-1)/p }{(p-1)/p+4\eta} ({u_\ve+2\eta}).$ Since the Pucci operators are intrinsically uniformly continuous with $\omega_{G}\equiv0$,  Lemma \ref{prop-u-u-e-superj} implies that 
 \begin{align*}
|\D \tilde u_\ve|^{p-2}\left(\cM^-_{\ld,\Ld}(D^2\tilde u_\ve)\right)^+&=\left(\frac{(p-1)/p}{(p-1)/p+4\eta}\right)^{p-1}|\D   u_\ve|^{p-2}\left(\cM^-_{\ld,\Ld}(D^2  u_\ve)\right)^+\\
&\leq \left(\frac{(p-1)/p}{(p-1)/p+4\eta}\right)^{p-1}\left\{\ve \tilde\kappa n\Ld |\D u_\ve|^p+ \beta|\D u_\ve|^{p-1} + r^{-p}\right\}\\
&\leq \frac{(p-1)/p+4\eta}{(p-1)/p}  \ve\tilde\kappa n\Ld(r^{-p}2R_0)^{\frac{p}{p-1}}+  2\beta R_0r^{-p} +\left(\frac{(p-1)/p}{(p-1)/p+4\eta}\right)^{p-1}r^{-p}
\end{align*}
almost everywhere in $ \left\{r^p|\D \tilde u_\ve|^{p-1}\leq 2R_0\right\}\cap   B_{4r}(z_0)$.  
Therefore,  letting $\ve$ and $\eta$ go to $0$,  we obtain   
  \begin{equation*}
  \begin{split}
  |B_r(z_0)|&\leq  \frac{p-1}{n^n\ld^n} {\sS^{n-1}\left(2\sqrt{\frac{\kappa}{\ld}}R_0\right)}  
   \int_{\left\{   u\leq  \tilde M_p\right\}\cap   B_{4r}(z_0)} \left\{ \frac{\Ld}{p-1}+(n-1)\Ld\sH\left( 2\sqrt{\frac{\kappa}{\Ld}}R_0\right) + {2\beta R_0+1}\right\}^n,
  \end{split}
\end{equation*}   
from which    the   result follows   by using  Bishop-Gromov's volume comparison. 
\end{proof}


\subsection{Singular operators}

Now we consider  the $p$-Laplacian type  operator for $1<p<2.$
In this case, the operator    $\La_p u$ 
becomes  singular when its gradient  vanishes.  
To   deal with  singularities, we make use of a regularized operator 
$$\left(|\D u(x)|^2+\delta\right)^{\frac{p-2}{2}}\cM^-_{p-1,1}(D^2u)$$
for any $  \delta>0$; see Lemma \ref{lem-pucci-delta-p-laplace-delta}.
So we first show the ABP type estimate for nonlinear $p$-Laplacian type operators. 

 \begin{thm}[ABP type estimate]\label{thm-abp-type-p-1}
Assume that          $1<p<2$   and  $\cM^-_{\ld,\Ld}(R(e))\geq - {(n-1)}\kappa$   with  $\kappa\geq0$ for any unit vector $e\in TM$.  
  For a bounded open set  $\Omega\subset M,$  
let  $f\in C(\Omega),$ and     $u\in C\left(\overline\Omega\right)$  be smooth in $\Omega$   
such that   $ |\D u|^{p-2}\cM^-_{\ld,\Ld}(D^2u)\leq f$ in $\Omega $ in the viscosity sense.      
For   a compact set $E\subset M,$ we assume that 
  $$\cA^p_{1}\left(E;\overline \Omega;u\right)\subset \Omega.$$ Then  we have 
\begin{equation*}
|E|\leq \int_{\cA_1^{p}\left(E;\overline \Omega;u\right)}\sS^{n-1}\left(\sqrt{ \frac{\kappa}{\lambda}}   {|\D u|^{p-1}}{} \right)
\left[\frac{1}{n\ld}\left\{
   {f^+}  +\frac{\Ld}{p-1}+  {(n-1)}\Ld  \sH\left(\sqrt{ \frac{\kappa}{\Ld}}|\D u|^{p-1}\right)  \right\}\right]^n.
\end{equation*}
  \end{thm}
  
\begin{proof}
 From Definition \ref{def-visc-sol}, it is not difficult to see that     $u$ is   a  viscosity supersolution to  
 \begin{equation}\label{eq-approx-op-p-less2}
 \left( |\D u|^2+\delta\right)^{\frac{p-2}{2}}\cM^-_{\ld,\Ld}(D^2u)\leq f^+\quad\mbox{ in $\Omega $ }
 \end{equation}
 for any $\delta>0$   since $1<p<2$; refer to the proof of \cite[Lemma 2]{ACP}. 
 Moreover, $u$ is a viscosity supersolution to  \eqref{eq-approx-op-p-less2} in the usual sense       according to  Lemma \ref{lem-visc-sol-singular-usual} since  the regularized  operator (for a given $\delta>0$) is      continuous   with respect  to $\D u$ and $D^2u$.   Thus, we deduce that  $u$ solves \eqref{eq-approx-op-p-less2} in the classical sense.

 For $\delta>0,$   define  a regularized  map $\Phi_{p,\delta}:\cA^p_{1}\left(E;\overline \Omega;u\right)\to M$   as 
$$\Phi_{p,\delta}(x):=\exp_x\left( |\D u(x)|^2+\delta\right)^{\frac{p-2}{2}}\D u(x).$$
Note that   if $|\D u(x)|>0,$ then $$\Phi_{p,\delta}(x)=\phi_p\left(x, t_{\delta}(x)\right)\quad\mbox{for}\,\, t_\delta(x):=\left(\frac{ |\D u(x)|^2}{|\D u(x)|^2+ \delta}\right)^{\frac{2-p}{2}}\in(0,1),$$ where 
$\phi_{p}(x,t):=\exp_x t{|\D u(x)|^{p-2}}{ }\D u(x)$ for $ t\in[0,1]$ as in Proposition \ref{prop-jacobi-estimate}.  
We also define $\Phi_{p,0}:   \cA^p_1\left(E;\overline \Omega;u\right)\cap \{|\D u|>0\} \to M $ 
 by $$\Phi_{p,0}(x):=\exp_x|\D u(x)|^{p-2}\D u(x).$$ 
 
 For any $x\in\cA^p_{1}\left(E;\overline \Omega;u\right)\subset \Omega,  $ there is $y\in E$ 
 such that 
\begin{equation}\label{eq-inf-contact-pt-p<2}
\inf_{  \Omega}\left\{u+\frac{p-1}{p}d_y^{\frac{p}{p-1}}\right\}=u(x)+\frac{p-1}{p}d_y^{\frac{p}{p-1}}(x).
\end{equation} 
Since $\cA^p_{1}\left(E;\overline \Omega;u\right)\subset \Omega$,     the argument in the proof of Proposition \ref{prop-jacobi-estimate}   asserts  that 
for  each $x\in\cA^p_{1}\left(E;\overline \Omega;u\right),$ a point $y\in E$ satisfying \eqref{eq-inf-contact-pt-p<2} is not a cut point of $x.$ In fact, $y$ is uniquely determined, and the map $\cA^p_1\left(E;\overline \Omega;u\right)\owns x\mapsto y\in E$ is well-defined as 
\begin{equation*}
y=\left\{
\begin{split}
 &x,\quad &\mbox{for}\,\,\,|\D u(x)|=0,\\
&\exp_x {|\D u(x)|^{p-2}}{ }\D u(x)=\Phi_{p,0}(x)\not\in \Cut(x)\cup\{x\},\quad &\mbox{for}\,\,\,|\D u(x)|>0,
\end{split}
\right.
\end{equation*}
which is  surjective. 
Moreover,  when  $|\D u(x)|=0$ for $x\in\cA^p_1\left(E;\overline \Omega;u\right),$    we have  $y=x=\Phi_{p,\delta}(x)  \not\in\Cut(x)$ in \eqref{eq-inf-contact-pt-p<2}  for any $\delta>0,$ and   $D^2u(x)\geq- \frac{p-1}{p}D^2d^{\frac{p}{p-1}}_x(x)=0 $  
  since $\frac{p}{p-1}>2.$  
 If $|\D u(x)|>0$ for $x\in\cA^p_1\left(E;\overline \Omega;u\right),$    then 
\begin{align*} 
 y=\exp_x {|\D u(x)|^{p-2}}{ }\D u(x)=\Phi_{p,0}(x)=\lim_{\delta\to0}\Phi_{p,\delta}(x).
\end{align*} 
 In the latter case,  the curve  
 $\gamma(t):=\phi_p(x,t)=\exp_x t\,{|\D u|^{p-2}}{ }\D u(x)$ is a unique minimizing geodesic joining $\gamma(0)=x$ to $\gamma(1)=\exp_x {|\D u|^{p-2}}\D u(x)\not\in \Cut(x)\cup\{x\}$  with velocity $\dot\gamma(0)={|\D u|^{p-2}}\D u(x)\in \cE_x\setminus\{0\}$, and 
 $\gamma(t_\delta(x))=\Phi_{p,\delta}(x)\not\in \Cut(x)$ for $\delta>0$.  
 
 Notice that for each $\delta>0$,   the regularized map   $\Phi_{p,\delta}$  is Lipschitz continuous on a compact set  $\cA^p_1\left(E;\overline \Omega;u\right)\subset \Omega$   since    
 \begin{equation}\label{eq-Phi-delta-lip-lem-jac}
 \left( |\D u|^2+\delta\right)^{\frac{p-2}{2}}\D u(x)\in\cE_x,\quad\forall       x\in \cA_1^ {p}\left(E;\overline \Omega;u\right) \quad\mbox{ for any $\delta>0$}.
 \end{equation}

Now  we    claim   that 
\begin{equation}\label{eq-E-contatined-image-phi}
|E| 
\leq   \limsup_{\delta\to0} \int_{\cA_1^ {p}\left(E;\overline \Omega;u\right)} \Jac\Phi_{p,\delta}(x)dx.
\end{equation}
 From the argument above,  it follows  that  
\begin{equation*}
\begin{split}
E&=\Phi_{p,0}\left(  \cA^p_1\left(E;\overline \Omega;u\right)\cap \{|\D u|>0\}\right)\,   \bigcup \,\left\{  \cA^p_1\left(E;\overline \Omega;u\right)\cap \{|\D u|=0\}\right\}\\
&=\Phi_{p,0}\left(    \cA^p_1\left(E;\overline \Omega;u\right)\cap \{|\D u|>0\} \right)\, \bigcup   \,\Phi_{p,\delta}\left(  \cA^p_1\left(E;\overline \Omega;u\right)\cap \{|\D u|=0\}\right)\quad\forall \delta>0,
\end{split}
\end{equation*}
   where we note  $\left.\Phi_{p,\delta}\right|_{\cA^p_1\left(E;\overline \Omega;u\right)\cap \{|\D u|=0\}}$ is the identity for $\delta>0$. 
Letting  $A_k:= \cA^p_1\left(E;\overline \Omega;u\right) \cap\left\{|\D u|\geq {1}/{k}\right\}$ for $   k\in\N,$   we have 
\begin{equation}\label{eq-split-E}
\begin{split}
E
&=\Phi_{p,0}\left(   \bigcup_{k=1}^{\infty}  A_k \right)\, \bigcup   \,\Phi_{p,\delta}\left(  \cA^p_1\left(E;\overline \Omega;u\right)\cap \{|\D u|=0\}\right) \quad\forall \delta>0.
\end{split}
\end{equation}
Notice that $\left|\Phi_{p,0}\left(  \bigcup_{k=1}^{\infty}  A_k \right)\right|= \displaystyle\lim_{k\to\infty} \left|\Phi_{p,0}\left(A_k\right)\right|$, and 
  the map $\left. \Phi_{p,0}\right|_{A_k}$  is    Lipschitz continuous 
for each  $k\in\N$ since ${|\D u|^{p-2}}\D u(x)\in \cE_x\setminus \{ 0\}$ for   $x\in  \cA^p_1\left(E;\overline \Omega;u\right)\cap \{|\D u|>0\}.$ 
  We use the area formula to obtain  that for $k\in \N,$
$$\left|\Phi_{p,0}\left(A_k\right)\right|\leq\int_{A_k} \Jac\Phi_{p,0}(x)dx.
$$
Lemma \ref{lem-jacobian-exp} yields that   for $x\in\cA^p_1\left(E;\overline \Omega;u\right)\cap\{|\D u|>0\},$ and  for   $\tilde \delta\geq0,$ 
\begin{align*}
&\Jac\Phi_{p, \tilde\delta}(x)=\Jac  \exp_x \left(\left(|\D u|^2+ \tilde\delta\right)^{\frac{p-2}{2}}\D u(x)\right) \\
&  \cdot \left|\det\left\{\left(|\D u|^2+ \tilde\delta\right)^{\frac{p-2}{2}}\left({\bf I}+(p-2) \frac{\D u}{\sqrt{|\D u|^2+ \tilde\delta}}\otimes \frac{\D u}{\sqrt{|\D u|^2+ \tilde\delta}}\right)\circ D^2u(x) +D^2 (d_{ \Phi_{p, \tilde\delta}(x)}^2/2)(x)\right\}\right|, 
  \end{align*}
 and hence  we use Fatou's lemma to have  
 $$\left|\Phi_{p,0}\left(A_k\right)\right|\leq\int_{A_k} \Jac\Phi_{p,0}\leq \liminf_{\delta\to 0}\int_{A_k} \Jac\Phi_{p,\delta}\leq \liminf_{\delta\to 0}\int_{\cA^p_1\left(E;\overline \Omega;u\right)\cap \{|\D u|>0\}} \Jac\Phi_{p,\delta}.
 $$
 Thus  
  it follows from \eqref{eq-split-E} and  the area formula   that 
\begin{align*}
|E|&\leq   \limsup_{\delta\to0}\left\{  \int_{\cA^p_1\left(E;\overline \Omega;u\right)\cap \{|\D u|>0\}} \Jac\Phi_{p,\delta}(x)dx 
+\left|\Phi_{p,\delta}\left(  \cA^p_1\left(E;\overline \Omega;u\right)\cap \{|\D u|=0\}\right)\right|\right\} \\
 &\leq   \limsup_{\delta\to0} \int_{\cA_1^ {p}\left(E;\overline \Omega;u\right)} \Jac\Phi_{p,\delta}(x)dx
 \end{align*}
 since $\Phi_{p,\delta}$ is Lipschitz continuous on    $\cA^p_1\left(E;\overline \Omega;u\right)$ for each $\delta>0.$ This   
finishes  the proof of \eqref{eq-E-contatined-image-phi}. 

Now we will  prove a uniform  estimate for  Jacobian determinant of $\Phi_{p,\delta}$ on $\cA_1^{p}\left(E;\overline \Omega;u\right)$ withe respect to   $\delta>0.$     
From   \eqref{eq-Phi-delta-lip-lem-jac} and  Lemma \ref{lem-jacobian-exp}, we have  that
  for $x\in  \cA_1^ {p}\left(E;\overline \Omega;u\right)$, 
 \begin{equation*}
 \begin{split}
d&\Phi_{p,\delta}(x)= d \exp_x \left(\left(|\D u|^2+\delta\right)^{\frac{p-2}{2}}\D u(x)\right) \\
& \circ \left\{\left(|\D u|^2+\delta\right)^{\frac{p-2}{2}}\left({\bf I}+(p-2) \frac{\D u}{\sqrt{|\D u|^2+\delta}}\otimes \frac{\D u}{\sqrt{|\D u|^2+\delta}}\right)\circ D^2u(x) +D^2 (d_{ \Phi_{p,\delta}(x)}^2/2)(x)\right\}\\
&= d \exp_x \left(\left(|\D u|^2+\delta\right)^{\frac{p-2}{2}}\D u(x)\right) \circ\left({\bf I}+(p-2) \frac{\D u}{\sqrt{|\D u|^2+\delta}}\otimes \frac{\D u}{\sqrt{|\D u|^2+\delta}}\right) \\
&\circ\left\{ \left(|\D u|^2+\delta\right)^{\frac{p-2}{2}}  D^2u(x) +\left({\bf I}+\chi_{\{|\D u|>0\}}\frac{(2-p)|\D u|^2}{(p-1)|\D u|^2+\delta} \frac{\D u}{|\D u|}\otimes \frac{\D u}{|\D u|}\right)\circ D^2 (d_{ \Phi_{p,\delta}(x)}^2/2)(x)\right\}.
  \end{split}
  \end{equation*}
   For $ x\in \cA_1^{p}\left(E;\overline \Omega;u\right),$  we define   the     endomorphism  $H_\delta(x)$     on $T_xM$ by 
  $$ H_\delta(x):=\left(|\D u|^2+\delta\right)^{\frac{p-2}{2}}  D^2u(x) +\left({\bf I}+\chi_{\{|\D u|>0\}}\frac{2-p }{p-1 } \frac{\D u}{|\D u|}\otimes \frac{\D u}{|\D u|}\right)\circ D^2 (d_{ \Phi_{p,\delta}(x)}^2/2)(x).$$
  We  prove that   $H_\delta(x)$ is   symmetric,  and positive semi-definite   for   $x\in \cA_1^{p}\left(E;\overline \Omega;u\right)   $. In fact, 
for  $x\in \cA_1^{p}\left(E;\overline \Omega;u\right)\cap\{z\in\Omega    :|\D u(z)|>0\},$  it follows     from (e) of    
Proposition \ref{prop-jacobi-estimate} since  $\Phi_{p,\delta}(x)=\phi_p\left(x, t_{\delta}(x)\right)$ with $t_\delta(x)=\left(\frac{ |\D u(x)|^2+\delta}{|\D u(x)|^2}\right)^{\frac{p-2}{2}}\in(0,1)$.    For $x\in \cA_1^{p}\left(E;\overline \Omega;u\right)\cap\{z\in\Omega   :|\D u(z)|=0\}$, $H_\delta(x) $ is   symmetric and  positive  definite   since   
 $D^2u(x)\geq- \frac{p-1}{p}D^2d^{\frac{p}{p-1}}_x(x)=0,$ and 
 $D^2 (d_{x}^2/2)(x)={\bf I} .$

     From  \eqref{eq-eigenvector-along-gamma}, recall  that  if $|\D u(x)|>0$ for $x\in  \cA_1^{p}\left(E;\overline \Omega;u\right)   ,$ then 
 \begin{equation*}
    \frac{\D u}{|\D u|}\otimes \frac{\D u}{|\D u|} \circ  D^2(d^2_{   \Phi_{p,\delta}(x)}/2)(x)=\frac{\D u}{|\D u|}\otimes \frac{\D u}{|\D u|}(x). 
\end{equation*} 
Thus we have 
 $$ H_\delta(x)=\left(|\D u|^2+\delta\right)^{\frac{p-2}{2}}  D^2u(x) +  D^2 (d_{ \Phi_{p,\delta}(x)}^2/2)(x) +\chi_{\{|\D u|>0\}}\frac{2-p }{p-1 }\, \frac{\D u}{|\D u|}\otimes \frac{\D u}{|\D u|}(x),$$
  which is symmetric and positive semi-definite. 
 Now we   claim  that for $x\in  \cA_1^{p}\left(E;\overline \Omega;u\right)   ,$
  \begin{equation}\label{eq-unif-bd-det-H-delta}
  \begin{split}
\det H_{\delta}&\leq  \left( \frac{1}{n}\tr H_\delta\right)^n   \\
  &\leq \left[\frac{1}{n\ld}\left\{
   {f^+}  +   {\Ld} +  {(n-1)}\Ld  \sH\left(\sqrt{\frac{\kappa}{\Ld}}|\D u|^{p-1}\right) 
 +\chi_{\{|\D u|>0\}}\frac{2-p }{p-1} \Ld\right\}\right]^n=:h(x).
 \end{split}
  \end{equation}
Since $H_\delta$ is symmetric and  positive semi-definite,   the arithmetic-geometric means inequality  combined with \eqref{eq-approx-op-p-less2} 
implies that for $x\in  \cA_1^{p}\left(E;\overline \Omega;u\right)   ,$

\begin{align*}
\left(\det H_{\delta}(x)\right)^{\frac{1}{n}}&\leq\frac{1}{n}\tr H_\delta(x)=\frac{1}{n\ld}\cM^-_{\ld,\Ld} (H_\delta(x))\\
 &\leq \frac{1}{n\ld} \left\{\left(|\D u(x)|^2+\delta\right)^{\frac{p-2}{2}}  \cM^-_{\ld,\Ld}\left(D^2u(x)\right)+  \cM^+_{\ld,\Ld }\left(D^2 (d_{  \Phi_{p,\delta}(x)}^2/2)(x)\right)\right\}
 \\&+\frac{1}{n\ld}\chi_{\{|\D u|>0\}}\frac{2-p }{p-1 } \cM^+_{\ld,\Ld}\left(\frac{\D u}{|\D u|}\otimes \frac{\D u}{|\D u|}(x)\right)\\
  &\leq \frac{1}{n\ld} \left\{f^+(x)+  \cM^+_{\ld,\Ld}\left(D^2 (d_{  \Phi_{p,\delta}(x)}^2/2)(x)\right)+ 
\chi_{\{|\D u|>0\}}\frac{2-p}{p-1 } \Ld \right\}.
\end{align*}
According to Lemma \ref{lem-pucci-ric-dist-sqrd}, 
we have 
\begin{align*}
\cM^+_{\ld,\Ld}\left(D^2 (d_{   \Phi_{p,\delta}(x)}^2/2)(x)\right)
&\leq\Ld+(n-1)\Ld \sH\left(\sqrt{  \frac{\kappa}{\Ld}}d\left(x,   \Phi_{p,\delta}(x)\right) \right)\\
&\leq\Ld+(n-1)\Ld \sH\left(\sqrt{  \frac{\kappa}{\Ld}}|\D u(x)|^{p-1}\right)
\end{align*}
since $\sH(\tau)$ is nondecreasing for $\tau\geq0$. This  proves \eqref{eq-unif-bd-det-H-delta}. 
 

In terms of 
the operator $H_\delta$,  we  have that  for $x\in  \cA_1^{p}\left(E;\overline \Omega;u\right)   ,$
  \begin{align*}
d\Phi_{p,\delta}(x)
&= d \exp_x \left(\left(|\D u|^2+\delta\right)^{\frac{p-2}{2}}\D u(x)\right)  \circ\left({\bf I}+(p-2) \frac{\D u}{\sqrt{|\D u|^2+\delta}}\otimes \frac{\D u}{\sqrt{|\D u|^2+\delta}}\right) \\
&\circ\left\{H_\delta(x) -\chi_{\{|\D u|>0\}}  \frac{\delta(2-p)}{(p-1)\{(p-1)|\D u|^2+\delta\}}  \frac{\D u}{|\D u|}\otimes \frac{\D u}{|\D u|} (x)
\right\},
  \end{align*}
  where we used again that  $\frac{\D u}{|\D u|}\otimes \frac{\D u}{|\D u|} \circ  D^2(d^2_{   \Phi_{p,\delta}(x)}/2)(x)=\frac{\D u}{|\D u|}\otimes \frac{\D u}{|\D u|}(x) $ for $x\in \cA_1^{p}\left(E;\overline \Omega;u\right)\cap\{z\in\Omega:|\D u(z)|>0\}$.  
 We notice that 
  \begin{equation}\label{eq-jac-det-phi-du-0}
|\det d\Phi_{p,\delta}(x)|=\det H_\delta(x),\quad\forall x\in \cA_1^{p}\left(E;\overline \Omega;u\right)\cap\{z\in\Omega   : |\D u(z)|=0\}.
\end{equation}

Now let    $x\in \cA_1^{p}\left(E;\overline \Omega;u\right)\cap\{z\in\Omega   : |\D u(z)|>0\}$, and
let $\{e_1,\cdots,e_n\}$ be an orthonormal basis of $T_xM$ with $e_1=
\frac{\D u(x)}{|\D u(x)|}$. 
Denote  
$$H_{\delta,ij}:=\left\langle  H_\delta(x)e_j,e_i\right\rangle,$$
and let $ \tilde H_{\delta} $ be the $(n-1) \times (n-1)$-matrix that results from $\left(H_{\delta,ij}\right)$ by removing the $1$-st row and the $1$-st column, 
which is  symmetric and positive semi-definite. Then we have 
\begin{equation*}
\begin{split}
\Jac\Phi_{p,\delta}(x)&=\Jac \exp_x \left(\left(|\D u|^2+\delta\right)^{\frac{p-2}{2}}\D u(x)\right)\cdot  \det\left({\bf I}+(p-2) \frac{\D u}{\sqrt{|\D u|^2+\delta}}\otimes \frac{\D u}{\sqrt{|\D u|^2+\delta}}\right)  \\
&\cdot \left|\det\left(H_{\delta,ij}-\chi_{\{|\D u|>0\}}  \frac{\delta(2-p)}{(p-1)\{(p-1)|\D u|^2+\delta\}}\delta_{1i}\delta_{1j}\right)\right|\\
&=\Jac \exp_x \left(\left(|\D u|^2+\delta\right)^{\frac{p-2}{2}}\D u(x)\right)\cdot  \det\left({\bf I}+(p-2) \frac{\D u}{\sqrt{|\D u|^2+\delta}}\otimes \frac{\D u}{\sqrt{|\D u|^2+\delta}}\right)  \\
&\cdot \left|\det H_{\delta}-  \frac{\chi_{\{|\D u|>0\}}\,\delta(2-p)}{(p-1)\{(p-1)|\D u|^2+\delta\}} \det \tilde H_\delta\right|\\
&\leq \Jac \exp_x \left(\left(|\D u|^2+\delta\right)^{\frac{p-2}{2}}\D u(x)\right) \cdot  \det\left({\bf I}+(p-2) \frac{\D u}{\sqrt{|\D u|^2+\delta}}\otimes \frac{\D u}{\sqrt{|\D u|^2+\delta}}\right)  \\
&\cdot    \left\{\det H_{\delta}+ \frac{\chi_{\{|\D u|>0\}}\, \delta(2-p)}{(p-1)\{(p-1)|\D u|^2+\delta\}} \det \tilde H_\delta\right\}. 
\end{split}
\end{equation*}
Since 
 $\left(1\cdot \det \tilde H_{\delta}\right)^{\frac{1}{n}}\leq  \frac{1}{n}\left(1+\tr \tilde H_\delta\right) \leq  \frac{1}{n}\left(1+\tr   H_\delta\right),
 $
it follows from \eqref{eq-unif-bd-det-H-delta}   that 
$$\det H_{\delta}(x)\leq h(x),\,\,\,\mbox{and}\,\,\,\det \tilde H_{\delta} (x)\leq\left\{ \frac{1}{n}+\frac{1}{n}\tr   H_\delta (x)\right\}^n\leq  \left\{\frac{1}{n}+ h(x)^{\frac{1}{n}}\right\}^n,\quad\forall \delta>0,$$ and hence 
\begin{equation}\label{eq-jac-Phi-delta-det-H-delta}
\begin{split}
\Jac\Phi_{p,\delta}(x)
&\leq \Jac \exp_x \left(\left(|\D u|^2+\delta\right)^{\frac{p-2}{2}}\D u(x)\right)  \cdot  \det\left({\bf I}+(p-2) \frac{\D u}{\sqrt{|\D u|^2+\delta}}\otimes \frac{\D u}{\sqrt{|\D u|^2+\delta}}\right)  \\
&\cdot   \left\{\det H_{\delta}(x)+ \frac{\chi_{\{|\D u|>0\}}\, \delta(2-p)}{(p-1)\{(p-1)|\D u|^2+\delta\}} \left( \frac{1}{n} +\frac{1}{n}\tr   H_\delta(x)\right)^n\right\}\\
&\leq \Jac \exp_x \left(\left(|\D u|^2+\delta\right)^{\frac{p-2}{2}}\D u(x)\right) 
 \left\{h(x)+\frac{  \chi_{\{|\D u|>0\}} \,\,\delta\,(2-p)}{(p-1)\{(p-1)|\D u|^2+\delta\}}\left(\frac{1}{n}+ h(x)^{\frac{1}{n}}\right)^n\right\}.
\end{split}
\end{equation}

Using \eqref{eq-unif-bd-det-H-delta}, \eqref{eq-jac-det-phi-du-0} and Theorem \ref{thm-BG},   we deduce that  for $ x\in \cA_1^{p}\left(E;\overline \Omega;u\right)   ,$
  \begin{align*}
\Jac \Phi_{p,\delta}(x)
&\leq\sS^{n-1}\left( \sqrt{\frac{\kappa}{\lambda}} |\D u(x)|^{p-1}\right) \left\{h(x)+\frac{  \chi_{\{|\D u|>0\}} \,\,\delta\,(2-p)}{(p-1)\{(p-1)|\D u|^2+\delta\}}\left(\frac{1}{n}+ h(x)^{\frac{1}{n}}\right)^n\right\}
\end{align*}
 since  $\Ric(e,e)\geq \cM^-_{\ld,\Ld}(R(e))/\ld \geq-(n-1)\kappa/\ld $ for any unit vector $e\in TM$, and  $\sS(\tau)$ is nondecreasing for $\tau\geq0$. 
  Notice that  the last term  $ \chi_{\{|\D u|>0\}}  \frac{\,\delta\, (2-p)}{(p-1)\{(p-1)|\D u|^2+\delta\}}\left(\frac{1}{n}+ h(x)^{\frac{1}{n}}\right)^n
$  is uniformly bounded with respect to $\delta>0$, and converges pointwise to $0$ for any $x\in \cA_1^{p}\left(E;\overline \Omega;u\right)   $ as $\delta$ tends to $0.$  Therefore,     the dominated convergence theorem asserts that 
 \begin{equation*}
 \begin{split}
&   \limsup_{\delta\to0} \int_{\cA_1^ {p}\left(E;\overline \Omega;u\right)} \Jac\Phi_{p,\delta}(x)dx\\
&\leq  \int_{\cA_1^ {p}\left(E;\overline \Omega;u\right)}  \sS^{n-1}\left( \sqrt{\frac{\kappa}{\lambda}} |\D u(x)|^{p-1}\right)
 \cdot  h(x)dx\\
 &  \leq   \int_{\cA_1^ {p}\left(E;\overline \Omega;u\right)} \sS^{n-1}\left(\sqrt{ \frac{\kappa}{\ld}}   {|\D u|^{p-1}}{} \right)\left[\frac{1}{n\ld}\left\{
   {f^+} +\frac{\Ld}{p-1} +  {(n-1)}\Ld  \sH\left(\sqrt{ \frac{\kappa}{\Ld}}|\D u|^{p-1}\right)  \right\}\right]^ndx, 
\end{split}
\end{equation*}
 which finishes the proof together with 
 \eqref{eq-E-contatined-image-phi} 
  \end{proof}

Now we show that Theorem \ref{thm-abp-type-p-1} holds true for semiconcave functions.  

 \begin{lemma}
 \label{lem-abp-type-p-1-nonlinear}
Assume that          $1<p<2$   and  $\cM^-_{\ld,\Ld}(R(e))\geq - {(n-1)}\kappa$   with  $\kappa\geq0$ for any unit vector $e\in TM$.  
 For a bounded open set  $\Omega\subset M,$  
let  $f\in C(\Omega)$ and   $u\in C\left(\overline\Omega\right)$  be   $C_0$-semiconcave in $\Omega$  for  $C_0>0$ 
such that   $ |\D u|^{p-2}\cM^-_{\ld,\Ld}(D^2u)\leq f$ in $\Omega $ in the viscosity sense.      
For   a compact set $E\subset M,$ we assume that 
  $$\cA^p_{1}\left(E;\overline \Omega;u\right)\subset \Omega.$$ Then  we have 
\begin{equation*}
|E|\leq \int_{\cA_1^{p}\left(E;\overline \Omega;u\right)}\sS^{n-1}\left(\sqrt{ \frac{\kappa}{\lambda}}   {|\D u|^{p-1}}{} \right)
\left[\frac{1}{n\ld}\left\{
   {f^+}  +\frac{\Ld}{p-1}+  {(n-1)}\Ld  \sH\left(\sqrt{ \frac{\kappa}{\Ld}}|\D u|^{p-1}\right)  \right\}\right]^n.
\end{equation*}
  \end{lemma}

\begin{proof} 
First,  we    prove   that $u$ is differentiable on $\cA^p_{1}\left(E;\overline \Omega;u\right)\subset \Omega$.   
It suffices to prove subdifferentiability of $u$ on $\cA^p_{1}\left(E;\overline \Omega;u\right)$ since   $u$ is superdifferentiable  in $\Omega$ from semiconcavity.  
Indeed, for     any  $x\in\cA^p_1\left(E;\overline \Omega;u\right),$ there is  $y\in E$ such that 
\begin{equation}\label{eq-inf-contact-pt-p<2-semiconcave}
\inf_{  \Omega}\left\{u+\frac{p-1}{p}d_y^{\frac{p}{p-1}}\right\}=u(x)+\frac{p-1}{p}d_y^{\frac{p}{p-1}}(x).
\end{equation} 
As seen in  the proofs  of Corollary \ref{cor-jacobi-estimate} and  Proposition \ref{prop-jacobi-estimate},  $x$ is not a cut point of $y$  since $u $  is semiconcave  in $ \Omega$.  Since $d_y^\frac{p}{p-1}$  is of class  $C^2$ on $M\setminus\Cut(y)$ for $1<p<2,$   it follows from \eqref{eq-inf-contact-pt-p<2-semiconcave}   that $u$ is subdifferential at    $x\in\cA^p_{1}\left(E;\overline \Omega;u\right). $  Hence 
   $u$ is differentiable at any $x\in\cA^p_{1}\left(E;\overline \Omega;u\right) \subset\Omega.$   
  Notice that $u$ is locally Lipschitz continuous in $\Omega$ from the semi-concavity, and then $|\D u|$ is uniformly bounded in a compact set $\cA^p_{1}\left(E;\overline \Omega;u\right)\subset\Omega.$

      As in the proof of Theorem \ref{thm-abp-type-p-1},  for $\delta>0,$  define  a regularized  map $\Phi_{p,\delta}:\cA^p_{1}\left(E;\overline \Omega;u\right)\to M$   as 
$$\Phi_{p,\delta}(x):=\exp_x\left( |\D u(x)|^2+\delta\right)^{\frac{p-2}{2}}\D u(x).$$
We also define $\Phi_{p,0}:   \cA^p_1\left(E;\overline \Omega;u\right)\cap \{|\D u|>0\} \to M $ 
 by $$\Phi_{p,0}(x):=\exp_x|\D u(x)|^{p-2}\D u(x).$$ 
 Since  $u$ is differentiable on $\cA^p_{1}\left(E;\overline \Omega;u\right)\subset \Omega$,     a similar  argument to  the proofs of Proposition \ref{prop-jacobi-estimate} (a),  and Theorem \ref{thm-abp-type-p-1} yields   that 
for  each $x\in\cA^p_{1}\left(E;\overline \Omega;u\right),$ a point $y\in E$ satisfying \eqref{eq-inf-contact-pt-p<2-semiconcave} is not a cut point of $x.$ In fact,  $y$ is uniquely determined, and    
the map $\cA^p_1\left(E;\overline \Omega;u\right)\owns x\mapsto y\in E$ is well-defined  as 
\begin{equation*}
y=\left\{
\begin{split}
 &x,\quad &\mbox{for}\,\,\,|\D u(x)|=0,\\
&\exp_x {|\D u(x)|^{p-2}}{ }\D u(x)=\Phi_{p,0}(x)\not\in \Cut(x)\cup\{x\},\quad &\mbox{for}\,\,\,|\D u(x)|>0,
\end{split}
\right.
\end{equation*}
which is  surjective. 
Moreover,  when  $|\D u(x)|=0$ for $x\in\cA^p_1\left(E;\overline \Omega;u\right),$    we have  $y=x=\Phi_{p,\delta}(x)  \not\in\Cut(x)$ in \eqref{eq-inf-contact-pt-p<2-semiconcave}  for any $\delta>0. $   
 If $|\D u(x)|>0$ for $x\in\cA^p_1\left(E;\overline \Omega;u\right),$    then 
\begin{align*} 
 y=\exp_x {|\D u(x)|^{p-2}}{ }\D u(x)=\Phi_{p,0}(x)=\lim_{\delta\to0}\Phi_{p,\delta}(x).
\end{align*} 
 In the latter case,  the curve  
 $\gamma(t):=\phi_p(x,t)=\exp_x t\,{|\D u|^{p-2}}{ }\D u(x)$ is a unique minimizing geodesic joining $\gamma(0)=x$ to $\gamma(1)=\exp_x {|\D u|^{p-2}}\D u(x)\not\in \Cut(x)\cup\{x\}$  with velocity $\dot\gamma(0)={|\D u|^{p-2}}\D u(x)\in \cE_x\setminus\{0\}$, and 
 $\gamma(t_\delta(x))=\Phi_{p,\delta}(x)\not\in \Cut(x)$ for $\delta>0$, where $ t_\delta(x):=\left(\frac{ |\D u(x)|^2}{|\D u(x)|^2+ \delta}\right)^{\frac{2-p}{2}}\in(0,1)$.  We remark that  
  \begin{equation}\label{eq-Phi-delta-lip-lem-jac-semiconcave}
 \left( |\D u|^2+\delta\right)^{\frac{p-2}{2}}\D u(x)\in\cE_x,\quad\forall       x\in \cA_1^ {p}\left(E;\overline \Omega;u\right) \quad\mbox{ for any $\delta>0$},
 \end{equation}
 and      the above argument   implies that  
 \begin{equation}\label{eq-split-E-semiconcave}
\begin{split}
E
&=\Phi_{p,0}\left(     \cA^p_{1}\left(E;\overline \Omega;u\right)\cap \{|\D u| >0\} \right)\, \bigcup   \,\Phi_{p,\delta}\left(  \cA^p_1\left(E;\overline \Omega;u\right)\cap \{|\D u|=0\}\right) \quad\forall \delta>0.
\end{split}
\end{equation}

According to Theorem \ref{thm-AB},  $u$ is twice differentiable almost everywhere in $\Omega$  since $u$ is semiconcave in $\Omega$.
So there is a set $N\subset\Omega$ of measure zero such that    $u$ is twice differentiable in $\Omega\setminus N,$  and then      $u$ satisfies 
   \begin{equation}\label{eq-approx-op-p-less2-semiconcave}
 \left( |\D u|^2+\delta\right)^{\frac{p-2}{2}}\cM^-_{\ld,\Ld}(D^2u)\leq f^+\quad\mbox{ in $\Omega\setminus N \,\,\,$  for any $\delta>0,$ }
 \end{equation}
where $D^2u$ is the Hessian  in the sense of Theorem \ref{thm-AB}. 
Using \eqref{eq-Phi-delta-lip-lem-jac-semiconcave}, \eqref{eq-approx-op-p-less2-semiconcave}   and     Corollary \ref{cor-jacobi-estimate},    a  similar  argument to  the proof of Theorem \ref{thm-abp-type-p-1}     yields    the  following  Jacobian estimate: 
for any $\delta>0$ and  for $ x\in \cA_1^{p}\left(E;\overline \Omega;u\right)\setminus N   ,$
  \begin{align*}
\Jac \Phi_{p,\delta}(x)\leq\sS^{n-1}\left( \sqrt{\frac{\kappa}{\lambda}} |\D u(x)|^{p-1}\right) \left\{h(x)+\frac{  \chi_{\{|\D u|>0\}} \,\,\delta\,(2-p)}{(p-1)\{(p-1)|\D u|^2+\delta\}}\left(\frac{1}{n}+ h(x)^{\frac{1}{n}}\right)^n\right\},
\end{align*}
 where  $$h(x):=\left[\frac{1}{n\ld}\left\{
   {f^+}(x)  +   {\Ld} +  {(n-1)}\Ld  \sH\left(\sqrt{\frac{\kappa}{\Ld}}|\D u(x)|^{p-1}\right) 
 +\chi_{\{|\D u|>0\}}(x)\Ld\frac{2-p }{p-1}\right\}\right]^n.$$
From the dominated convergence theorem,  we deduce that
  \begin{equation}\label{eq-jac-Phi-delta-uniform-semiconcave}
 \begin{split}
&   \limsup_{\delta\to0} \int_{\cA_1^ {p}\left(E;\overline \Omega;u\right)} \Jac\Phi_{p,\delta}(x)dx\\
 &  \leq   \int_{\cA_a^ {p}\left(E;\overline \Omega;u\right)} \sS^{n-1}\left(\sqrt{ \frac{\kappa}{\ld}}   {|\D u|^{p-1}}{} \right)\left[\frac{1}{n\ld}\left\{
   {f^+}  +  {(n-1)}\Ld  \sH\left(\sqrt{ \frac{\kappa}{\Ld}}|\D u|^{p-1}\right) +\frac{\Ld}{p-1} \right\}\right]^ndx
\end{split}
\end{equation}
since $N$ has measure zero.  


Following  the argument in   the  proof of  Theorem \ref{thm-abp-type-p-1}, we claim that
\begin{align}\label{eq-E-contatined-image-phi-semiconcave}
|E|\leq   \limsup_{\delta\to0} \int_{\cA_1^ {p}\left(E;\overline \Omega;u\right)} \Jac\Phi_{p,\delta}(x)dx. 
 \end{align}
In order to  prove \eqref{eq-E-contatined-image-phi-semiconcave}, 
we will  show  that $\D u$ is   Lipschitz continuous  on   $\cA^p_{1}\left(E;\overline \Omega;u\right)\subset\Omega$ by proving the following:  there is a  uniform   constant $C_1>0$  such that for each $x_0\in \cA^p_{1}\left(E;\overline \Omega;u\right)  $, there exists  
  $0< r_0 <  \min ( i_{\overline\Omega }, c_{\overline\Omega }) $  
   such that   
     \begin{equation}\label{eq-grad-u-lip}
    |\D u(x_1)- L_{x_2,x_1}\D u(x_2)|\leq C_1 \,d(x_1,x_2),\quad\forall x_1,x_2 \in B_{r_0}(x_0)\cap \cA^p_{1}\left(E;\overline \Omega;u\right),
    \end{equation}
   where  $L_{x_2,x_1}$ is the parallel transport along the minimizing geodesic joining $x_2$ to $x_1,$ and  $i_{\overline\Omega}>0$ and $c_{\overline\Omega}>0$ are   the minimum of the injectivity radius  and the convexity  radius  of $x\in\overline\Omega$, respectively.   See     \cite[Definition 1.2]{AF} for the notions  of   $C^{1,1}$ smoothness. 

 Once    Lipschitz continuity of $\D u$ on  a compact set  $ \cA^p_{1}\left(E;\overline \Omega;u\right)$ is achieved,  
 we deduce        Lipschitz continuity of  the maps $\left.\Phi_{p,0}\right|_{\cA^p_{1}\left(E;\overline \Omega;u\right)\cap \{|\D u|\geq 1/k\}}$ ($k\in\N$), and $\left.\Phi_{p,\delta}\right|_{\cA^p_{1}\left(E;\overline \Omega;u\right)}$  by using  \eqref{eq-Phi-delta-lip-lem-jac-semiconcave} since the exponential map $\exp\left|_{\cE}: \cE\to M\right.$ is smooth.  
 For   the  proof of \eqref{eq-E-contatined-image-phi-semiconcave},  we apply  the area formula to   Lipschitz maps $\left.\Phi_{p,0}\right|_{\cA^p_{1}\left(E;\overline \Omega;u\right)\cap \{|\D u|\geq 1/k\}}$ ($k\in\N$), and $\left.\Phi_{p,\delta}\right|_{\cA^p_{1}\left(E;\overline \Omega;u\right)}$ as  in the proof of    Theorem \ref{thm-abp-type-p-1}.  
 Here,  we  remark that  the area formula on Riemannian  manifolds     follows from  the area formula in Euclidean space   using    a partition of unity. So the above notion of  Lipschitz continuity on     $\cA^p_{1}\left(E;\overline \Omega;u\right)$,  which is    uniformly locally Lipschitz,  suffices to employ  the area formula in our setting.  
 Thus  a similar  argument      to      the  proof of \eqref{eq-E-contatined-image-phi} in   Theorem \ref{thm-abp-type-p-1}     with  the help of Lipschitz continuity  and       \eqref{eq-split-E-semiconcave}  yields   \eqref{eq-E-contatined-image-phi-semiconcave}, 
  and therefore   it follows  from \eqref{eq-jac-Phi-delta-uniform-semiconcave}  that 
  \begin{equation*}
|E|\leq \int_{\cA_1^{p}\left(E;\overline \Omega;u\right)}\sS^{n-1}\left(\sqrt{ \frac{\kappa}{\lambda}}   {|\D u|^{p-1}}{} \right)
\left[\frac{1}{n\ld}\left\{
   {f^+}  +  {(n-1)}\Ld  \sH\left(\sqrt{ \frac{\kappa}{\Ld}}|\D u|^{p-1}\right) +\frac{\Ld}{p-1} \right\}\right]^ndx.
\end{equation*}

To complete the proof,    
 it remains   to show   Lipschitz continuity of $\D u$ on $\cA^p_{1}\left(E;\overline \Omega;u\right)$.  We claim that  there is a  uniform   constant $C_1>0$  such that for each  $x_0\in \cA^p_{1}\left(E;\overline \Omega;u\right),$  there is   $0< r_0  < \min( i_{\overline\Omega}, c_{\overline\Omega}) $  such that  for any $x\in B_{r_0}(x_0)\cap \cA^p_{1}\left(E;\overline \Omega;u\right)$, 
   \begin{equation}\label{eq-grad-u-lip-claim}
   \left| u(z)-u(x)- \left\langle\D u(x),\exp_x^{-1}z\right\rangle\right|\leq C_1d^2(x,z),\quad \forall z\in B_{r_0}(x_0). 
   \end{equation}
We put off the proof of \eqref{eq-grad-u-lip-claim}, and first prove that this claim \eqref {eq-grad-u-lip-claim}
implies  \eqref{eq-grad-u-lip}.  Assume that the claim holds. Let us  fix  $x_0\in \cA^p_{1}\left(E;\overline \Omega;u\right),$ and let    $0< r_0  < \min( i_{\overline\Omega}, c_{\overline\Omega}) $ be a constant satisfying \eqref{eq-grad-u-lip-claim}. Let   $x_1,x_2\in B_{r_0/2}(x_0)\cap \cA^p_{1}\left(E;\overline \Omega;u\right)$,   $r:= d(x_1,x_2)$ and $V:=\exp_{x_1}^{-1}x_2$. From \eqref{eq-grad-u-lip-claim}, it follows   that for  any  $z=\exp_{x_2} L_{x_1,x_2}X $ with $X\in T_{x_1}M$    and  $|X|\leq r/4$, 
\begin{align*}
&\left|  \left\langle\D u(x_2),\exp_{x_2}^{-1}z\right\rangle- \left\langle\D u(x_1),\exp_{x_1}^{-1}z-\exp_{x_1}^{-1}(x_2)\right\rangle \right|\\
=&\left| u(z)-u(x_1) -  \left\langle\D u(x_1),\exp_{x_1}^{-1}z\right\rangle-\left\{u(z)-u(x_2) -  \left\langle\D u(x_2),\exp_{x_2}^{-1}z\right\rangle\right\}\right.\\
&\left. - \left\{ u(x_2)-u(x_1)-  \left\langle\D u(x_1),\exp_{x_1}^{-1}x_2\right\rangle\right\}\right|\\
\leq&  C_1\left\{d^2(x_1,z)+d^2(x_2,z)+d^2(x_1,x_2)\right\}\\
<&  C_1\left\{\left(d(x_1,x_2)+d(x_2,z)\right)^2+(r/4)^2+r^2\right\} < 6C_1 r^2
\end{align*} 
since $d(z,x_0)\leq d(z,x_2)+d(x_2,x_0)< r/4+ r_0/2< r_0.$ 
Setting $z=\exp_{x_1} W=\exp_{x_2} L_{x_1,x_2}X $ for $W\in T_{x_1 } M$,  we can rewrite that 
\begin{equation}\label{eq-grad-u-lip-claim-to-lip-pre}
\begin{split}
6C_1r^2>&\left|  \left\langle\D u(x_2),\exp_{x_2}^{-1}z\right\rangle- \left\langle\D u(x_1),\exp_{x_1}^{-1}z-\exp_{x_1}^{-1}(x_2)\right\rangle \right|\\
=& \left|  \left\langle\D u(x_2), L_{x_1,x_2}X\right\rangle- \left\langle\D u(x_1), W-V\right\rangle \right|\\
=& \left|  \left\langle L_{x_2,x_1}\D u(x_2) -\D u(x_1), \, X\right\rangle- \left\langle\D u(x_1), W-X-V\right\rangle \right|. 
\end{split}
\end{equation} 
  Recall from    \cite[Claim 5.3]{AF} that   for a fixed  point $x_0,$  there are     $r_2>0$  and     $C_2>0 $ such that 
  $$|W-X-V|  \leq    C_2\left(|V|+|X|\right)^3 $$
  for any   $x_1,x_2\in B_{r_2}(x_0) $ and       $X\in T_{x_1}M$  with   $|X|\leq 2r_2$, where $V= \exp_{x_1}^{-1} x_2,  $ and $W=\exp_{x_1} ^{-1}\left(\exp_{x_2} L_{x_1,x_2}X\right)  $.     
Thus   we can  choose       $0<r_0 <r_2$ sufficiently small such  that  
  \begin{align*}
  \left| \left\langle\D u(x_1), W-X-V\right\rangle \right| &\leq ||\D u||_{L^{\infty}\left(\cA^p_{1}\left(E;\overline \Omega;u\right)\right)}\cdot  C_2\left(|V|+|X|\right)^3   \\
  &\leq ||\D u||_{L^{\infty}\left(\cA^p_{1}\left(E;\overline \Omega;u\right)\right)} \cdot C_2 \cdot 8r^3 \leq C_1 r^2, 
  \end{align*}
   since $|\D u|$ is uniformly bounded on $\cA^p_{1}\left(E;\overline \Omega;u\right),$ and $|X|\leq r/4=|V|/4<r_0/2.$ Therefore,  by selecting      $r_0>0$  sufficiently small, we obtain from \eqref{eq-grad-u-lip-claim-to-lip-pre}    that  for any    $x_1,x_2\in B_{r_0}(x_0)\cap \cA^p_{1}\left(E;\overline \Omega;u\right)$ with $r:= d(x_1,x_2)$,      and  $X\in T_{x_1}M$ with $|X|\leq r/4,$ 
$$\left|  \left\langle L_{x_2,x_1}\D u(x_2) -\D u(x_1), \, X\right\rangle \right|< 7C_1 r^2,$$ 
and hence  \eqref{eq-grad-u-lip} follows  since $X\in B_{r/4}(0)\subset T_{x_1}M$ is arbitrary.  We refer to \cite[Proposition 3.5]{IS} for the Euclidean case.   

Lastly,  we  prove \eqref{eq-grad-u-lip-claim} by  making use of semi-concavity  and  \eqref{eq-inf-contact-pt-p<2-semiconcave} combined with  Lemma \ref{lem-hess-dist-sqrd}.    Let us fix   $x_0\in \cA^p_{1}\left(E;\overline \Omega;u\right).$  Due to $C_0$-semiconcavity,   there is small  $0<r_0< \min( i_{\overline\Omega}, c_{\overline\Omega })$ such that  
 $u-C_0 d_{x_0}^2$ is  geodesically concave in  $B_{r_0}(x_0)\subset \Omega$. From differentiability of $u$ on  $ \cA^p_{1}\left(E;\overline \Omega;u\right),$ 
  we deduce  that for $x\in B_{r_0}(x_0)\cap \cA^p_{1}\left(E;\overline \Omega;u\right)$ and  $z\in B_{r_0}(x_0)$,
\begin{equation} \label{eq-grad-u-lip-claim-proof-upper}
\begin{split}
u(z)-u(x)-\left\langle\D u(x), \exp_x^{-1}z\right\rangle&\leq C_0\left\{d_{x_0}^2(z)-d_{x_0}^2(x)-\left\langle\D d_{x_0}^2(x), \exp_x^{-1}z\right\rangle\right\}\\
& \leq 2C_0\sH\left(\sqrt{\tilde\kappa} r_0\right)d^2(x,z)\leq2C_0\sH\left(\sqrt{\tilde\kappa} \,i_{\overline\Omega}\right) d^2(x,z)
\end{split}
\end{equation}
  by using the Taylor  expansion and Lemma   \ref{lem-hess-dist-sqrd},  where $-\tilde \kappa$ ($\tilde\kappa\geq 0$) is  a lower bound of the sectional curvature on $\overline\Omega. $  

On the other hand,  if $x\in  B_{r_0}(x_0)\cap \cA^p_{1}\left(E;\overline \Omega;u\right)\cap\{|\D u|=0\},$ then
\eqref{eq-inf-contact-pt-p<2-semiconcave} is satisfied with $y=x$, and it follows  that for any $ z\in B_{r_0}(x_0)\subset\Omega, $
\begin{align*}
u(z)-u(x)-\left\langle \D u(x), \exp_x^{-1}z\right\rangle&=u(z)-u(x)\\
&\geq -\frac{p-1}{p}d_x^{\frac{p}{p-1}}(z)\geq - (2r_0)^{\frac{2-p}{p-1}} d^2(x,z)\geq - (2i_{\overline\Omega})^{\frac{2-p}{p-1}} d^2(x,z)
\end{align*}
 since $1<p<2.$ This  proves \eqref{eq-grad-u-lip-claim} for $x\in  B_{r_0}(x_0)\cap \cA^p_{1}\left(E;\overline \Omega;u\right)\cap\{|\D u|=0\}. $     Here, the constants  $2C_0\sH\left(\sqrt{\tilde\kappa} \,i_{\overline\Omega}\right)$  and $(2i_{\overline\Omega})^{\frac{2-p}{p-1}}$ are  uniform. 
  
  Now, consider    $x\in  B_{r_0}(x_0)\cap \cA^p_{1}\left(E;\overline \Omega;u\right)\cap\{|\D u|>0\}.$ Note that the corresponding $y\in E$    in  \eqref{eq-inf-contact-pt-p<2-semiconcave}  does not belong to  $ \Cut(x)\cup\{x\}$ as seen before. 
  Let  $ \tilde r:=\min\left( i_{\overline\Omega}, c_{\overline\Omega}\right)/5>0.$ We may assume that $0<r_0 <\tilde r.$    If $d(x,y)\leq  3 \tilde r,$ then 
    $\frac{p-1}{p}d_y^{\frac{p}{p-1}} $ is of class $ C^2$ in $B_{r_0}(x_0) $ since  $  d(z,\tilde y) <5  \tilde r$ for   $z\in  B_{r_0}(x_0)\subset \Omega$ and $ \tilde r =  \min\left( i_{\overline\Omega}, c_{\overline\Omega}\right)/5.$ 
From \eqref{eq-inf-contact-pt-p<2-semiconcave},  we see that for  $z\in B_{r_0}(x_0)$,
\begin{align*}
u(z)-u(x)- \left\langle \D u(x), \exp_x^{-1}z\right\rangle 
 &\geq -\frac{p-1}{p}d_y^{\frac{p}{p-1}}(z)+\frac{p-1}{p} d^{\frac{p}{p-1}}_y(x)+\left\langle  \frac{p-1}{p}  \D d^{\frac{p}{p-1}}_y(x), \exp_x^{-1}z\right\rangle \\
&\geq  - (5\tilde r )^{\,\,\frac{2-p}{p-1}}\left\{ \sH\left(\sqrt{\tilde\kappa'}\, 5\tilde r\right)+\frac{2-p}{p-1}\right\}d^2(z,x)
\end{align*}    
  by using the Taylor  expansion and Lemma   \ref{lem-hess-dist-sqrd} again, 
  where $-\tilde \kappa'$ ($\tilde\kappa'\geq 0$) is  a lower bound of the sectional curvature on a bounded set $\{   z'\in M: d(  z',   x' )\leq 5\tilde r,\,\,  x'\in  \overline\Omega\}.$  
  
    When  $d(x,y)>3 \tilde r,$ we   define  $\tilde y:= \exp_{x} - 3 \tilde r\,\D d_y(x)\in \p B_{3 \tilde r}(x).$ We observe        that 
 \begin{align*}
 d_y^{\frac{p}{p-1}}(x)&=  \left\{ d_{\tilde y}(x)+d_{y}(\tilde y)\right\}^{\frac{p}{p-1}},\\
  d_y^{\frac{p}{p-1}}(z)&\leq   \left\{ d_{\tilde y}(z)+d_{y}(\tilde y)\right\}^{\frac{p}{p-1}},\,\,\, \mbox{and}\,\,\,\, \tilde r <d(z,\tilde y) <5  \tilde r,\quad  \forall z\in B_{r_0}(x_0)\subset  B_{2r_0}(x)
 \end{align*}
 since we assume $0<r_0<\tilde r.$  
 Let   $\rho(z):=\frac{p-1}{p}\left\{ d_{\tilde y}(z)+d_{y}(\tilde y)\right\}^{\frac{p}{p-1}} .$ Since $\tilde r< d(z,\tilde y) <5  \tilde r$ for any $z\in  B_{r_0}(x_0)\subset \Omega$ and $ \tilde r =  \min\left( i_{\overline\Omega}, c_{\overline\Omega}\right)/5$,  $\rho$ is smooth    in  $\ B_{r_0}(x_0)$, and 
\eqref{eq-inf-contact-pt-p<2-semiconcave}   implies  that  for    $z\in B_{r_0}(x_0)$,
\begin{align*}
u(z)-u(x) &\geq -\frac{p-1}{p}d_y^{\frac{p}{p-1}}(z)+\frac{p-1}{p} d^{\frac{p}{p-1}}_y(x)\geq  -\rho(z)+\rho(x),
\end{align*}    
and hence 
\begin{align*}
u(z)-u(x) - \left\langle \D u(x), \exp_x^{-1}z\right\rangle&\geq  -\rho(z)+\rho(x) + \left\langle \D \rho(x), \exp_x^{-1}z\right\rangle\\
&\geq  -  \left\langle D^2\rho(\sigma(s))\cdot\dot\sigma(s), \dot\sigma(s) \right\rangle^+d^2(z,x)/2
\end{align*}
for some $s\in[0,d(x,z)],$    
where $\sigma$ is the minimal geodesic joining $x$ to $z$   parametrized by arc length.  Note that $\sigma([0,d(x,z)])\subset B_{r_0}(x_0),$ and  $\tilde r< d(\tilde y, \sigma(s))<5 \tilde r.$ 
Using  Lemma \ref{lem-hess-dist-sqrd}, we see that 
\begin{align*}
 \left\langle D^2\rho(\sigma(s))\cdot\dot\sigma(s), \dot\sigma(s) \right\rangle
 &\leq \left\{ d_{\tilde y}(\sigma(s))+d_{y}(\tilde y)\right\}^{\frac{2-p}{p-1}} \left\{\left(d_{\tilde y}(\sigma(s))+d_{y}(\tilde y)\right) \left\langle D^2d_{\tilde y}(\sigma(s))\cdot\dot\sigma(s), \dot\sigma(s) \right\rangle^+ +1 \right\}\\
 &\leq \left\{ 5\tilde r+d_{y}(x)\right\}^{\frac{2-p}{p-1}} \left\{\left(1+\frac{d_{y}(\tilde y)}{d_{\tilde y}(\sigma(s))}\right)d_{\tilde y}(\sigma(s)) \left\langle D^2d_{\tilde y}(\sigma(s))\cdot\dot\sigma(s), \dot\sigma(s) \right\rangle^+ +1 \right\}\\ 
 &\leq \left\{ 5\tilde r+d_{y}(x)\right\}^{\frac{2-p}{p-1}} \left\{\left(1+ {d_{y}(x)}/{\tilde r}\right)\sH\left(5\sqrt{\tilde\kappa' }\,\tilde r\right) +1 \right\},
\end{align*}
  and hence 
 \begin{align*}
u(z)-u(x) - \left\langle \D u(x), \exp_x^{-1}z\right\rangle 
\geq -\left( 5\tilde r+ \tilde  d\right)^{\frac{2-p}{p-1}} \left\{\left(1+ {\tilde  d}/{\tilde r}\right)\sH\left(5\sqrt{\tilde\kappa' }\,\tilde r\right) +1 \right\} d^2(z,x)
\end{align*} for $\tilde d:=\diam (\Omega\cup E)$, where $\left( 5\tilde r+ \tilde  d\right)^{\frac{2-p}{p-1}} \left\{\left(1+ {\tilde  d}/{\tilde r}\right)\sH\left(5\sqrt{\tilde\kappa' }\,\tilde r\right) +1 \right\}$ depends only on $p, \Omega$ and $E.$  This  finishes the proof of  \eqref{eq-grad-u-lip-claim} for $x\in  B_{r_0}(x_0)\cap \cA^p_{1}\left(E;\overline \Omega;u\right)\cap\{|\D u|>0\}$ together with \eqref{eq-grad-u-lip-claim-proof-upper}.       Thus  we have proved  \eqref{eq-grad-u-lip-claim}. Therefore we 
  conclude that  $\D u$ is Lipschitz continuous on  $\cA^p_{1}\left(E;\overline \Omega;u\right),$ which completes the proof.   
\end{proof}

Making use of  Theorem \ref{thm-abp-type-p-1} and  Lemma \ref{lem-abp-type-p-1-nonlinear}, we prove the ABP type estimate for the $p$-Laplacian operator ($1<p<2$) assuming the Ricci curvature to be bounded from below.

 \begin{lemma}
 \label{lem-abp-type-p-1}
 Assume that          $1<p<2$   and $\Ric \geq - {(n-1)}\kappa$   for   $\kappa\geq0$. For a bounded open set  $\Omega\subset M,$  
let  $f\in C(\Omega)$ and   $u\in C\left(\overline\Omega\right)$  be   $C_0$-semiconcave in $\Omega$  for  $C_0>0$  such that   $\La_pu\leq f$  in $\Omega $ in the viscosity sense.      
For   a compact set $E\subset M,$ we assume that 
   $\cA^p_{1}\left(E;\overline \Omega;u\right)\subset \Omega.$ Then      
\begin{equation*}
|E|\leq \int_{\cA_1^{p}\left(E;\overline \Omega;u\right)}\frac{  \sS^{n-1}\left(\sqrt{ {\kappa}}   {|\D u|^{p-1}}{} \right)}{(p-1)^n}\left\{
   \frac{f^+}{n}  +    \sH\left(\sqrt{  {\kappa}}|\D u|^{p-1}\right) \right\}^ndx.
\end{equation*}
 In particular, if $\Ric\geq0,$  then  
\begin{equation*}
|E|\leq \int_{\cA_1^{p}\left(E;\overline \Omega;u\right)}\frac{ 1}{(p-1)^n}\left(
   \frac{f^+}{n}  +  1    \right)^ndx.
\end{equation*}
  \end{lemma}

\begin{proof}
First, we   recall    from  Lemma \ref{lem-pucci-delta-p-laplace-delta} that $u$ satisfies 
 \begin{equation}\label{eq-approx-op-p-laplace-p-1}
 \left( |\D u|^2+\delta\right)^{\frac{p-2}{2}}\cM^-_{p-1,1}(D^2u)\leq f^+ 
 \end{equation}
 in $\Omega $  for any $\delta>0$  
 in the sense of  Definition \ref{def-visc-sol}   and  in the usual sense        from Lemma \ref{lem-visc-sol-singular-usual}. Using  the same notation as in the proof of Lemma  \ref{lem-abp-type-p-1-nonlinear}, semiconcavity implies that  
   \eqref{eq-approx-op-p-laplace-p-1}    is satisfied  pointwise in $\Omega\setminus N$ with   the Hessian $D^2u$  in the sense of Theorem \ref{thm-AB},
 where $u$ is twice differentiable in $\Omega\setminus N$, and a set $N\subset\Omega$ has measure zero. 
   
    As in the proof of  Lemma \ref{lem-abp-type-p-1-nonlinear},     
 for $\delta>0,$       consider   a regularized  map $\Phi_{p,\delta}:\cA^p_{1}\left(E;\overline \Omega;u\right)\to M$ defined   as 
$\Phi_{p,\delta}(x):=\exp_x\left( |\D u(x)|^2+\delta\right)^{\frac{p-2}{2}}\D u(x).$ 
From Lemma \ref{lem-abp-type-p-1-nonlinear}, $\Phi_{p,\delta}$  is Lipschitz  continuous in $\cA^p_{1}\left(E;\overline \Omega;u\right)$ for each $\delta>0,$ and       
  \begin{equation}\label{eq-E-contatined-image-phi-semiconcave-ricci}
|E|\leq   \limsup_{\delta\to0} \int_{\cA_1^ {p}\left(E;\overline \Omega;u\right)} \Jac\Phi_{p,\delta}(x)dx. 
 \end{equation}

It suffices  to  obtain    a uniform estimate of $\Jac\Phi_{p,\delta}$ on $\cA_1^ {p}\left(E;\overline \Omega;u\right)\setminus N$ with respect to $\delta>0$  in terms of a lower  bound $-\kappa$ of  the Ricci curvature.    
First, let $R_0:=\diam\left(\Omega\cup E\right)$ and $z_0\in \Omega.$ We choose   
    $\tilde\kappa\geq0$  such that   
  \begin{equation}\label{eq-pucci-ricci-tilde-kappa}
  \cM^-_{{p-1},1}(R(e))\geq - {(n-1)}\tilde\kappa\quad  \mbox{   for any unit vector  $e\in T_xM$ and  $x\in \overline B_{2R_0}(z_0)$}.
  \end{equation} 
Note from   Lemma \ref{lem-abp-type-p-1-nonlinear}  that    $\Jac\Phi_{p,\delta}$ is uniformly bounded for $\delta>0$ in terms of $-\tilde \kappa$  using   \eqref{eq-approx-op-p-laplace-p-1} and \eqref{eq-pucci-ricci-tilde-kappa}.  
  With the same notation as  the proofs of  Theorem \ref{thm-abp-type-p-1} and Lemma \ref{lem-abp-type-p-1-nonlinear}, \eqref{eq-jac-Phi-delta-det-H-delta} yields  
that for $x\in \cA_1^{p}\left(E;\overline \Omega;u\right)\setminus N,$
\begin{equation}\label{eq-jacobi-est-nonlinear-to-trace-pre}
\begin{split}
\Jac \Phi_{p,\delta}(x)&\leq\Jac \exp_x \left(\left(|\D u(x)|^2+\delta\right)^{\frac{p-2}{2}}\D u(x)\right)\cdot   \det\left({\bf I}+(p-2) \frac{\D u}{\sqrt{|\D u|^2+\delta}}\otimes \frac{\D u}{\sqrt{|\D u|^2+\delta}}\right)  \\
&\cdot \left\{\det H_{\delta}(x)+  \frac{\chi_{\{|\D u|>0\}}\,\delta(2-p)}{(p-1)\{(p-1)|\D u|^2+\delta\}}\left( \frac{1}{n}+\frac{1}{n}\tr   H_\delta(x)\right)^n\right\},
\end{split}
\end{equation} 
where   we recall   
  $$ H_\delta(x):=\left(|\D u|^2+\delta\right)^{\frac{p-2}{2}}  D^2u(x) +\left({\bf I}+\chi_{\{|\D u|>0\}}\frac{2-p }{p-1 } \frac{\D u}{|\D u|}\otimes \frac{\D u}{|\D u|}\right)\circ D^2 (d_{ \Phi_{p,\delta}(x)}^2/2)(x),$$
which is symmetric and positive semi-definite.    Using  \eqref{eq-approx-op-p-laplace-p-1} and  \eqref{eq-pucci-ricci-tilde-kappa}, 
it follows  from  the proof of \eqref{eq-unif-bd-det-H-delta}      that  for     $x\in \cA_1^{p}\left(E;\overline \Omega;u\right)\setminus N,$
     \begin{equation}\label{eq-abp-H-delta-trace-nonlinear-to-trace}
     \begin{split}
     \det H_\delta(x)&\leq \left( \frac{1}{n}\tr H_\delta(x)\right)^n  \\
     &\leq \left[\frac{1}{n(p-1)}\left\{
   {f^+}  +   1 +  {(n-1)}   \sH\left(\sqrt{ {\tilde \kappa}{ }}\,|\D u|^{p-1}\right) 
 + \chi_{\{|\D u|>0\}}\frac{2-p }{p-1}\right\}\right]^n=:  \tilde h(x).
     \end{split}
     \end{equation}

    From \eqref{eq-jac-det-phi-du-0} and \eqref{eq-abp-H-delta-trace-nonlinear-to-trace},   
we have   that for  $x\in \cA_1^{p}\left(E;\overline \Omega;u\right)\cap \{z\in\Omega\setminus N: \D u(z)=0\}$, 
\begin{equation}\label{eq-p-laplace-p-leq2-grad-0}
\begin{split}
\Jac \Phi_{p,\delta}(x)&= \det H_{\delta}(x) 
  \leq \left[\frac{1}{n(p-1)} \left\{f^+(x)+ n \right\}\right]^n.
  \end{split}
\end{equation}
For   $x\in \cA_1^{p}\left(E;\overline \Omega;u\right)\cap\{z\in\Omega\setminus N: |\D u(z)|>0\}$, let         
$$D_\delta(x):={\bf I}+(p-2) \frac{\D u}{\sqrt{|\D u|^2+\delta}}\otimes \frac{\D u}{\sqrt{|\D u|^2+\delta}}.$$
Using  Theorem \ref{thm-BG} and the   arithmetic-geometric means inequality, we  deduce from  \eqref{eq-jacobi-est-nonlinear-to-trace-pre} and \eqref{eq-abp-H-delta-trace-nonlinear-to-trace}   that  for    $x\in \cA_1^{p}\left(E;\overline \Omega;u\right)\cap\{z\in\Omega\setminus N: |\D u(z)|>0\}$, 
 \begin{align*}
\Jac \Phi_{p,\delta}(x)
&\leq \Jac \exp_x \left(\left(|\D u(x)|^2+\delta\right)^{\frac{p-2}{2}}\D u(x)\right)\\
&\cdot  \left\{\det\left( D_\delta\circ H_\delta \right)+   \frac{\delta(2-p)}{(p-1)\{(p-1)|\D u|^2+\delta\}}\left( \frac{1}{n}+\frac{1}{n}\tr   H_\delta\right)^n\right\}\\
&\leq \sS^{n-1}\left( \sqrt{ {\kappa}{ }} |\D u|^{p-1}\right)
 \cdot\left\{\left( \frac{1}{n}\tr\left( D_\delta\circ H_\delta \right)\right)^n+   \frac{\delta\,(2-p)}{(p-1)\{(p-1)|\D u|^2+\delta\}}\left(  \frac{1}{n} +  \tilde h(x)^{\frac{1}{n}} \right)^n \right\}
\end{align*}
since $\Ric\geq-(n-1)\kappa$, and $D_\delta$ and $H_\delta$ are symmetric and positive semi-definite.    
Notice that  for $x\in \cA_1^{p}\left(E;\overline \Omega;u\right)\cap\{z\in\Omega\setminus N: |\D u(z)|>0\}$, $$\tr\left( D_\delta\circ H_\delta \right)(x)\,\,\mbox{ converges to}\,\,  \La_p u(x)+\La d^2_y(x)/2  $$ for $y=\exp_x|\D u|^{p-2}\D u(x)\not\in\Cut(x)$   as $\delta$ tends to $0$, and  
  $$\La_p u(x)+\La d^2_y(x)/2  \leq f^+(x)+1+(n-1)\sH\left(\sqrt{\kappa}|\D u(x)|^{p-1}\right)$$   from    Lemma \ref{lem-pucci-ric-dist-sqrd}.     
 Since $0\leq \tr \left(D_\delta\circ H_\delta\right)\leq \tr H_\delta\leq n\tilde h^{\frac{1}{n}}$ in $\cA_1^{p}\left(E;\overline \Omega;u\right)\setminus N,$ and $N\subset \Omega$ has measure zero,    
  we apply the dominated convergence theorem   to deduce   that   
 \begin{align*}
    \limsup_{\delta\to0} \int_{\cA_1^ {p}\left(E;\overline \Omega;u\right)\cap\{|\D u |>0\}} \Jac\Phi_{p,\delta}
   \leq    \int_{\cA_1^ {p}\left(E;\overline \Omega;u\right)\cap\{|\D u |>0\}}   \sS^{n-1}\left(\sqrt{ {\kappa}}   {|\D u|^{p-1}}{} \right) \left\{
   \frac{f^+}{n}  +  \sH\left(\sqrt{  {\kappa}}|\D u|^{p-1}\right) \right\}^n.
\end{align*} 
This combines   with  \eqref{eq-p-laplace-p-leq2-grad-0} and \eqref{eq-E-contatined-image-phi-semiconcave-ricci}  
to  complete the proof. 
\end{proof}

The following lemma is concerned with the ABP type estimate for viscosity solutions. 

 \begin{lemma}\label{lem-abp-type-measure-p-small}
 Assume that          $1<p<2$   and $\Ric \geq - {(n-1)}\kappa$   for   $\kappa\geq0$. For        $z_0,x_0\in M$ and $0<r\leq R\leq R_0,$ assume that $  \overline B_{4r}(z_0)\subset B_{R}(x_0).$   Let $f\in C\left(B_{R}(x_0)\right)$ and  $u\in C\left( \overline  B_{R}(x_0)\right)$ be    such that   $$\La_pu\leq f \quad\mbox{ in $B_{R}(x_0),$}$$
in the viscosity sense,   
\begin{equation}\label{eq-abp-cond-u-p-1}
u\geq0\quad\mbox{on $  B_{R}(x_0) \setminus B_{4r}(z_0) \quad$ 
and $\quad \displaystyle\inf_{B_{r}(z_0)}u\leq \frac{p-1}{p} .$}
\end{equation}
Then     
\begin{equation*}
\begin{split}
|B_r(z_0)|
&\leq \frac{\sS^{n-1}\left(2\sqrt{\kappa}R_0\right)}{  (p-1)^n}\int_{\{u\leq \tilde M_p\}\cap   B_{4r}(z_0)} \left\{\sH\left( 2\sqrt{\kappa}R_0\right) +\frac{r^p f^+}{n}\right\}^n
\end{split}
\end{equation*}
for a uniform constant $\tilde M_p:=\frac{p-1}{p}3^{\frac{p}{p-1}}$.  
  Moreover, if $r^pf\leq 1$ in $B_{4r}(z_0),$ then there exists a uniform constant $\delta\in(0,1)$ depending only on  $n, p$ and $\sqrt{\kappa}R_0,$ such that 
  \begin{equation}\label{eq-abp-type-measure-p-1}
  \left|\left\{ u\leq \tilde M_p\right\}\cap B_{4r}(z_0)\right|>\delta|B_{4r}(z_0)|.
  \end{equation}

    \end{lemma}

    \begin{proof}
         Let $\eta>0.$ According to  Lemma  \ref{lem-visc-u-u-e-properties} 
     and Remark \ref{rmk-p-La-eq-u-e},   the inf-convolution $u_\ve$ of $u$ with respect to $\overline B_{R}(x_0)$ (for small $\ve>0$) satisfies 
\begin{equation*}
\left\{
\begin{split}
&u_\ve \to u\quad\mbox{uniformly in $B_{R}(x_0)$},\\
&u_\ve\geq -\eta\quad\mbox{in $B_{R}(x_0)\setminus B_{4r}(z_0) $,}\qquad\inf_{B_{r}(z_0)}u_\ve\leq \frac{p-1}{p}+\eta,\\
&\mbox{$u_\ve $ is    $C_\ve$-semiconcave in $B_{R}(x_0),$} \\
&\La_pu_\ve -\tilde \kappa(n+p-2)|\D u_\ve|^{p-2}  {}{ } \min\left\{\ve|\D u_\ve|^2,  2{ }\omega\left(2\sqrt{m\ve}\right)\right\}     \leq f_\ve  \,\,\mbox{in $B_{4r}(z_0)$} 
 \end{split}\right. 
\end{equation*}
 in the viscosity sense, 
where $-\tilde\kappa $ ($\tilde\kappa\geq0$) is a lower bound of the sectional curvature on $\overline B_{3R}(x_0),$  $\displaystyle C_\ve:=\frac{1}{\ve}{\sH\left(2\sqrt{ \tilde\kappa}\,R\right)}$,  and 
$$f_\ve(z):= \sup_{B_{2\sqrt{m\ve } }(z)}f^+, \quad\forall z\in B_{4r}(z_0);\quad m:=\|u\|_{L^{\infty}(B_R(x_0))}.$$ 
 Then we show  that 
 $$\La_pu_\ve  \leq f_\ve + \tilde\kappa (n+p-2)\, \max\left\{\ve,  2{ }\omega\left(2\sqrt{m\ve}\right)\right\}      \quad\mbox{in $B_{4r}(z_0)\,\,$ in the viscosity sense}.$$
 Indeed,  let   $\varphi\in C^2(B_{4r}(z_0))$ be  such that 
$u_\ve-\varphi$ has a local minimum     at $x$ with $\D \varphi(x)\not=0.$
 If   $0<|\D \vp (x)|\leq 1,$ then $|\D \vp|^{p-2}  {}{ } \min\left\{\ve|\D \vp|^2,  2{ }\omega\left(2\sqrt{m\ve}\right)\right\}\leq \ve |\D \vp|^{p}\leq\ve .$  If $|\D \vp (x)|> 1,$  then  $|\D \vp|^{p-2}  {}{ } \min\left\{\ve|\D  \vp|^2,  2{ }\omega\left(2\sqrt{m\ve}\right)\right\}\leq 2 \omega\left(2\sqrt{m\ve}\right) .$  So it follows  that $$\La_p\vp(x)  \leq f_\ve(x) + \tilde\kappa (n+p-2)\, \max\left\{\ve,  2{ }\omega\left(2\sqrt{m\ve}\right)\right\}.$$  
   
  For small $\ve>0,$ define   
  $$ c_\eta:=\frac{(p-1)/p }{(p-1)/p+2\eta} ,\quad \mbox{and}\quad\tilde u_\ve:=c_\eta\left(u_\ve+\eta\right),$$ which 
  satisfies \eqref{eq-abp-cond-u-p-1} and 
 \begin{align*}
 r^{ {p}}\La_p\tilde u_\ve &= r^pc_\eta^{p-1} \La_p u_\ve
 \leq r^pc_\eta^{p-1} \left[f_\ve +\tilde \kappa n\, \max\left\{\ve,  2{ }\omega\left(2\sqrt{m\ve}\right)\right\} \right]  \quad\mbox{in $B_{4r}(z_0)$}
\end{align*}
 in the viscosity sense. 
 Using  \eqref{eq-abp-cond-u-p-1}, a similar argument  to  the proof of  Lemma \ref{lem-abp-type-measure-p-large} implies   
  \begin{equation}\label{eq-p-contact-set-level-set-grad-p-leq-2}
  \begin{split}
   \cA^p_{1}\left( \overline B_{r}(z_0); \overline  B_{4r}(z_0);r^{\frac{p}{p-1}}\tilde u_\ve\right)&= \cA^p_{1}\left( \overline B_{r}(z_0); \overline  B_{R}(x_0);r^{\frac{p}{p-1}}\tilde u_\ve\right)\\
  & \subset B_{4r}(z_0)\cap \{ \tilde u_\ve< \tilde M_p,\, r^p|\D \tilde u_\ve|^{p-1}< 2R_0 \}
   \end{split}
   \end{equation}
   for   $\tilde M_p:=\frac{p-1}{p}3^{\frac{p}{p-1}}>1,$ where we note that $r^{\frac{p}{p-1}}\tilde u_\ve$ is differentiable in  $  \cA^p_{1}\left( \overline B_{r}(z_0); \overline  B_{4r}(z_0);r^{\frac{p}{p-1}}\tilde u_\ve\right).$
Using \eqref{eq-p-contact-set-level-set-grad-p-leq-2}, we    apply     Lemma \ref{lem-abp-type-p-1}  to $r^{\frac{p}{p-1}}\tilde u_\ve$ in order  to   deduce that
  \begin{align*}
| B_{r}(z_0)|&\leq \frac{  \sS^{n-1}\left(2\sqrt{ {\kappa}}R_0  \right)}{n^n(p-1)^n}\int_{\cA_1^{p}\left(\overline B_{r}(z_0);  \overline  B_{4r}(z_0);r^{\frac{p}{p-1}}\tilde u_\ve\right)}\left[r^pc_\eta^{p-1} \left\{f_\ve +\tilde \kappa n\, \max\left\{\ve,  2{ }\omega\left(2\sqrt{m\ve}\right)\right\} \right\}  +  n  \sH\left({  2 \sqrt{\kappa}} R_0\right)\right]^n\\
&\leq \frac{  \sS^{n-1}\left(2\sqrt{ {\kappa}}R_0  \right)}{n^n(p-1)^n}\int_{B_{4r}(z_0)\cap \{ \tilde u_\ve< \tilde M_p \}} \left[ r^pc_\eta^{p-1} \left\{f_\ve +\tilde \kappa n\, \max\left\{\ve,  2{ }\omega\left(2\sqrt{m\ve}\right)\right\} \right\}  +  n  \sH\left(\ 2\sqrt{ {\kappa}}R_0 \right)\right]^n.
\end{align*}
Letting    $\ve$ go  to $0$ and then  $\eta$  go to $0$, we conclude     that 
 \begin{align*}
 & |B_r(z_0)| \leq \frac{\sS^{n-1}\left(2\sqrt{\kappa}R_0\right)}{ (p-1)^n}\int_{\{  u\leq \tilde M_p\}\cap   B_{4r}(z_0)} \left\{\frac{r^p f^+}{n}+\sH\left( 2\sqrt{\kappa}R_0\right) \right\}^n.
\end{align*} 
 Lastly,   \eqref{eq-abp-type-measure-p-1}   follows from   Bishop-Gromov's volume comparison. 
    \end{proof}

   \begin{cor}\label{cor-abp-type-measure-p-small-nonlinear}
 Assume that  $ 1< p<2,$ and 
 {$\cM^-_{\ld,\Ld }(R(e))\geq - {(n-1)} \kappa$}   with  $\kappa\geq0$ for any unit vector $e\in TM$.   
For  $z_0,x_0\in M$ and $0<r\leq R\leq R_0,$ assume that $ \overline  B_{4r}(z_0)\subset B_{R}(x_0).$     Let  $\beta\geq0,$ and  $u \in C(\overline  B_{R}(x_0))$ be such that  $$|\D u|^{p-2}\cM^-_{\ld,\Ld}(D^2u)-  \beta |\D u|^{p-1}\leq { {}r^{-p}} \quad\mbox{ in $B_{R}(x_0)$}$$
in the viscosity sense,     
\begin{equation}\label{eq-abp-cond-u-p-small-nonlinear}
u\geq0\quad\mbox{on $  B_{R}(x_0)\setminus B_{4r}(z_0) \quad$ and $\quad \displaystyle\inf_{B_{r}(z_0)}u\leq \frac{p-1}{p} .$}
\end{equation}
Then  there exists a uniform constant $\delta\in(0,1)$ depending only on  $n, p, \sqrt{\kappa}R_0, \ld,\Ld,$ and $\beta R_0$, such that 
  \begin{equation*} 
  \left|\left\{ u\leq \tilde M_p\right\}\cap B_{4r}(z_0)\right|>\delta|B_{4r}(z_0)|
  \end{equation*}
  for $\tilde M_p=\frac{p-1}{p}3^{\frac{p}{p-1}}>1.$
\end{cor}

\begin{proof} 
The  proof is similar to  Lemma \ref{lem-abp-type-measure-p-small} by  replacing 
Lemma  \ref{lem-abp-type-p-1}  by Lemma  \ref{lem-abp-type-p-1-nonlinear}.   
Let $\eta>0.$ Making use of  Lemmas  \ref{lem-visc-u-u-e-properties} and \ref{prop-u-u-e-superj},
   the  inf-convolution $u_\ve$ of $u$ with respect to $\overline B_{R}(x_0)$ (for small $\ve>0$) satisfies 
\begin{equation*}
\left\{
\begin{split}
&u_\ve \to u\quad\mbox{uniformly in $B_{R}(x_0)$},\\
&u_\ve\geq -\eta\quad\mbox{in $B_{R}(x_0)\setminus B_{4r}(z_0) $,}\qquad\inf_{B_{r}(z_0)}u_\ve\leq \frac{p-1}{p}+\eta,\\
&\mbox{$u_\ve $ is      $C_\ve$-semiconcave in $B_{R}(x_0),$}\\
& |\D u_\ve|^{ {p-2}}\cM^-_{\ld,\Ld}\left(D^2u_\ve \right)-\tilde \kappa n\Ld |\D u_\ve|^{p-2}  {}{ } \min\left\{\ve|\D u_\ve|^2,  2{ }\omega\left(2\sqrt{m\ve}\right)\right\} -\beta|\D u_\ve|^{p-1}    \leq r^{-p}   \,\,\mbox{ in $B_{4r}(z_0)$} 
 \end{split}\right. 
\end{equation*}
 in the viscosity sense, 
where $-\tilde\kappa $ ($\tilde\kappa\geq0$) is a lower bound of the sectional curvature on $\overline B_{3R}(x_0),$ $\displaystyle C_\ve:=\frac{1}{\ve}{\sH\left(2\sqrt{ \tilde\kappa}R\right)}$, and $m:=\|u\|_{L^{\infty}(B_R(x_0))}.$ 

 For small $\ve>0,$ consider  $\displaystyle \tilde u_\ve:= \frac{(p-1)/p }{(p-1)/p+2\eta} \left(u_\ve+\eta\right)=c_\eta\left(u_\ve+\eta\right)$.  Arguing similarly as in   the proof of Lemma \ref{lem-abp-type-measure-p-small} together with \eqref{eq-abp-cond-u-p-small-nonlinear},    we deduce that   
   \begin{align*}
   \cA^p_{1}\left( \overline B_{r}(z_0); \overline  B_{4r}(z_0);r^{\frac{p}{p-1}}\tilde u_\ve\right)&= \cA^p_{1}\left( \overline B_{r}(z_0); \overline  B_{R}(x_0);r^{\frac{p}{p-1}}\tilde u_\ve\right)\\
  & \subset B_{4r}(z_0)\cap \{ \tilde u_\ve< \tilde M_p,\, r^p|\D \tilde u_\ve|^{p-1}< 2R_0 \},
   \end{align*}
   and 
$$r^{p}|\D \tilde u_\ve|^{ {p-2}}\cM^-_{\ld,\Ld}(D^2 \tilde u_\ve)\leq  c_\eta^{p-1}\left[ r^{p} \tilde\kappa n\Ld\max\left\{\ve,  2{ }\omega\left(2\sqrt{m\ve}\right)\right\} +1
   \right]+ 2\beta R_0$$ in $B_{4r}(z_0)\cap \{  r^p|\D \tilde u_\ve|^{p-1}< 2R_0 \}$ 
 in the viscosity sense.  
Applying     Lemma  \ref{lem-abp-type-p-1-nonlinear} to $r^{\frac{p}{p-1}}\tilde u_\ve$,   we    obtain  
\begin{align*}
|B_r(z_0)|&\leq \sS^{n-1}\left(2\sqrt{ \frac{\kappa}{\lambda}}  R_0{} \right) \int_{\cA_1^{p}\left(\overline B_{r}(z_0);  \overline  B_{4r}(z_0);r^{\frac{p}{p-1}}\tilde u_\ve\right)}
\left[\frac{1}{n\ld}\left\{
   c_\eta^{p-1}\left(   r^{p} \tilde\kappa n\Ld\max\left\{\ve,  2{ }\omega\left(2\sqrt{m\ve}\right)\right\} +1\right)\right.\right.\\
&\qquad\left.\left.+ 2\beta R_0 +\frac{\Ld}{p-1}+  {(n-1)}\Ld  \sH\left(2\sqrt{ \frac{\kappa}{\Ld}} R_0\right) \right\}\right]^ndx\\
&\leq \sS^{n-1}\left(2\sqrt{ \frac{\kappa}{\lambda}}  R_0{} \right) \int_{ B_{4r}(z_0)\cap \{ \tilde u_\ve< \tilde M_p \}}
\left[\frac{1}{n\ld}\left\{
   c_\eta^{p-1}\left(   r^{p} \tilde\kappa n\Ld\max\left\{\ve,  2{ }\omega\left(2\sqrt{m\ve}\right)\right\} +1\right)\right.\right.\\
&\qquad\left.\left.  +2\beta R_0+\frac{\Ld}{p-1}+  {(n-1)}\Ld  \sH\left(2\sqrt{ \frac{\kappa}{\Ld}} R_0\right) \right\}\right]^ndx.
\end{align*}
Letting    $\ve$  and  $\eta$  go to $0$, we conclude     that 
 \begin{align*}
 & |B_r(z_0)| \leq  \sS^{n-1}\left(2\sqrt{ \frac{\kappa}{\lambda}}  R_0{} \right) \int_{\{  u\leq \tilde M_p\}\cap   B_{4r}(z_0)} \left[\frac{1}{n\ld}\left\{ 1+2\beta R_0+\frac{\Ld}{p-1}+ {(n-1)}\Ld  \sH\left(2\sqrt{ \frac{\kappa}{\Ld}} R_0\right)  \right\}\right]^n.
\end{align*} 
 Since  $\Ric(e,e)\geq \cM^-_{\ld,\Ld}(R(e))/\ld \geq-(n-1)\kappa/\ld $ for any unit vector $e\in TM$, this finishes  the proof from Bishop-Gromov's volume comparison.
     \end{proof}

 \section{$L^{\epsilon}$-estimates} 
\label{sec-L-e-est}

In this section, we derive   $L^\epsilon$-estimates for $p$-Laplacian type operators by comparing a viscosity solution with a barrier function, and using the volume doubling property in Lemma \ref{lem-doubling-property}.  We follow  a similar  argument to    \cite[Sections 4 and 5]{IS}, but  it is not straightforward due to  the existence of   the cut-locus in     the setting of Riemannian manifolds. 

\begin{lemma}\label{lem-barrier} 
Let     $1<p<\infty ,$ and $\Ric \geq - {(n-1)} \kappa$    for $\kappa\geq0$. 
For  $z_0,x_0\in M$ and $0<r\leq R\leq R_0$,    assume that $    B_{5r}(z_0)\subset B_{R}(x_0).$ 
  There exists a large constant $\tilde  M>1$  such that if for $\beta\geq0,$      $u\in C\left(  \overline    B_{5r}(z_0)\right)$  is {\color{black}  semiconcave} in  $   B_{5r}(z_0)$,   and satisfies    
\begin{equation*}
\left\{
\begin{split}
&u\geq 0\quad\mbox{in $\overline B_{5r}(z_0)$},\qquad  u> \tilde  M\quad\mbox{ in $ B_{r}(z_0)$,} \\
& \La_pu- \beta |\D u|^{p-1}\leq  r^{-p}\quad\mbox{ in $B_{5r}(z_0)$\,\,  in the viscosity sense},
\end{split}\right.
\end{equation*} 
then     
\begin{equation*} 
u> 1 \quad\mbox{in $B_{4r}(z_0)$,}
\end{equation*}
where   $\tilde M>1$ depends     only on  $n,   p, \sqrt{\kappa} R_0 ,$  and $\beta R_0.$
\end{lemma}

 \begin{proof}
For  $\alpha>1$ and $\tilde M>1$, define  
$$v(x):= \tilde M\left\{ r^{\alpha}d(x,z_0)^{-\alpha}- 5^{-\alpha}\right\}.$$
  By selecting    uniform constants  $\alpha>1$ and $\tilde  M>1$,    depending only on  $n,  p, \sqrt{\kappa} R_0$  and $\beta R_0,$ we have    
\begin{equation}\label{eq-barrier-properties}
\left\{
\begin{split}
&r^{p}\La_p  v-\beta r^p|\D v|^{p-1} \geq  2\quad\mbox{ in $B_{5r}(z_0)\setminus \left(\Cut(z_0)\cup \{z_0\}\right),$}\\
&v= 0\quad\mbox{ in $  \p B_{5r}(z_0)$},  \\
& v>1 \quad\mbox{in $B_{4r}(z_0)$},\\
&  \sup_{\p B_{r}(z_0)}v<\tilde  M\\
&    |\D v(x)|>0   \quad \forall x\not\in\Cut(z_0) \cup \{z_0\},
\end{split}
\right.
\end{equation}
 In fact,  for $x \in B_{5r}(z_0)\setminus\left(\Cut(z_0)\cup\{z_0\}\right),$
\begin{align*}
\beta r^p |\D v(x)|^{p-1}&= \left(\frac{r}{d_{z_0}(x)}\right)^{p} \left\{\alpha \tilde M \left(\frac{r}{d_{z_0}(x)} \right)^{\alpha} \right\}^{p-1}\beta d_{z_0}(x)\leq \left(\frac{r}{d_{z_0}(x)}\right)^{p} \left\{\alpha \tilde M \left(\frac{r}{d_{z_0}(x)} \right)^{\alpha} \right\}^{p-1}\beta R_0, 
\end{align*}
and 
\begin{align*}
r^p\La_pv(x)&=  r^p|\D v|^{p-2}\tr\left\{\left({\bf I}+(p-2)\D d_{z_0}\otimes \D d_{z_0}\right)D^2v(x)\right\}\\
&= \left(\frac{r}{d_{z_0}(x)}\right)^{p} \left\{\alpha \tilde  M \left(\frac{r}{d_{z_0}(x)} \right)^{\alpha} \right\}^{p-1} \left\{(\alpha+1)(p-1)+1-\La d^2_{z_0}(x)/2  \right\},
\end{align*}
where we recall that $\D d_{z_0}(x)$ is an eigenvector of  $D^2 d_{z_0}(x)$ and $D^2(d_{z_0}^2/2)(x)$ associated with  eigenvalues $0$ and $1$, respectively.  
By using  Lemma \ref{lem-pucci-ric-dist-sqrd}, we choose  $\alpha>1$ and $\tilde M>1$ sufficiently  large to  obtain 
\begin{align*}
r^p\La_pv-\beta r^p |\D v|^{p-1}&\geq \left(\frac{r}{d_{z_0}}\right)^{p} \left\{\alpha  \tilde  M \left(\frac{r}{d_{z_0}} \right)^{\alpha} \right\}^{p-1}\left\{ (\alpha+1)(p-1)-(n-1)\sH\left(\sqrt{\kappa}R_0\right)-  \beta R_0\right\}\\
&\geq 5^{-p }\left(\alpha \tilde M 5^{-\alpha} \right)^{p-1}\cdot 1\geq  2 \qquad\mbox{ in $B_{5r}(z_0)\setminus \left(\Cut(z_0)\cup \{z_0\}\right)$}. 
\end{align*}
   It is  easy to check   other properties in \eqref{eq-barrier-properties}  for  large $\alpha>1$ and $\tilde  M>1$. 

Now we prove that  if $u-v$ has a minimum at  an interior point $x\in B_{5r}(z_0)\setminus \overline B_{r}(z_0),$ then $x$ is not a cut  point of $z_0$. 
  Suppose to the contrary that     $x\in (B_{5r}(z_0)\setminus \overline B_{r}(z_0))\cap \Cut(z_0)$  is  an interior minimum point  of $u-v.$ 
    Note  that $x\not= z_0$ if $x\in\Cut(z_0)$.  Define $\psi(s):=- \tilde M\left\{ r^{\alpha}s^{-\alpha/2}- 5^{-\alpha}\right\}.$ We notice that  
  $v=-\psi\circ d_{z_0}^2$  and $\psi'>0$ in $(0,\infty).$   By Corollary \ref{cor-dist-sqrd-cut-compoed-increasing-ft},  there is   a unit vector $X\in T_xM$ such that  
  $$ \liminf_{t\to0}\frac{1}{t^2}\left\{\psi\left(d_{z_0}^2\left(\exp_xtX\right)\right)+\psi\left(d_{z_0}^2\left(\exp_x -tX\right)\right)-2\psi\left(d^2_{z_0}(x)\right)\right\} =-\infty. 
$$ Since $u-v$ has a minimum at  an interior point $x\in (B_{5r}(z_0)\setminus \overline B_{r}(z_0) )\cap\Cut(z_0),$ we have 
   \begin{align*}
-\infty&=   \liminf_{t\to0}\frac{1}{t^2}\left\{\psi\left(d_{z_0}^2\left(\exp_xtX\right)\right)+\psi\left(d_{z_0}^2\left(\exp_x -tX\right)\right)-2\psi\left(d^2_{z_0}(x)\right)\right\} \\
&\geq -\limsup_{t\to0} \frac{1}{t^2}\left\{  u\left(\exp_xtX\right)+u\left(\exp_x-tX\right) -2 u\left(x\right)\right\},
\end{align*} 
which  is a contradiction due to semi-concavity of $u$. 
 So   $x$ is not a cut point of $z_0.$  
 
According to  
 the  comparison principle, we   
conclude that $u-v\geq0$ in $B_{5r}(z_0)\setminus \overline B_r(z_0)$ since $v$ is smooth  in  $ B_{5r}(z_0)\setminus \left(\{z_0\}\cup\Cut(z_0)\right)  $ with  non-vanishing gradient, and $u-v\geq 0$ on $\p B_{5r}(z_0)\cup\p  B_{r}(z_0).$  Thus \eqref{eq-barrier-properties} implies that $u>1$ in $B_{4r}(z_0).$
        \end{proof}

\begin{cor}\label{cor-barrier-viscosity} 
Let   $1<p<\infty ,$ and $\Ric \geq - {(n-1)} \kappa$    for $\kappa\geq0$. 
For  $z_0,x_0\in M$ and $0<r\leq R\leq R_0$,  assume that $    B_{6r}(z_0)\subset B_{R}(x_0).$ 
  There exists a large constant $\tilde  M>1$  such that if    $u\in C\left(  \overline    B_{6r}(z_0)\right)$ satisfies  
\begin{equation*}
\left\{
\begin{split}
&u\geq 0\quad\mbox{in $  B_{6r}(z_0)$},\qquad  u> \tilde  M\quad\mbox{ in $ B_{r}(z_0)$,}\\
& \La_pu \leq  r^{-p}\quad\mbox{ in $B_{6r}(z_0)\,\,$  in the viscosity sense},
\end{split}\right.
\end{equation*} 
then     
\begin{equation*} 
u> 1 \quad\mbox{in $B_{4r}(z_0)$,}
\end{equation*}
where   $\tilde M>1$ depends     only on  $n,   p, $  and $\sqrt{\kappa}R_0.$
\end{cor}
\begin{proof}
We use  again Jensen's inf-convolution to   approximate a viscosity supersolution $u$. 
 Let $\eta>0.$ According to  Lemma  \ref{lem-visc-u-u-e-properties} 
     and Remark \ref{rmk-p-La-eq-u-e},    the inf-convolution $u_\ve$ of $u$ with respect to $\overline B_{6r}(z_0)$ (for small $\ve>0$)   satisfies the following: 
  \begin{equation*}
\left\{
\begin{split}
&u_\ve \to u\quad\mbox{uniformly in $B_{6r}(z_0),$}\\
&u_\ve\geq -\eta\quad\mbox{in $B_{6r}(z_0)  $,}\qquad\inf_{B_{r}(z_0)}u_\ve>  \tilde M -\eta , \\
&\mbox{$u_\ve $ is     $C_\ve$-semiconcave in $B_{6r}(z_0),$}\\
&\La_pu_\ve - \tilde\kappa(n+p-2) \left\{2\ve\omega\left(2\sqrt{m\ve}\right)\right\}^{\frac{1}{2}}|\D u_{\ve}|^{p-1}    \leq r^{-p}  \,\,\mbox{in $B_{5r}(z_0)$ }
 \end{split}\right. 
\end{equation*}
in the viscosity sense, 
where $-\tilde\kappa $ ($\tilde\kappa\geq0$) is a lower bound of the sectional curvature on $\overline B_{3R}(x_0),$ $C_\ve:=\frac{1}{\ve}{\sH\left(2\sqrt{ \tilde\kappa}R\right)}$,    
 $m:=\|u\|_{L^{\infty}(B_{6r}(z_0))},$ and $\omega$  denotes a modulus of continuity of $u$ on $\overline B_{6r}(z_0).$  We observe that 
 for   sufficiently     small $\ve>0$, $u_\ve$ satisfies 
$$\La_p u_{\ve} - \frac{1}{ R_0}|\D u_{\ve}|^{p-1}\leq r^{-p}   \quad\mbox{in $B_{5r}(z_0) $}$$
in the viscosity sense. 

Let $\tilde M_0>1$ be the constant as in Lemma  \ref{lem-barrier} with $\beta= 1/R_0.$ Here, $\tilde M_0>1$ depends     only on  $n,   p, $ and $ \sqrt{\kappa} R_0 .$    Let  $\tilde M:=\tilde M_0+1. $  For  $0<\eta<1/ (2\tilde M_0),$ we apply Lemma \ref{lem-barrier} to $\tilde u_\ve:=\frac{u_\ve+\eta }{1+2\eta}$ in order to deduce that $\tilde u_\ve>1$ in $B_{4r}(z_0)$. By letting $\ve$ go to $0$,  it follows that $ u \geq {1+\eta}>1$ in $B_{4r}(z_0)$.  \end{proof}

Similarly,   we  obtain the following corollary by  constructing  a barrier function   for nonlinear $p$-Laplacian type operators with the help of   Lemma \ref{lem-pucci-ric-dist-sqrd}.

\begin{cor}\label{lem-barrier-nonlinear} 
Let    $1<p<\infty ,$ and $\cM^-_{\ld,\Ld}(R(e)) \geq - {(n-1)} \kappa$   with  $\kappa\geq0$ for any unit vector $e\in TM$.  
For  $z_0,x_0\in M$ and $0<r\leq R\leq R_0,$ assume that $    B_{6r}(z_0)\subset B_{R}(x_0).$  
  There exists a large constant $\tilde  M>1$  such that if  for $\beta\geq0,$   $u\in C\left(   \overline    B_{6r}(z_0)\right)$ satisfies 
\begin{equation*}
\left\{
\begin{split}
&u\geq 0\quad\mbox{in $  B_{6r}(z_0)$},\qquad  u> \tilde  M\quad\mbox{ in $ B_{r}(z_0)$,}\\ 
&|\D u|^{p-2}\cM^-_{\ld,\Ld}(D^2u)-\beta|\D u|^{p-1}\leq  r^{-p}\quad\mbox{ in $B_{6r}(z_0)\,\,$  in the viscosity sense},
\end{split}\right.
\end{equation*} 
then     
\begin{equation*} 
u> 1 \quad\mbox{in $B_{4r}(z_0)$,}
\end{equation*}
where   $\tilde M>1$ depends     only on  $n,   p,  \sqrt{\kappa} R_0,\ld,\Ld,$  and $\beta R_0.$
\end{cor}

Combined with Lemmas \ref{lem-abp-type-measure-p-large} and  \ref{lem-abp-type-measure-p-small}, we have the following measure estimate. 
   \begin{cor}\label{cor-barrier}   
   Let   $1<p<\infty ,$ and $\Ric \geq - {(n-1)} \kappa$   for $\kappa\geq0$. 
For  $z_0,x_0\in M$ and $0<r\leq R\leq R_0,$ assume that $  B_{2r}(z_0)\subset B_{R}(x_0).$ There exist     constants $\tilde  M>1$  and   $0<\delta<1$ such that if   $u\in C(B_{2r}(z_0))$ satisfies 
\begin{equation*}
\left\{
\begin{split}
&u\geq 0\quad\mbox{in $  B_{2r}(z_0)$},\\ 
&r^{p} \La_pu\leq  1\quad\mbox{ in $B_{2r}(z_0)$  in the viscosity sense},\\
& \left|\left\{u> \tilde M\right\} \cap B_{r}(z_0)\right|> (1-\delta)|B_{r}(z_0)|,
\end{split}\right.
\end{equation*} 
then     
\begin{equation*} 
u> 1 \quad\mbox{in $B_{r}(z_0)$,}
\end{equation*}
where   $ \tilde M>1$ and  $\delta\in(0,1)$   depend  only on  $n, p, $ and $\sqrt{\kappa}R_0.$
\end{cor}
\begin{proof}
Let $  \tilde M_0=\frac{p-1}{p}3^{\frac{p}{p-1}}>1$ and  $0<\delta<1$ be the constants 
  as    in Lemmas \ref{lem-abp-type-measure-p-large} and   \ref{lem-abp-type-measure-p-small}, and let $\tilde M_1$  be as in Corollary \ref{cor-barrier-viscosity}. 
Let $\tilde M:=\frac{p}{p-1} \tilde M_0\tilde M_1.$   Applying Lemmas \ref{lem-abp-type-measure-p-large} and   \ref{lem-abp-type-measure-p-small}  to $\displaystyle \frac{p-1}{p}\frac{u}{\tilde M_1}$ in $B_{r}(z_0) $, we  obtain  that   $\displaystyle\inf_{B_{r/4}(z_0)} u > \tilde M_1.$  From Corollary  \ref{cor-barrier-viscosity},  it follows that $u>1$ in $B_r(z_0). $
\end{proof}
   
    The homogeneity of the $p$-Laplacian operator implies the following. 
   \begin{cor}\label{cor-barrier-theta}   
   Let    $1<p<\infty ,$ and $\Ric \geq - {(n-1)} \kappa$   for $\kappa\geq0$. 
For  $z_0,x_0\in M$ and $0<r\leq R\leq R_0,$ assume that $  B_{2r}(z_0)\subset B_{R}(x_0).$   
There exist     constants $\tilde  M>1$  and   $0<\delta<1$ such that if  for $ \theta>0$,    $u\in C(B_{2r}(z_0))$ satisfies 
\begin{equation*}
\left\{
\begin{split}
&u\geq 0\quad\mbox{in $  B_{2r}(z_0)$},\\ 
&r^{p} \La_pu\leq  \theta^{p-1}\quad\mbox{ in $B_{2r}(z_0)\,\,$  in the viscosity sense,  }\\
& \left|\left\{u>\theta \tilde M\right\} \cap B_{r}(z_0)\right|> (1-\delta)|B_{r}(z_0)|,
\end{split}\right.
\end{equation*} 
then     
\begin{equation*} 
u> \theta \quad\mbox{in $B_{r}(z_0)$,}
\end{equation*}
where     $ \tilde M>1$ and $\delta\in(0,1)$ are the constants in Corollary \ref{cor-barrier}.     
\end{cor}

By using Corollaries \ref{cor-abp-type-measure-p-large-nonlinear}, \ref{cor-abp-type-measure-p-small-nonlinear} and \ref{lem-barrier-nonlinear},  the same argument  as the proof of  Corollary \ref{cor-barrier}  yields  the following measure estimate for nonlinear $p$-Laplacian type operators. 
   \begin{cor}\label{cor-barrier-nonlinear}   
   Let   $1<p<\infty $, and $\cM^-_{\ld,\Ld}(R(e)) \geq - {(n-1)} \kappa$ with  $\kappa\geq0$ for any unit vector $e\in TM$. For  $z_0,x_0\in M$ and $0<r\leq R\leq R_0,$ assume that $  B_{2r}(z_0)\subset B_{R}(x_0).$   There exist   constants $\tilde  M>1$  and  $0<\delta<1$ such that if for   $\beta\geq0$ and $\theta>0,$     $u\in C(B_{2r}(z_0))$ satisfies 
\begin{equation*}
\left\{
\begin{split}
&u\geq 0\quad\mbox{in $  B_{2r}(z_0)$},\\ 
&|\D u|^{p-2}\cM^-_{\ld,\Ld}(D^2 u)-\beta |\D u|^{p-1}\leq  \theta^{p-1}r^{-p}\quad\mbox{ in $B_{2r}(z_0)\,\,$  in the viscosity sense},\\
& \left|\left\{u>\theta \tilde M\right\} \cap B_{r}(z_0)\right|> (1-\delta)|B_{r}(z_0)|,
\end{split}\right.
\end{equation*} 
then     
\begin{equation*} 
u> \theta \quad\mbox{in $B_{r}(z_0)$,}
\end{equation*}
where   $ \tilde M>1$ and $\delta\in(0,1)$ depend  only on  $n, p, \sqrt{\kappa} R_0,\ld,\Ld, $ and $\beta R_0.$
\end{cor}
 
 Now we prove $L^{\e}$-estimates for viscosity supersolutions by obtaining power decay   of measure of  super-level sets. 

\begin{thm}
\label{thm-L-epsilon}
  Let    $1<p<\infty ,$ and $\Ric \geq - {(n-1)} \kappa$   for $\kappa\geq0$. 
Let   $x_0\in M$ and $0< R\leq R_0.$   
There exist   constants $\tilde  M>1$  and  $0<\delta<1$ such that if    $u\in C(B_{2R}(x_0))$ satisfies 
\begin{equation*}
\left\{
\begin{split}
&u\geq 0\quad\mbox{in $  B_{2R}(x_0)$},\qquad \inf_{B_R(x_0)}u\leq1,
\\ 
&R^{p} \La_pu\leq  1\quad\mbox{ in $B_{2R}(x_0)\,\,$  in the viscosity sense},
\end{split}\right.
\end{equation*} 
then     
\begin{equation*} 
 \left |\left\{u> \tilde M^k\right\} \cap B_{R}(x_0)\right|< (1-\delta)^k|B_{R}(x_0)|,\quad\forall k=1,2,\cdots. 
 \end{equation*}
 Furthermore,  
 \begin{equation*} 
 \left|\left\{u>t \right\} \cap B_{R}(x_0)\right|<  ct^{-\epsilon}|B_{R}(x_0)|,
 \end{equation*}
where    $\tilde M>1,$ $\delta\in(0,1),$ $c>0,$ and $\epsilon\in(0,1)$ depend  only on  $n,   p, $  and $\sqrt{\kappa}R_0.$
\end{thm}
\begin{proof}
We give a sketch of the proof  
since the proof is  similar to   \cite[Theorem 5.1]{IS}. 
Let $A_k:=\{u>\tilde M^k\}\cap B_R(x_0)$, where  $\tilde M>1$ is the constant in Corollary \ref{cor-barrier}. According to Corollary \ref{cor-barrier}, we have $|A_1|\leq (1-\delta)|B_R(x_0)|.$ We claim that
 $$|A_k|\leq (1-c_0\delta )^k|B_R(x_0)|\quad \forall k=1,2,\cdots,$$ 
 where $0<c_0<1$ is the constant as in Lemma \ref{lem-doubling-property} and depends  on $n,$ and $\sqrt{\kappa} R_0$.  By induction,   suppose that 
$|A_k|\leq (1-c_0\delta)^k|B_R(x_0)|$ for  some $k\in\N.$  Applying  Corollary \ref{cor-barrier-theta} with $\theta=\tilde M^k$, it follows that  if a   ball  $B\subset B_R(x_0)$ satisfies the property that $|A_{k+1}\cap B|>(1-\delta)|B|,$ then    $B\subset A_k$.  
 Therefore, Lemma \ref{lem-doubling-property} yields  that $$|A_{k+1}|\leq (1-c_0\delta)|A_k|\leq (1-c_0\delta)^{k+1}|B_R(x_0)|,$$ which finishes the proof. 
\end{proof}

\begin{cor}[$L^{\epsilon}$-estimate]
   Let   $1<p<\infty, $ and $\Ric \geq - {(n-1)} \kappa$  for $\kappa\geq0$. 
Let  $x_0\in M$ and $0< R\leq R_0.$   For  $C_0\geq0$,  let $u\in C(B_{2R}(x_0))$ be a  nonnegative viscosity supersolution   of   
 $$ \La_pu\leq   C_0 \quad\mbox{  in $B_{2R}(x_0)$.  }$$
Then    
\begin{equation*} 
\left(\fint_{B_{R}(x_0)} u^\epsilon(x)dx\right)^{1/\epsilon}\leq C\left(\inf_{B_R(x_0)}u+R^{\frac{p}{p-1}} C_0^{\frac{1}{p-1}} \right),
  \end{equation*}
where the   constants $\epsilon\in(0,1)$ and  $C>0$ depend only on $n, p$, and $\sqrt{\kappa}R_0.$
 
 \end{cor}
 \begin{proof}
 We may assume   $C_0>0.$ 
  By applying Theorem \ref{thm-L-epsilon} to  $\displaystyle\frac{u}{\inf_{B_R(x_0)}u+R^{\frac{p}{p-1}}C_0^{\frac{1}{p-1}} }$, we deduce  power decay  estimate  for measure of super-level sets, from which   the   $L^{\e}$-esitmate follows.  
 \end{proof}
 
 Proceeding with 
  the same argument to the proof  of  Theorem \ref{thm-L-epsilon} with the use of Corollary \ref{cor-barrier-nonlinear},
  we also  obtain      $L^\e$-estimates for nonlinear $p$-Laplacian type operators;  we note that $\ld \Ric(e,e)\geq \cM^-_{\ld,\Ld}(R(e)) $ for any unit vector $e\in TM$.   
 
\begin{cor}[$L^{\epsilon}$-estimate]\label{cor-L-epsilon-nonlinear}
  Let   $1<p<\infty, $ and $\cM^-_{\ld,\Ld}(R(e)) \geq - {(n-1)} \kappa$  with  $\kappa\geq0$ for any unit vector $e\in TM$.    Let  $x_0\in M$ and $0< R\leq R_0.$   For $\beta\geq0$ and $C_0\geq0,$ let $u\in C(B_{2R}(x_0))$ be a nonnegative viscosity supersolution of 
$$ |\D u|^{p-2}\cM^-_{\ld,\Ld}(D^2u)-\beta|\D u|^{p-1}\leq  C_0\quad\mbox{  in $B_{2R}(z_0).$ }$$
 Then
\begin{equation*} 
\left(\fint_{B_{R}(x_0)} u^\epsilon(x)dx\right)^{1/\epsilon}\leq C\left(\inf_{B_R(x_0)}u+R^{\frac{p}{p-1}} C_0^{\frac{1}{p-1}} \right),
  \end{equation*}
where the constants   $\epsilon\in(0,1)$  and $C>0$ depend  only on  $n,   p, \sqrt{\kappa}R_0,$ $ \ld,\Ld, $ and $\beta R_0$. 
\end{cor}


 \section{Harnack inequality}\label{sec-Harnack}
 This section is devoted to the proof of   Harnack inequality   using the scale-invariant $L^\e$-estimates.    We follow the method of \cite[Theorem 1.3]{IS} for the proof. 
 
 \begin{thm}[Harnack inequality]
 
   Let  $1<p<\infty $,  and $\Ric \geq - {(n-1)} \kappa$   for $\kappa\geq0$. 
For  $z_0\in M$ and $0< R\leq R_0,$   let $u\in C(B_{2R}(z_0))$ be a nonnegative  viscosity solution  of  
\begin{equation*}
  \La_pu=  f  \quad\mbox{ in $B_{2R}(z_0)$. } 
\end{equation*} 
Then   
\begin{equation*} 
\sup_{B_R(z_0)} u\leq C\left(\inf_{B_R(z_0)}u+R^{\frac{p}{p-1}} \|f\|^{\frac{1}{p-1}}_{L^\infty(B_{2R}(z_0))}\right),
  \end{equation*}
where a    constant  $C>0$ depends only on $n, p$, and $\sqrt{\kappa}R_0.$
  \end{thm}
  
\begin{proof}
We may assume that $\displaystyle\inf_{B_R(z_0)}u\leq 1$ and  $\|f\|_{L^\infty(B_{2R}(z_0))}\leq R^{-p}$    replacing $u$    by
 $  \frac{u}{\inf_{B_R(z_0)}u+R^{\frac{p}{p-1}} \|f\|^{\frac{1}{p-1}}_{L^\infty(B_{2R}(z_0))}} $. 
   For   $\alpha>0$ to be  chosen later (depending only on $n,p,$ and $\sqrt{\kappa}R_0$),   consider  a continuous function  
 \begin{equation*}
 h_\tau(x):=
   \tau \left\{4/3-d_{z_0}(x)/R\right\}^{-\alpha},   \quad \forall x\in B_{4R/3}(z_0),
   \end{equation*}   
  where $\tau>0$ is  selected  to be the minimal constant such that 
  $$u\leq h_\tau\quad\mbox{in $B_{4R/3}(z_0).$}$$
      Let $x_0\in B_{4R/3}(z_0)$  be a point such that 
   $u(x_0)=h_\tau(x_0)=:H_0>0$, and let $r:=\frac{4R/3-d_{z_0}(x_0)}{2}\in(0,2R/3].$ Then we have $B_{2r}(x_0)\subset B_{4R/3}(z_0)$,   
 and $H_0=h_\tau(x_0)= \tau(2r/R)^{-\alpha}. $ We may assume that $\tau\geq (4/3)^{\alpha}$ and then $H_0\geq1$; otherwise  $\displaystyle\sup_{B_R(z_0)}u\leq\sup_{B_R(z_0)}h_\tau\leq (4/3)^{\alpha}\cdot 3^{\alpha}=4^{\alpha}.$

  After  using a standard covering argument, we deduce from    Theorem \ref{thm-L-epsilon}  that 
 \begin{equation}\label{eq-proof-HI-1}
   \left|\left\{ u>H_0/2\right\}\cap B_{4R/3}(z_0)\right|\leq cH_0^{-\epsilon}|B_{4R/3}(z_0)|,
 \end{equation} 
 where the   constants $c>0$ and  $\epsilon\in(0,1)$ depend only on $n,p,$ and $\sqrt{\kappa}R_0. $ 
 
 On the other hand,  for a given $\mu\in(0,1),$
 $$u\leq h_\tau\leq \tau (r/R)^{-\alpha}\left( 2-\mu \right)^{-\alpha}=H_0\left(\frac{2-\mu}{2}\right)^{-\alpha} \quad\mbox{ 
 in $B_{\mu r}(x_0).$}$$ 
  Define
 $$\tilde u(x):=\frac{\left(\frac{2-\mu}{2}\right)^{-\alpha}H_0-u\left(x\right)}{\left\{\left(\frac{2-\mu}{2}\right)^{-\alpha}-1\right\}H_0}\qquad\forall x\in B_{\mu r}(x_0),$$
 which  is nonnegative,  and satisfies $\tilde u(x_0)=1$  and 
 \begin{equation*}
\left(\frac{\mu r}{2}\right)^{p}\La_p\tilde u\leq \mu \left[\frac{\mu }{\left\{\left(\frac{2-\mu}{2}\right)^{-\alpha}-1\right\}H_0}\right]^{p-1} \left(\frac{r}{2R}\right)^p\quad\mbox{ in $ B_{\mu r}(x_0)$  }.
\end{equation*} 
Let  
$\tilde M>1$ be the constant in  Corollary \ref{cor-barrier}.
 We  select  a large constant $\alpha>0$ and a small constant  $\mu\in(0,1)$
  such that 
 $$\alpha:=\frac{4\log_2\cD }{\epsilon} ,\quad \ \frac{\mu}{ \left(\frac{2-\mu}{2}\right)^{-\alpha}-1 } \leq\frac{4}{\alpha}\leq 1,\quad \mbox{and }\quad  \left\{\left(\frac{2-\mu}{2}\right)^{-\alpha}-1\right\}\leq \frac{1}{2\tilde M},$$
 for $\cD:=2^n\cosh^{n-1}\left(6\sqrt{\kappa}R_0\right)$
 since $\displaystyle\lim_{\mu\to0+} \frac{\mu}{ \left(\frac{2-\mu}{2}\right)^{-\alpha}-1 }= \frac{2}{\alpha}.$ 
    Then   
  $\tilde u$ satisfies 
 \begin{equation*}
 (\mu r/2)^{p}\La_p\tilde u\leq1\quad\mbox{ in $ B_{\mu r}(x_0)$  },
\end{equation*}
 since  $p>1,$ $H_0\geq1,$     $0<\mu<1,$ and $0<r<R.$ 
 By applying Corollary \ref{cor-barrier} to $\tilde u$ in $B_{\mu r}(x_0)$ with $\tilde u(x_0)=1,$ we have 
$$|\{\tilde u\leq \tilde M\}\cap  B_{\mu r/2}(x_0) |> \delta \left| B_{\mu r/2}(x_0)\right|,$$
which implies  
$$|\{u> H_0/2\}\cap B_{\mu r/2}(x_0)| > \delta \left| B_{\mu r/2}(x_0)\right|,$$
since $$H_0\left[\left(\frac{2-\mu}{2}\right)^{-\alpha}-\tilde M\left\{\left(\frac{2-\mu}{2}\right)^{-\alpha}-1\right\}\right]>\frac{H_0}{2}.$$
 Combined with \eqref{eq-proof-HI-1}, we have 
 \begin{align*}
\delta|B_{\mu r/2}(x_0)|  &<|\{u> H_0/2\}\cap B_{\mu r/2}(x_0)|\\&\leq |\{u> H_0/2\}\cap B_{4R/3}(z_0)|\\
&   \leq cH_0^{-\epsilon}|B_{4R/3}(z_0)| \leq cH_0^{-\epsilon}|B_{8R/3}(x_0)|=c \tau^{-\epsilon}2^{\epsilon\alpha}\left(\frac{r}{R}\right)^{\epsilon\alpha}|B_{8R/3}(x_0)|\\&\leq c \tau^{-\epsilon}2^{\epsilon\alpha}\left(\frac{r}{R}\right)^{\epsilon\alpha} \cD\left(\frac{16R}{3\mu r}\right)^{\log_2\cD}|B_{\mu r/2}(x_0)|
 \end{align*}
 for $\cD:=2^n\cosh^{n-1}\left(6\sqrt{\kappa}R_0\right)$ from \eqref{eq-doubling-cont}.
Therefore, it follows that $\tau$ is uniformly bounded from above since $\epsilon\alpha \geq \log_2\cD,$ and hence $\displaystyle\sup_{B_R(z_0)}u\leq\sup_{B_R(z_0)}h_\tau\leq \tau \cdot 3^{\alpha}.$
\end{proof}
 
 By replacing  Theorem \ref{thm-L-epsilon} and   Corollary \ref{cor-barrier}  by Corollaries \ref{cor-L-epsilon-nonlinear} and     \ref{cor-barrier-nonlinear}, we have the following Harnack inequality for nonlinear $p$-Laplacian type operators.   
  \begin{thm}
   Let   $1<p<\infty $ and $\cM^-_{\ld,\Ld}(R(e)) \geq - {(n-1)} \kappa$  with  $\kappa\geq0$ for any unit vector $e\in TM$.  Let  $z_0\in M$ and $0< R\leq R_0.$ For $\beta\geq0, $ and $ C_0\geq0,$   let $u$ be a nonnegative  viscosity solution  to
\begin{equation*}
\left\{
\begin{split}
&|\D u|^{p-2}\cM^-_{\ld,\Ld}(D^2u)-\beta |\D u|^{p-1}\leq  C_0\quad\mbox{ in $ B_{2R}(z_0)$  },\\
 &|\D u|^{p-2}\cM^+_{\ld,\Ld}(D^2u)+\beta |\D u|^{p-1}\geq -  C_0 \quad\mbox{ in $ B_{2R}(z_0)$.   }
\end{split}\right.
\end{equation*}
Then   
\begin{equation*} 
\sup_{B_R(z_0)} u\leq C\left(\inf_{B_R(z_0)}u+R^{\frac{p}{p-1}}C_0^{\frac{1}{p-1}} \right),
  \end{equation*}
where a    constant  $C>0$ depends only on $n,p, \sqrt{\kappa}R_0 , \ld, \Ld, $ and $ \beta R_0.$
  \end{thm}

\end{document}